%% file: g2m.tex
\documentclass[11pt,a4paper,pdftex]{amsart}
\RequirePackage{booktabs}
\RequirePackage{multirow}
\RequirePackage[noadjust]{cite}
\RequirePackage{amsmath}
\RequirePackage{amssymb}
\RequirePackage{graphics}
\RequirePackage{amsthm}
\RequirePackage{graphicx}
\RequirePackage{color}
\RequirePackage{epsfig}
\RequirePackage{amscd}
\RequirePackage{xspace}
\RequirePackage[utf8]{inputenc}
\RequirePackage[all]{xy}
\usepackage{enumitem}
\input{macros_g2m}

\title[\gtwo-geometry via semi-Fano 3-folds]{\gtwo-manifolds and
  associative submanifolds via semi-Fano $3$-folds}
\author[A.~Corti]{Alessio~Corti}
\author[M.~Haskins]{Mark~Haskins}
\author[J.~Nordström]{Johannes~Nordström}
\author[T.~Pacini]{Tommaso~Pacini}


\setlist{leftmargin=*}

\begin{document}

\begin{abstract}
We provide a significant extension of the twisted connected sum construction of \gtmfd s, \ie 
Riemannian $7$-manifolds with holonomy group $\gtwo$, first developed by Kovalev; 
along the way we address some foundational questions at the heart of the twisted connected sum construction.
Our extension allows us to prove many new results about compact \gtmfd s and leads to new perspectives 
for future research in the area. Some of the main contributions of the paper are:
\begin{enumerate}
\item
We correct, clarify and extend several aspects of the K3 ``matching problem'' that occurs as a key
step in the twisted connected sum construction.
\item
We show that the large class of asymptotically cylindrical Calabi-Yau $3$-folds built from
semi-Fano 3-folds (a subclass of weak Fano 3-folds) can be used as components
in the twisted connected sum construction.
\item
We construct  many new topological types of compact $\gtwo$-manifolds by applying the twisted connected sum 
to asymptotically Calabi-Yau $3$-folds of semi-Fano type
studied in \cite{chnp1}.
\item
We obtain much more precise topological information about twisted connected sum \mbox{\gtwo-manifolds}; one application is the
determination for the first time of the diffeo\-morphism type of many compact $\gtwo$-manifolds.
\item
We describe ``geometric transitions'' between  $\gtwo$-metrics on different $7$-manifolds mimicking 
``flopping'' behaviour among semi-Fano $3$-folds and ``conifold transitions" between Fano and 
semi-Fano $3$-folds.
\item
We construct many $\gtwo$-manifolds that contain rigid compact associative $3$-folds. 
\item
We prove that many smooth $2$-connected $7$-manifolds can be realised  as 
twisted connected sums in numerous ways; by varying the building blocks matched 
we can vary the number of rigid associative $3$-folds constructed therein.
\end{enumerate}
The latter result leads to speculation that the moduli space of $\gtwo$-metrics on a given $7$-manifold may consist of 
many different connected components, and opens up many further questions for future study.
For instance,  the higher-dimensional gauge theory invariants 
proposed by Donaldson may provide ways to 
detect $\gtwo$-metrics on a given $7$-manifold that are not deformation equivalent.

\end{abstract}
\maketitle

\setcounter{secnumdepth}{1}

\section{Introduction}
\input{g2_intro}

\section{Preliminaries: \gtwo and \sunitaryn{} geometry}
\label{sec:prelim}

\input{sec2}

\section{The twisted connected sum construction of \gtwo-manifolds}
\label{sec:twisted_kovalev}

\input{sec4}

\section{Topology of the \gtwo-manifolds}
\label{sec:top}

\input{sec_top}

\section{Construction of associative submanifolds}
\label{sec:geom_assoc}

\input{sec6}

\section{The matching problem}
\label{sec:k3}

\input{sec_k3}

\section{Examples: \gtwo-manifolds}
\label{sec_g2mfds}

\input{sec_g2examples}

\section{Review and Outlook}
\label{sec:close}

\input{sec_close.tex}

\bibliographystyle{amsinitial}
\bibliography{bib_mark,bib_alessio,bib_jn}
\end{document}

%% file: macros_g2m.tex
\DeclareMathOperator{\rk}{rk}

\DeclareMathOperator{\Tor}{Tor}

\DeclareMathOperator{\Amp}{Amp}

\newcommand{\oo}{\mathcal{O}}
\newcommand{\bfa}{\mathbf{a}}
\newcommand{\QQ}{\mathbb{Q}}
\newcommand{\RR}{\mathbb{R}}

\newcommand{\CC}{\mathbb{C}}
\newcommand{\TT}{\mathbb{T}}
\newcommand{\bbS}{\mathbb{S}}
\newcommand{\ZZ}{\mathbb{Z}}

\newcommand{\NN}{\mathbb{N}}
\newcommand{\FF}{\mathbb{F}}

\newcounter{gex}
\newcommand{\rsl}[1]{\refstepcounter{gex} \label{#1}}
\newcommand{\rslsub}[1]{\rsl{#1} \subsubsection{No \arabic{gex}}}

\newcommand{\tsfrac}[2]{\textstyle\frac{#1}{#2}}
\newcommand{\half}{{\tsfrac{1}{2}}}

\newcommand{\quart}{{\tsfrac{1}{4}}}

\newcommand{\bbz}{\mathbb{Z}}
\newcommand{\bbr}{\mathbb{R}}
\newcommand{\bbc}{\mathbb{C}}
\newcommand{\bbo}{\mathbb{O}}
\newcommand{\bbp}{\mathbb{P}}
\newcommand{\bbrp}{\mathbb{R}^{+}}

\newcommand{\ie}{\emph{ie} }
\newcommand{\eg}{\emph{eg} }
\newcommand{\cf}{\emph{cf} }
\newcommand{\into}{\hookrightarrow}

\newcommand{\gtstr}{\gtwo--structure}

\newcommand{\gtmfd}{\gtwo--manifold}

\newcommand{\gtmetric}{\gtwo-metric}

\newcommand{\sob}[1]{L^p_{#1}}

\newcommand{\harm}{\mathcal{H}}
\newcommand{\lnorm}[1]{\Vert #1 \Vert_{L^{2}}}

\newcommand{\caln}{\mathcal{N}}

\newcommand{\calg}{\mathcal{G}}
\newcommand{\cala}{\mathcal{A}}

\newcommand{\sff}{\mathcal{Y}}
\newcommand{\fbb}{\mathcal{Z}}
\DeclareMathAlphabet{\df}{U}{eus}{m}{n}
\DeclareMathAlphabet{\matheur}{U}{eur}{m}{n}
\newcommand{\hkr}{\matheur{r}}
\newcommand{\hdg}{h}
\newcommand{\ncap}{R}
\newcommand{\nv}{v}
\newcommand{\nw}{w}
\newcommand{\cg}[1]{\ZZ/#1\ZZ}

\newcommand{\asm}{A}

\newcommand{\anglex}{\theta}
\newcommand{\anglen}{\vartheta}
\newcommand{\trivr}{\underline{\bbr}}
\newcommand{\trivc}{\underline{\bbc}}
\newcommand{\ogr}{\widetilde{Gr}}

\newcommand{\cyl}{\infty}
\DeclareMathOperator{\coker}{coker}
\DeclareMathOperator{\Sq}{Sq}
\renewcommand{\mod}{\textup{ mod }}
\makeatletter
\def\l@part{\@tocline{-1}{6pt plus2pt}{0pt}{}{\bfseries}}
\newcommand{\setindent}[1]{
  \addtocontents{toc}{
  \protect\def\protect\l@section{\protect\@tocline{1}{0pt}{#1}{}{}}}}
\newcommand{\setsubindent}[1]{
  \addtocontents{toc}{
  \protect\def\protect\l@subsection{\protect\@tocline{1}{0pt}{#1}{}{}}}}
\makeatother

\newcommand{\ccell}[1]{\multicolumn{1}{c}{#1}}

\newcommand{\gen}[1]{\langle#1\rangle}
\newcommand{\inner}[1]{{<}#1{>}}
\newcommand{\linner}[1]{{<}#1{>}_{L^{2}}}

\newcommand{\contra}[1]{\textstyle\frac{\partial}{\partial #1}}

\newcommand{\tv}{\tilde \varphi}

\newcommand{\posp}{p_{1/2}}

\usepackage{slashed}

\newcounter{mtheorem}

\numberwithin{equation}{section}
\newtheorem{theorem}[equation]{Theorem}
\newtheorem{lemma}[equation]{Lemma}
\newtheorem{prop}[equation]{Proposition}
\newtheorem{corollary}[equation]{Corollary}

\theoremstyle{definition}
\newtheorem{definition}[equation]{Definition}
\newtheorem{example}[equation]{Example}

\theoremstyle{remark}
\newtheorem{remark}[equation]{Remark}
\newtheorem*{remark*}{Remark}

\newcommand{\abs}[1]{\lvert#1\rvert}
\newcommand{\norm}[1]{\left\lvert#1\right\rvert}

\newcommand{\glr}[1]{\textup{GL$(#1,\R)$}}
\newcommand{\glrn}{\textup{GL$(n,\R)$}}

\newcommand{\SL}[1]{\textup{SL$(#1,\C)$}}

\newcommand{\unitaryn}{\textup{U$(n)$}}
\newcommand{\sunitary}[1]{\textup{SU$(#1)$}}
\newcommand{\sunitaryn}{\textup{SU$(n)$}}
\newcommand{\sorth}[1]{\textup{SO$(#1)$}}
\newcommand{\Sp}[1]{\textup{Sp$(#1)$}}
\newcommand{\orth}[1]{\textup{O$(#1)$}}
\newcommand{\spin}[1]{\textup{Spin$(#1)$}}
\newcommand{\sspin}{\textup{Spin}}
\newcommand{\bspin}{\textup{BSpin}}
\newcommand{\espin}{\textup{ESpin}}
\newcommand{\bgtwo}{\textup{BG}_2}
\newcommand{\sso}{\textup{SO}}
\newcommand{\bso}{\textup{BSO}}

\newcommand{\bsu}{\textup{BSU}}

\newcommand{\vol}{\operatorname{vol}}
\newcommand{\Proj}{\operatorname{Proj}}
\newcommand{\Pic}{\operatorname{Pic}}
\newcommand{\D}{\slashed{D}}
\newcommand{\debar}{\overline{\partial}}
\newcommand{\Hol}{\operatorname{Hol}}

\newcommand{\gdiv}{\operatorname{div}}

\newcommand{\Q}{\mathbb{Q}}
\newcommand{\quat}{\mathbb{H}}
\newcommand{\C}{\mathbb{C}}
\newcommand{\R}{\mathbb{R}}
\newcommand{\Z}{\mathbb{Z}}

\newcommand{\PP}{\mathbb{P}}
\newcommand{\CP}{\mathbb{P}}
\newcommand{\Sph}{\mathbb{S}}
\newcommand{\T}{\mathbb{T}}

\newcommand{\ra}{\rightarrow}
\newcommand{\Hom}{\operatorname{Hom}}

\newcommand{\imag}{\operatorname{Im}}
\newcommand{\real}{\operatorname{Re}}
\newcommand{\diag}{\operatorname{diag}}
\newcommand{\Imag}{\operatorname{Im}}
\newcommand{\Real}{\operatorname{Re}}

\newcommand{\Id}{\operatorname{Id}}
\newcommand{\hk}{hyper-Kähler\xspace}

\newcommand{\gtwo}{\ensuremath{\textup{G}_2}\xspace}

%% file: g2_intro.tex
In this paper
we construct a large number of new compact
$\gtwo$-manifolds, that is Riemannian 7-manifolds $(M,g)$ whose holonomy
group  is the compact exceptional Lie group $\gtwo$,
using the so-called \emph{twisted connected sum} construction; 
since any $\gtwo$-manifold is Ricci-flat this yields many Ricci-flat metrics on compact $7$-manifolds.
As an alternative to Joyce's original pioneering construction of 
compact \gtwo-manifolds via ``orbifold resolutions" \cite{joyce:g2,joyce:holonomybook},
Kovalev (based on a suggestion of Donaldson) developed the twisted connected sum construction 
\cite{kovalev:connectsums} as a way to obtain a compact $\gtwo$-manifold by combining a pair of 
(exponentially) asymptotically cylindrical (ACyl) Calabi-Yau $3$-folds.
Loosely speaking, this 
method seeks to construct $\gtwo$-manifolds that contain a sufficiently long almost cylindrical neck-like region; 
in this sense it resembles familiar ``stretching the neck'' constructions in a number of other geometric PDE problems.

Kovalev constructed about one hundred ACyl Calabi-Yau 3-folds for use in
twisted connected sums by starting from Fano 3-folds.
In \cite{chnp1} we exhibited at least several hundred thousand families of ACyl
Calabi-Yau $3$-folds built using weak Fano $3$-folds. We identify a subclass of
``semi-Fano'' 3-folds that have a well-behaved deformation theory, which we use
to prove that ACyl Calabi-Yau 3-folds obtained from semi-Fanos can be
``matched'' to form twisted connected sum \gtmfd s. In doing so we address some
errors in \cite{kovalev:connectsums} detailed below. Starting from semi-Fanos
instead of Fanos lets us construct many more examples of compact \gtmfd s, and
in many cases also makes it possible to exhibit compact associative submanifolds.

\subsubsection{Associative submanifolds}
$\gtwo$-manifolds have two natural classes of interesting calibrated submanifolds: 
$3$-dimensional \emph{associative} submanifolds and $4$-dimensional \emph{coassociative} submanifolds.
Constructing associative $3$-folds in compact $\gtwo$-manifolds is
difficult and relatively few examples are known; part of the difficulty is that---unlike 
that of its calibrated cousins: special Lagrangians or coassociatives---the  
deformation theory of compact associative $3$-folds  can be obstructed. 
In many of the $\gtwo$-manifolds we construct we can exhibit a finite number of
rigid---and therefore unobstructed---associative $3$-folds diffeomorphic to
$\Sph^1 \times \Sph^2$; the use of ACyl Calabi-Yau $3$-folds constructed from 
semi-Fano (rather than Fano) $3$-folds  is crucial here.
The associative $3$-folds we construct are the first compact associative
$3$-folds in compact $\gtwo$-manifolds that are proven to be rigid. 
Thus our work yields a large number of compact $\gtwo$-manifolds in which we can
attempt to ``count'' associative $3$-folds---in the spirit of the 
higher-dimensional invariants envisioned in the papers of Donaldson-Thomas 
\cite{donaldson:thomas} and Donaldson-Segal \cite{donaldson:segal}.
In a related aspect of this programme,
Walpuski \cite{walpuski12} has recently shown how one can construct
\gtwo-instantons that ``bubble'' at suitable rigid associative submanifolds, 
reversing the fundamental bubbling analysis of Tian \cite{tian:calibrated}.

\subsubsection{The topology of twisted connected sums}
We obtain detailed topological information about \gtwo-manifolds of twisted connected sum type
under certain mild extra assumptions on the building blocks used (which we show are satisfied in many cases).
Joyce and Kovalev  had been
able to compute the fundamental group and the Betti numbers of the
$7$-manifolds appearing in their constructions, 
but not further topological information; in particular not even the integral cohomology was known for their examples.
We compute the full integral cohomology of many of our 
\gtwo-manifolds including determining the torsion in $H^3$ and $H^4$
and also the characteristic class $p_1$,
using topological information about the %
ACyl Calabi-Yau $3$-folds from \cite{chnp1}. Calculating characteristic classes
of a manifold constructed by gluing can be quite difficult, but the twisting in
the twisted connected sum construction is sufficiently mild that we can get a
handle on $p_1$.
By distinguishing between examples with the same Betti numbers but different
torsion or different $p_1$ we are able to prove the existence of many new
compact \gtwo-manifolds. 

The twisted connected sum construction often yields \emph{$2$-connected} $7$-manifolds, 
whereas Joyce's ``orbifold resolution'' constructions typically yield $7$-manifolds with 
relatively large second Betti number (only a single example in Joyce's book
\cite{joyce:holonomybook} has $b^{2}=0$, see Remark \ref{R:joyce:g2}). 
Using the classification theory for $2$-connected $7$-manifolds developed by
Wall and Wilkens \cite{wilkens72}, the topological data we have computed
determines completely the diffeomorphism (or almost-diffeomorphism) type of
many of the smooth $7$-manifolds on which we
construct our $\gtwo$-holonomy metrics; these are the first compact
$\gtwo$-manifolds for which the diffeomorphism type is known. 
These smooth $7$-manifolds have simple topological realisations as connected
sums of a nontrivial $\Sph^3$-bundle over $\Sph^4$
with an appropriate number of copies of $\Sph^3 \times \Sph^4$.
This is one of very few instances of geometrically interesting $2$-connected $7$-manifolds 
for which the computations needed to determine the (almost) diffeomorphism classification 
have actually been carried out. 

Even in the simplest case where we use a pair of ACyl Calabi-Yau $3$-folds constructed from 
rank one Fano $3$-folds---as considered in Kovalev's original twisted connected sum construction---our 
refined topological results allow us to prove  a variety of new results. For instance in 
Section \ref{sec_g2mfds} we will show the following
\begin{itemize}
\item
The simplest matching between such ACyl Calabi-Yau $3$-folds 
leads to $2$-connected \gtwo-manifolds with torsion-free cohomology, 
in which case $b^3$ is the only Betti number we need to consider; 
$46$ different possible values of $b^{3}$ are realised this way.
\item
By distinguishing between examples with the same Betti numbers but different $p_1$ we show that 
at least $82$ different smooth $7$-manifolds are realised this way.
\item
Using the classification theory for $2$-connected $7$-manifolds we determine the homeomorphism 
type of all these $7$-manifolds and in the majority of cases  the diffeomorphism type.
\item
We exhibit a pair of \gtwo-manifolds whose underlying $7$-manifolds we prove to
be homeomorphic but not necessarily diffeomorphic; this raises the question
whether there exist different smooth structures on the same topological
$7$-manifold which each admit \gtwo-holonomy metrics. 
To determine the diffeomorphism types of this pair of $7$-manifolds requires
the calculation of an extension of the classical Eells--Kuiper invariant,
as will be discussed in \cite{7class}. We suspect that these 7-manifolds are
in fact diffeomorphic.
\item
Other possible ways to match such ACyl Calabi-Yau $3$-folds exist and lead to simply-connected $7$-manifolds 
with $H^2=0$ but with non-trivial torsion in $H^3$; at least $41$ other topological types 
of \gtwo-manifold arise in this way.
\end{itemize}
The last point illustrates a general phenomenon: it is often possible to arrange matching between the same
pair of building blocks in various ways and thereby to obtain topologically distinct $7$-manifolds from the same pair 
of blocks.

The same techniques applied to more general ACyl Calabi-Yau $3$-folds constructed from other 
Fano $3$-folds or more generally semi-Fano $3$-folds (or the Kovalev-Lee examples arising from 
non-symplectic involutions on K3s \cite{kovalev:lee}) allow us to construct a variety of new topo\-logical types of compact 
\gtwo-manifolds.  Many of the examples built using semi-Fano $3$-folds also contain compact rigid associative $3$-folds.

\subsubsection{Different \gtwo-metrics on the same manifold?}
In the opposite direction using the diffeomorphism classification of $2$-connected smooth $7$-manifolds we show
the following: there exist many $2$-connected smooth $7$-manifolds 
that arise as the underlying smooth manifold for $\gtwo$-manifolds of twisted connected sum-type 
but in several (sometimes many) different ways. 
In other words, 
the twisted connected sum of different pairs of ACyl Calabi-Yau $3$-folds can often yield the same underlying smooth $7$-manifold.
This phenomenon already occurs when matching pairs of ACyl Calabi-Yau $3$-folds constructed from 
rank one Fano $3$-folds as above. More generally, using ACyl Calabi-Yau $3$-folds constructed 
from semi-Fano $3$-folds we can construct diffeomorphic \gtwo-manifolds containing different numbers 
of obvious rigid associative $3$-folds.
These facts naturally suggest the following:

\smallskip
\noindent
\textbf{Question.} When do these $\gtwo$-metrics on the same $7$-manifold belong 
to different connected components of the moduli space of $\gtwo$-metrics?

\smallskip

While answering this question currently goes %
beyond the techniques we have available to us, it provides a further strong
impetus for developing the higher-dimensional enumerative invariants proposed
by Donaldson.
A more elementary tool for distinguishing between components of the moduli
space is to study the homotopy classes of \gtstr s; however, it appears that in
our examples of diffeomorphic \gtmfd s one can always choose the diffeomorphism
so that their \gtstr s are homotopic (\ie connected by a continuous
path of \gtstr s without any constraint on the torsion) \cite{cn}.

\subsubsection{Twisted connected sums and \hk rotations}
To describe some other features of the paper we need to recall the basic ingredients involved in the twisted 
connected sum construction.

In order to construct a metric $g$ with Riemannian holonomy the full group $\gtwo$ and not a smaller subgroup the underlying 
compact $7$-manifold $M$ must have finite fundamental group. Given a pair of ACyl Calabi-Yau $3$-folds $V_{+}$ and $V_{-}$ 
we need a way to glue the two non-compact $7$-manifolds $M_{+}=\Sph^{1}\times V_{+}$ and $M_{-}=\Sph^{1}\times V_{-}$
to get such a compact $7$-manifold. 
By construction the ends of our ACyl Calabi-Yau $3$-folds have the form $\R^{+} \times \Sph^{1} \times S_{\pm}$ where 
$S_{\pm}$  are smooth K3 surfaces. The obvious connected sum construction would yield a manifold with infinite 
fundamental group which could not admit a metric of full holonomy $\gtwo$. Instead we choose to identify the 
cross-section of our ends $\T^{2} \times S_{\pm}$ using a diffeomorphism which exchanges the two circle factors of ~$\T^{2}$. 
However, in order to get a well-defined $\gtwo$-structure on $M$ we also need to identify the two K3 surfaces $S_{\pm}$
using a special diffeomorphism $\hkr:S_+ \ra S_-$. The K3 surfaces $S_\pm$ inherit \hk structures from the
geometry at infinity of~$V_{\pm}$, which can be defined in terms of a
Ricci-flat metric and a triple of parallel complex structures
$I_\pm, J_\pm, K_\pm$. We need to construct a diffeomorphism $\hkr$ which is an
isometry and satisfies
\begin{equation*}
\hkr^{*}I_{-}= J_{+},\quad \hkr^{*}J_{-}= I_{+},
\quad \text{and hence} \quad \hkr^{*}K_{-} = -K_{+}.
\end{equation*}
We call such a map a \emph{\hk rotation}. Even given a plentiful supply of ACyl Calabi-Yau 3-folds
it is highly non-trivial to construct \emph{pairs} of $V_{\pm}$ for which there
is such a \hk rotation $\hkr$. 
However, as soon as we can construct a \hk rotation $\hkr$ for a pair of ACyl Calabi-Yau
$3$-folds $V_\pm$ we can use $\hkr$ to form a twisted connected sum $7$-manifold $M_\hkr$, and build on it
a closed $\gtwo$-structure which has small torsion. 
The perturbation theory for closed \mbox{$\gtwo$-structures} developed by Joyce 
(or the specific ACyl version proved by Kovalev) therefore shows that we can then 
choose an appropriate small perturbation to yield a metric with holonomy \gtwo 
on~$M_\hkr$. 

Thus two main steps are needed to implement the twisted connected sum construction:
\begin{enumerate}
\item
Construct (families of) exponentially ACyl Calabi-Yau structures on suitable smooth quasiprojective $3$-folds.
\item
Understand how to find ACyl Calabi-Yau $3$-folds for which there exists a \hk
rotation, within a given pair of families.
\end{enumerate}
We explain below in more detail how together with  \cite{chnp1} and \cite{hhn} this paper addresses 
problems with both steps (i) and (ii) in Kovalev's original paper \cite{kovalev:connectsums}---and
therefore puts the twisted connected sum construction on a firm foundation---and also extends 
substantially the settings in which solutions to (i) and (ii) can be constructed.

\subsubsection{Exponentially ACyl Calabi-Yau $3$-folds}
There are two ingredients, one analytic and one complex algebraic, for producing exponentially 
ACyl Calabi-Yau 3-folds. The analytic ingredient is to solve the complex Monge-Amp\`ere 
equation on suitable smooth quasiprojective varieties and to obtain sufficiently strong 
estimates for these solutions. 
The proof of the exponential asymptotics of solutions to the complex
Monge-Amp\`ere equation in \cite{kovalev:connectsums} is not valid, but
a complete, short self-contained proof of the existence of
exponentially asymptotically cylindrical Calabi-Yau metrics was given recently
in \cite{hhn}. With a suitable analytic existence theory in place 
the remaining complex algebraic task is to find a (large) supply of suitable quasiprojective varieties. 

\subsubsection{ACyl Calabi-Yau $3$-folds from Fano and weak Fano $3$-folds}

In \cite{kovalev:connectsums} Kovalev showed %
how to construct some suitable quasiprojective $3$-folds by blowing up special curves
(the smooth base locus of a sufficiently generic anticanonical pencil) in
smooth Fano $3$-folds and removing a smooth anticanonical divisor.
Recall a smooth \emph{Fano} $3$-fold $F$ is a smooth projective variety for which $-K_{F}$ is ample or positive.
There are exactly 105 deformation families of smooth Fano $3$-folds:
complex projective space $\CP^{3}$, smooth quadrics, cubics and quartics in $\CP^{4}$ 
being the simplest examples.  
We call the ACyl Calabi-Yau $3$-folds obtained this way ACyl Calabi-Yau 3-folds of  \emph{Fano type}.

A smooth \emph{weak Fano} $3$-fold is a smooth projective variety for which the
anticanonical bundle is big and nef (but not ample). 
Differential geometers are encouraged to think of a big and nef line bundle as the algebraic-geometric 
formulation of the line bundle admitting a hermitian metric whose curvature is  sufficiently semi-positive.
There are at least \emph{hundreds of thousands} of deformation families of smooth weak Fano $3$-folds 
and the topology of smooth weak Fano $3$-folds is less restrictive than for Fano $3$-folds; 
unlike the Fano case there is at present no classification theory for weak Fano $3$-folds, except 
under very special geometric assumptions, but many examples are known. 
In our paper \cite{chnp1} we proved that one can construct suitable quasiprojective 3-folds 
from any weak Fano $3$-fold (satisfying one further very mild assumption); 
combining this weak Fano construction with the analytic existence results from \cite{hhn} we
thereby increased the number of known ACyl Calabi-Yau $3$-folds from a few hundred to several hundred thousand.
We call these ACyl Calabi-Yau $3$-folds of \emph{weak Fano type}.

\subsubsection{Constructing \hk rotations}
With a plentiful supply of exponentially asymptotically cylindrical Calabi-Yau 3-folds now at hand,  
to complete the twisted connected sum construction it remains to solve (ii): 
find pairs of (families) of ACyl Calabi-Yau 3-folds for which there exists a \hk rotation.
In \cite{kovalev:connectsums} Kovalev developed an approach to proving the
existence of \hk rotations between pairs of ACyl Calabi-Yau $3$-folds of Fano
type. Unfortunately in almost all cases his argument relies on Lemma 6.47 in 
\cite{kovalev:connectsums} which is false. 
In the course of the current paper we revisit carefully the construction of \hk
rotations, to correct, clarify and extend the methods available.
In particular our results allow us to construct \hk rotations for many pairs of (families) of ACyl Calabi-Yau 3-folds
of so-called semi-Fano type---a large subclass of weak Fano 3-folds described below.
(After conversations with two of the present authors the final version of
Kovalev-Lee \cite{kovalev:lee} contains a similar construction of \hk rotations 
for ACyl Calabi-Yau 3-folds of Fano or ``nonsymplectic'' type). 

We should point out immediately that given a pair of deformation families of
ACyl Calabi-Yau $3$-folds $V_\pm$  there are typically
various ways to construct \hk rotations between the asymptotic ends $\R^+ \times \Sph^1 \times S_{\pm}$. 
The topology of the resulting $7$-manifold $M$ depends on the choice of the \hk rotation 
$\hkr \colon S_+ \to S_-$. 
At present we do not understand in a systematic way all possible different ways to match a given 
pair of (deformation families of) 
ACyl Calabi-Yau $3$-folds $V_{\pm}$; however, we will exhibit various examples where several different 
matchings are possible and lead to topologically distinct \gtwo-manifolds. 
In some simple cases we do understand essentially all different possible ways to match a given pair.

\subsubsection{Deformation theory for weak and semi-Fano 3-folds}
In Kovalev's original approach to the construction of \hk rotations between pairs of ACyl Calabi-Yau 3-folds 
of Fano type the deformation theory for pairs $(F,S)$ where $F \in \mathcal{F}$ a deformation
type of smooth Fano $3$-folds and $S \in \abs{-K_{F}}$ is a smooth
anticanonical divisor plays a crucial role. 
Our approach to the construction of \hk rotations also rests on a sufficiently good deformation 
theory: For pairs of ACyl Calabi-Yau $3$-folds 
of weak Fano (or Fano) type the construction of \hk rotations is possible  \emph{provided}  we have a 
sufficiently good deformation theory for pairs $(Y,S)$, where $Y \in \sff$, a given deformation type of 
\emph{weak Fano} $3$-folds, and $S \in \abs{-K_{Y}}$ is a smooth section.
In \cite{chnp1} we showed that the deformation theory of such pairs is well-behaved for the subclass 
of \emph{semi-Fano} $3$-folds.  A semi-Fano\footnote{There seems to be no established 
terminology for this particular subclass of weak Fano 3-folds, so the term \emph{semi-Fano} is our invention: 
it is intended to suggest that a semi-Fano 3-fold has \emph{semi-}small anticanonical morphism. 
Warning: semi-Fano has also been used to mean something even weaker than weak Fano, 
\ie a complex manifold for which $-K_Y$ is nef (but not necessarily big), but this terminology 
is not well-established.}
 $3$-fold is a weak Fano $3$-fold on which we
impose an extra assumption on the geometry of its anticanonical morphism, 
namely that it contracts no divisor to a point.
This assumption guarantees that certain cohomology vanishing theorems that are true for 
Fano $3$-folds (but which are false for general weak Fano $3$-folds) still hold for semi-Fano $3$-folds.
There are still hundreds of thousands of deformation families of semi-Fano $3$-folds. 

We prove that we can construct \hk rotations for pairs of ACyl Calabi-Yau
$3$-folds of semi-Fano type under the same sorts of conditions as for pairs
arising from (deformation types of) Fano $3$-folds. 
As we already explained, together with the constructions of \cite{chnp1} this 
immediately gives rise to many new twisted connected sum compact \gtwo-manifolds.
Using just ACyl Calabi-Yau 3-folds from Fano or semi-Fano $3$-folds of rank at most two
or from toric semi-Fano 3-folds we get at least 50 million matching pairs
(and probably many more) that produce 2-connected \gtmfd s.
On the other hand, the smooth classification results show that the
number of different topological types realised is much smaller,
so as mentioned earlier some smooth 7-manifolds must arise as 
twisted connected sums in many different ways.

\subsubsection{Rigid curves and rigid associative $3$-folds}
There is one further immediate advantage of generalising from ACyl Calabi-Yau 3-folds 
of Fano type to those of semi-Fano type.
If $Y$ is a Fano $3$-fold then $-K_{Y}\cdot C> 0$ for any algebraic curve $C$ 
and hence $C$ meets any anticanonical divisor $S \in \abs{-K_Y}$. However, 
semi-Fano $3$-folds can contain special complex curves $C$ for which $K_{Y}\cdot C =0$;
the weakening of $-K_{Y}$ being positive to sufficiently semi-positive is crucial here.
Moreover, in many cases $C$ is a smooth rational curve with normal bundle 
$\mathcal{O}(-1)\oplus \mathcal{O}(-1)$: in this case $C$ is infinitesimally rigid, 
\ie  has no infinitesimal (holomorphic) deformations. 

Precisely because these special $K_{Y}$-trivial curves are rigid and do not intersect 
anticanonical divisors we can use them to construct \emph{compact rigid} holomorphic 
curves in our ACyl Calabi-Yau $3$-folds. 
These rigid curves allow us to produce rigid associative $3$-folds
diffeomorphic to $\Sph^{1}\times \Sph^{2}$ in our twisted connected sum
$\gtwo$-manifolds. The key point is that we have related the deformation
problem for a holomorphic curve $C$ to that of an associative of product form
$\Sph^{1} \times C$. Algebraic geometry provides many tools to understand the
deformation theory of $C$; for a general associative $3$-fold we have no such
tools to understand its deformation theory.

\subsubsection{\gtwo-transitions}
In the geometry of Calabi-Yau $3$-folds, especially in some of their applications to String Theory, an important role 
is played by so-called \emph{geometric transitions}. The simplest and most important such transitions are 
\emph{flops} and \emph{conifold transitions}. These two types of transitions also appear in the context of weak 
Fano $3$-folds; many smooth weak Fano $3$-folds can be flopped to yield other smooth weak Fano $3$-folds 
(which typically are not deformation equivalent to the original weak Fano $3$-fold). 
However, unlike the Calabi-Yau setting where the condition $c_{1}=0$ is preserved, 
a conifold transition that begins with a Fano $3$-fold $F$ will yield only a weak Fano $3$-fold $Y$.
We can construct ACyl Calabi-Yau metrics on blocks constructed from both the Fano $F$ and the weak Fano $Y$, 
and then try to match both types of block to some other given (deformation family of) ACyl Calabi-Yau structure.
This gives rise to the idea of related \gtwo-manifolds or \emph{\gtwo-transitions.}
For the moment we present \gtwo-transitions as a convenient organisational principle that explains certain features 
of the geography of twisted connected sum \gtwo-manifolds. However, there is the future prospect of realising 
these \gtwo-transitions at the level of metric geometry; we explain some of the technical difficulties that would need 
to be overcome to achieve this.

\subsubsection{Connections to M-theory} 
\gtmfd s play a similar role in M-theory as Calabi-Yau 3-folds do in String
Theory. Two questions of significance for M-theory concern the existence of
coassociative K3-fibrations and singular \gtmfd s.

Any twisted connected sum \gtwo-manifold is K3-fibred---essentially because the building blocks $Z$ from which we construct our 
ACyl Calabi-Yau $3$-folds are K3-fibred. Generically the only singular fibres of a building block $Z$, and therefore 
of our \gtwo-manifolds, are $A_{1}$ singularities. 
Because of subtleties due to the singular fibres it is still unknown if these topological ``almost'' coassociative K3-fibrations can be 
made into coassociative K3-fibrations as expected in \cite{gukov:yau:zaslow,acharya:mirror}.

To obtain realistic particle physics (\ie non-abelian gauge groups and chiral fermions) from M-theory on \gtwo-manifolds 
it appears necessary to consider singular \gtwo-spaces with very particular kinds of singularity as explained in 
\cite{acharya:witten,acharya:mirror,acharya:gukov,atiyah:witten}. 
For some recent physical predictions from M-theory on \gtwo-spaces see \cite{acharya:kane:2010,acharya:kane2012,acharya:string}; 
see also \cite{baulieu, deboer,roiban,harvey:moore,pantev,acharya:dw} for some other aspects of M-theory on \gtwo-spaces.
In the present paper we consider only smooth compact 
\gtwo-manifolds (apart from the discussion in the \gtwo-transitions section where we discuss potential ways 
to realise singular \gtwo-manifolds as degenerate limits of our constructions).
There are potential extensions of the present constructions that might in some circumstances
allow the construction of blocks fibred by generically singular K3 fibres. 
However it is not yet clear that these could give rise to \gtwo-spaces with the sort of singularity structure apparently required.

It would be interesting to know: does the presence of torsion in $H^{3}$ or $H^{4}$ of a compact \gtwo-manifold have any 
significance in M-theory?  What if any significance do the \gtwo-transitions discussed in Section \ref{sec:close} 
have in M-theory? Does the existence of many potentially different \gtwo-metrics on the same 
smooth $7$-manifold have any M-theory interpretation?

\subsubsection{Structure of paper}We now describe the structure of the rest of the paper. 

Section \ref{sec:prelim} reviews basic facts about the group \gtwo, \gtwo-structures, \gtwo-holonomy metrics, 
the two natural calibrations on $\gtwo$-manifolds and the relations to Calabi-Yau $3$-folds, \hk K3 surfaces and 
the groups $\sunitary{2}$ and $\sunitary{3}$. 
We include this standard material for two reasons: 
to make the paper more accessible to readers with backgrounds in algebraic geometry or topology 
and also to establish the notation and the conventions we adopt in our paper. 
The reader familiar with the basics of $G_2$-holonomy metrics can safely skip most of this section.

Section \ref{sec:twisted_kovalev} introduces ACyl Calabi-Yau $3$-folds and
recalls how to construct them from certain algebro-geometric data that we call
a \emph{building block}. We recall Kovalev's \emph{twisted connected sum
construction}: it takes a pair of ACyl Calabi-Yau $3$-folds 
together with the specification of a so-called \hk rotation and produces a compact $7$-manifold $M$ with 
finite fundamental group and a $1$-parameter family of closed \gtwo-structures on $M$ with small torsion.
We explain how the known perturbation theory for closed \gtwo-structures with small torsion allow us to 
construct a $1$-parameter family of \gtwo-holonomy metrics on $M$. In summary, this section reduces the problem of 
finding twisted connected sum \gtwo-manifolds to finding \hk rotations between a given pair of  (deformation families) of 
ACyl Calabi-Yau $3$-folds.

Section \ref{sec:top} develops various tools to compute topological invariants of compact 
twisted connected sum \gtwo-manifolds building on results on the cohomology of the $6$-dimensional 
building blocks proved in \cite{chnp1}.
Theorem \ref{thm:g2topology} computes the integral cohomology groups of our 
twisted connected sum \gtwo-manifolds and proves under our assumptions that they are all simply-connected.
In general there can be torsion in $H^{3}(M)$ and $H^{4}(M)$; we establish various 
restrictions on the possible torsion linking form of general twisted connected sums. 
We review the almost-diffeomorphism classification of closed $2$-connected $7$-manifolds including cases 
in which almost-diffeomorphism can be replaced with diffeomorphism. 
We explain a simple sufficient condition for our twisted connected sum \gtwo-manifold $M$ 
to be $2$-connected with torsion-free $H^{4}(M)$ and then 
study the diffeomorphism and almost-diffeomorphism types of such $7$-manifolds. 
A key role is played by the divisibility of the first Pontrjagin class $p_{1}(M)$; 
in turn the divisibility of $p_{1}(M)$ is  closely related to the divisibility of the second Chern class $c_{2}$ of 
our building blocks, as studied in \cite{chnp1}. We apply all these results later to study the diffeomorphism 
type of concrete \gtwo-manifolds constructed in Section \ref{sec_g2mfds}, but our methods apply 
to twisted connected sum manifolds more generally.

Section \ref{sec:geom_assoc} deals with the construction of associative submanifolds in our 
twisted connected sum \gtwo-manifolds. After recalling some basic features of the geometry of 
associative $3$-folds we prove Theorem \ref{perturbthm}: this gives 
a general persistence result for a closed infinitesimally rigid associative $3$-fold $A$ under 
a small deformation of the \gtwo-structure. Theorem \ref{thm:genperturb} gives a generalisation 
in which we prove persistence of  multi-parameter families of closed associative $3$-folds 
in a multi-parameter torsion-free deformation of the original \gtwo-structure 
under a surjectivity assumption on the family.
These results are extensions of the deformation theory established by McLean
\cite[\S 5]{mclean}.
Next we consider the relation between associative $3$-folds 
in the product of a circle $\Sph^{1}$ and a Calabi-Yau $3$-fold and holomorphic curves in the Calabi-Yau. 
This relation is fundamental to the construction of associative $3$-folds in many of our twisted connected sum 
\gtwo-manifolds. Particularly important is the simple observation---see Lemma \ref{l:rigidity}---that $\Sph^{1}\times C$ 
is rigid as an associative $3$-fold if and only if $C$ is rigid as a holomorphic curve.
Putting together the results from the section we prove Proposition \ref{prop:cx_to_assoc}: 
each closed rigid holomorphic curve $C$ in one of our ACyl Calabi-Yau $3$-folds can be perturbed to a
compact rigid associative $3$-fold diffeomorphic to $\Sph^{1} \times C$ 
in our twisted connected sum \gtwo-manifold for all sufficiently long ``neck lengths''.
With a little more work, we also show how to produce closed associative
$3$-folds, including some non-rigid ones, in twisted connected sum
\gtwo-manifolds from certain closed special Lagrangian $3$-folds in our ACyl
Calabi-Yau $3$-folds.

Section \ref{sec:k3} deals with the so-called ``matching problem'', \ie the
construction of pairs of ACyl Calabi-Yau $3$-folds with a \hk rotation.
Our approach to solving the matching problem requires some well-known facts
about moduli spaces of lattice polarised K3 surfaces, the global Torelli
theorem in this context and some results from deformation theory proved
in \cite{chnp1}; we review these very briefly.
We describe in detail one general strategy which we call ``orthogonal gluing'' 
which  yields solutions to the matching problem in a large number of cases. 

A special case of ``orthogonal gluing'' is what we term ``primitive perpendicular gluing''. 
The benefit of ``primitive perpendicular gluing'' is that it is guaranteed to give a solution of the matching 
problem under the simple condition that the ranks of the Picard lattices 
of the deformation types $\sff_{\pm}$ of semi-Fano $3$-folds used to construct
the ACyl Calabi-Yau $3$-folds add up to at most 11. 
Primitive perpendicular gluing enables us to mass-produce  
\mbox{$2$-connected} twisted connected sum \gtwo-manifolds and to understand their homeomorphism and 
diffeomorphism type in many cases.

The more general ``orthogonal gluing'' produces manifolds $M$ with second Betti number $b^{2}(M) > 0$. 
It requires some sort of compatibility between the Picard lattices of the pair of semi-Fanos initially chosen 
and this does not always hold: see Example \ref{exa:no_pushout}. 
When it does hold, some further restrictions on the ranks of the Picard lattices allow the solution 
of the matching problem in this case: see Proposition \ref{prop:orth_gluing}.

In Section \ref{sec_g2mfds} we make examples of twisted connected sum \gtwo-manifolds---often containing 
compact rigid associative $3$-folds---and compute the topology of these examples. 
In this section we aim to give examples that illustrate the main points of what 
is achievable by the construction rather than to be systematic. 
Our examples are built mainly using specific ACyl Calabi-Yau $3$-folds of semi-Fano type
which we constructed recently in \cite{chnp1}. 
Most of the examples we construct use perpendicular or orthogonal gluing. 
However, to demonstrate that these are labour-saving devices rather than a necessity 
we also consider what we have termed ``handcrafted nonorthogonal gluing''. 
While this method applies in situations where ``orthogonal gluing'' fails it needs much more 
precise information about K3 moduli spaces and that information is usually very expensive to obtain. 
Therefore we give only a single example to illustrate the  method and its potential complexities.

In Section \ref{sec:close} we describe the more general possibilities and limitations 
of the construction and make some comments about the ``geography'' of examples 
that can be achieved by matching currently known pairs. 
The existing smooth classification theory for $7$-manifolds allows us to determine the diffeomorphism type of the majority of 
the $2$-connected twisted connected sums; to complete the diffeomorphism
classification in the remaining $2$-connected cases we would need to be able to
compute an extension of the classical Eells-Kuiper invariant. 

Moving beyond the $2$-connected world, 
we explain how to realise many simply-connected twisted connected sum \gtwo-manifolds with $\pi_{2}(M)$ 
isomorphic to $\Z$ or $\cg{k}$ and discuss the prospects for extending the smooth classification theory 
to cover these cases. 
Finally we discuss a way to organise various different twisted connected sum \gtwo-manifolds 
constructed by matching pairs of Fano or semi-Fano $3$-folds related via flops
or conifold transitions; 
by analogy we term these \emph{\gtwo-transitions}.
We close by discussing the prospects for proving more precise topological and metric statements 
about \gtwo-manifolds related by these \gtwo-transitions.

\subsection*{Acknowledgements}
The authors would like to thank Bobby Acharya, Kevin Buzzard, Paolo Cascini, 
Tom Coates, Diarmuid Crowley, Simon Donaldson, 
Igor  Dolgachev, Bert van Geemen,  Anne-Sophie Kaloghiros, Al Kasprzyk and Viacheslav Nikulin.
Computations related to toric semi-Fanos were performed in collaboration with Tom Coates and Al Kasprzyk
and were carried out on the Imperial College mathematics cluster and the Imperial College High Performance
Computing Service;  we thank Simon Burbidge, Matt Harvey, and Andy
Thomas for technical assistance. 
Part of these computations were performed on hardware supported by AC's EPSRC grant EP/I008128/1.
MH would like to thank the EPSRC for their continuing support of his research under Leadership Fellowship 
EP/G007241/1, which also provided postdoctoral support for JN. 
TP gratefully acknowledges the financial support provided by a Marie Curie European Reintegration Grant.

%% file: sec2.tex
In this section we collect some basic facts and definitions concerning the
linear algebra and geometry associated to the Lie groups \gtwo and \sunitaryn.
The material in this section is standard and the reader may find proofs of
various quoted facts in the articles by Bryant \cite{bryant:exceptionalholonomy}
and Harvey-Lawson \cite{harveylawson:calgeometry} and the books by Joyce
\cite{joyce:holonomybook} and Salamon \cite{salamon:holonomybook}. 
We include this material to establish our conventions and notation and to 
make the paper more self-contained and accessible to topologists and algebraic geometers.

\subsection*{The octonions, a cross product on $\R^7$ and the group \gtwo}
\label{ss:G2_algebra}

One way to define \gtwo is as the automorphism group of $\bbo$,
the normed algebra of octonions.
The automorphisms preserve the splitting $\bbo = \bbr \oplus \Imag \bbo$
and act trivially on $\bbr$, so can therefore be identified with a subgroup
of $\glr{7}$.
Since the inner product on $\Imag \bbo$ is defined in terms of the normed
algebra structure it is preserved by the automorphisms. We will see below that
the automorphisms also preserve orientation, so $\gtwo$ can be embedded in
$\sorth{7}$.

If we choose an isometry $\Imag \bbo \cong \bbr^{7}$ then we can define a
vector product on $\bbr^{7}$ by
\[ u \times v = \Imag uv . \]
The algebra structure on $\bbr \oplus \Imag \bbo$ can be recovered from the
vector product $\times$ and the standard inner product $g_{0}$ by
\[ (x,u)(y,v) = (xy - g_0(u,v), \, xv + yu + u \times v) . \]
An equivalent definition of \gtwo is therefore that it is the subgroup
of $\glr{7}$ that preserves both $g_{0}$ and $\times$.
From $g_{0}$ and $\times$ we can define the trilinear form
\begin{equation}
\label{multtableeq}
\varphi_{0}(u,v,w) = g_0(u \times v, w) .
\end{equation}
In fact this is alternating, so $\varphi_{0} \in \Lambda^{3}(\bbr^{7})^{*}$.
With a standard choice of isometry $\Imag \bbo \cong \bbr^7$
that we fix once and for all (our convention is the same as that used by
\eg Joyce \cite[\S 10]{joyce:holonomybook}) we can write
\begin{equation}
\label{eq:3form}
\varphi_{0} = dx^{123} + dx^{145} + dx^{167} + dx^{246}
 - dx^{257} - dx^{347} - dx^{356} .
\end{equation}
For any $\varphi \in \Lambda^3(\R^7)^*$ we can define a
form $vol_{\varphi} \in \Lambda^7(\R^7)^*$ in the following way:
For $v,w \in \R^7$ let
\[ B_{\varphi}(v,w) = \frac{1}{6} (v \lrcorner \varphi) \wedge
(w \lrcorner \varphi) \wedge \varphi . \]
$B_{\varphi}$ is a symmetric bilinear form on $\R^7$ with values in
$\Lambda^7(\R^7)^*$.
$B_{\varphi}$ induces a linear map
$K_{\varphi} : \R^{7}\to (\R^{7})^{*} \otimes \Lambda^{7}(\R^7)^{*}$, which has a determinant
$\det K_{\varphi} \in (\Lambda^{7}(\R^7)^{*})^{9}$.
If $K_\varphi \not= 0$ then we call $\varphi$ \emph{non-degenerate}, and
we can define a volume form $vol_\varphi$ and a symmetric bilinear form 
$g_{\varphi}$ on $\R^{7}$ by
\begin{subequations}
\label{eq:g2metric}
\begin{align}
(vol_{\varphi})^{9} &= \det K_{\varphi} , \\
g_{\varphi} \otimes vol_{\varphi} &= B_{\varphi},
\end{align}
\end{subequations}
(see Hitchin \cite[\S 7]{hitchin00}).
For $\varphi_{0}$ we can compute that $g_{\varphi_{0}} = g_{0}$, so the metric
can be recovered from~$\varphi_{0}$, and hence so can the vector
product $\times$.
Thus the stabiliser of $\varphi_{0}$ in $\glr{7}$ preserves $g_{0}$
and~$\times$, and must equal~\gtwo.
This gives yet another possible definition of \gtwo.
Since it is in terms of an alternating 3-form it is
a useful one for the purposes of differential geometry.

The set of 3-forms that are equivalent to $\varphi_0$, and whose associated
orientation, symmetric bilinear form and cross product are thus isomorphic
to the standard one, is in fact open in~$\Lambda^3 (\bbr^7)^*$.

\begin{prop}
\label{p:g2group}
\hfill
\begin{enumerate}
\item $\gtwo$ is a compact 2-connected Lie group of dimension $14$.
\item \label{it:su3}
The stabiliser in \gtwo of a non-zero vector in $\bbr^7$ is isomorphic
to \sunitary{3}.
\item \gtwo acts transitively on the unit sphere $\Sph^6 \subset \bbr^7$.
\item \label{it:openorbit}
The $\glr{7}$-orbit of $\varphi_0$ is open in $\Lambda^3(\bbr^7)^*$.
\end{enumerate}
\end{prop}

\begin{proof}
Since $\dim \Lambda^3(\bbr^7)^* = 35$ and $\dim \glr{7} = 49$, we must
have $\dim \gtwo \geq 14$ with equality if and only if the orbit of $\varphi_0$
is open.

We will prove below that the stabiliser in $\gtwo$ of $e_1$
can be identified with \sunitary3 in a natural way.
Because $\dim \sunitary{3} = 8$, the $\gtwo$-orbit of $e_1$ must have dimension
$\geq 6$. Since the orbit is contained in $\Sph^6$ equality must hold.
Consequently the $\gtwo$-orbit of $e_1$ is exactly $\Sph^6$, all unit vectors
have isomorphic stabilisers, $\dim \gtwo$ is exactly 14, and the
$\glr{7}$-orbit of $\varphi_0$ is open. The fibration
$\sunitary{3} \to \gtwo \to \Sph^6$ shows that $\gtwo$ is 2-connected.
\end{proof}

\begin{remark*}
The set of non-degenerate 3-forms on $\bbr^7$ is in fact the union of four
connected components: two $\glr{7}$-orbits, each of which splits into two
components inducing opposite orientation. The orbit not containing $\varphi_0$
consists of those non-degenerate 3-forms whose induced bilinear form has
signature $(3,4)$.
\end{remark*}

The Hodge dual $* \varphi_0$ of $\varphi_0$ is a $4$-form $\psi_0$
\begin{equation}
\label{eq:4form}
\psi_0 = -dx^{1247} - dx^{1256} - dx^{1346} + dx^{1357} + dx^{2345} +
dx^{2367} + dx^{4567} .
\end{equation}
We can use $\psi_0$ and the metric to obtain an alternating vector-valued
$3$-form $\chi_0: \R^7 \times \R^7 \times \R^7 \ra \R^7$ defined by
\begin{equation}
\label{E:triple:cross}
g_{0}(u,\tfrac{1}{2}\chi_0(v,w,x)) = \psi_0(u,v,w,x) \quad
\text{for all \ } u,v,w,x \in \R^7.
\end{equation}
$\chi_{0}$ is not a proper triple cross-product on $\R^{7}$ in the sense that there exist orthonormal triples $(u,v,w)$ with 
$\chi_{0}(u,v,w)=0.$

\begin{remark}
\label{rem:psi}
The stabiliser of $\psi_0$ in \glr7 is the subgroup $\Z_2 \times \gtwo$,
where $\Z_2$ is generated by $-\Id$. We can therefore recover $\varphi_0$
from $\psi_0$, modulo orientation.
\end{remark}

\begin{lemma}
\label{L:cross:product}
For all $u, v, w \in \R^7$ 
\begin{subequations}
\begin{eqnarray}
\label{eq:cross_product_vol}
\|u\times v\|^2 & = & \|u\|^2\|v\|^2-g_{0}(u,v)^2,\\
\label{eq:cross_product_triple}
u\times(v\times w) + (u \times v) \times w & = & 
2g_{0}(u,w)v -g_{0}(u,v)w -g_{0}(w,v)u,\\
\label{eq:phi'n'chi}
\varphi_0(u,v,w)^2+\quart|\chi_0(u,v,w)|^2 & = &
\abs{u \wedge v \wedge w}^2.
\end{eqnarray}
\end{subequations}
\end{lemma}
\begin{proof}
See \cite[p. 540]{bryant:exceptionalholonomy}, \cite[2.2]{bolton:woodward} and 
\cite[Thm. IV.1.6]{harveylawson:calgeometry}
for proofs of  \eqref{eq:cross_product_vol}, \eqref{eq:cross_product_triple} and \eqref{eq:phi'n'chi} respectively.
\end{proof}

\subsection*{$G$-structures on vector spaces}
\label{ss:G_on_vector}
Let $V$ be a $n$-dimensional real vector space. Let $P$ denote the set of
ordered bases of $V$; equivalently, the set of isomorphisms
$\beta:\R^n\rightarrow V$. We call $P$ the set of frames of $V$. $P$ has a free
transitive right $\glrn$-action determined by composition of maps:
\begin{equation*}
 g\cdot\beta:=\beta\circ g.
\end{equation*}
We can thus think of $P$ as a principal $\glr{n}$-fibre bundle over a point. 

\begin{definition}
\label{d:G_structure}
Let $G$ be a subgroup of $\glrn$. A \textit{$G$-structure} on $V$ is a
$G$-subbundle of~$P$, \textit{i.e.} an orbit of the induced action of $G$ on
$P$. The space of all $G$-structures can be identified with the quotient space
$P/G$.
\end{definition}

The above definition makes it clear that if $H$ is a subgroup of $G$, an
$H$-structure automatically defines a $G$-structure. 

\subsubsection*{\gtstr s on a vector space}
The subgroups $G$ of interest in this paper arise as isotropy groups of
algebraic structures on $\R^n$. In such cases one can give an alternative
definition of $G$-structure, which we exemplify in the case $G = \gtwo$. 

\begin{definition}
\label{def:g2form}
Let $V$ be a real vector space of dimension $7$. %
We call $\varphi \in \Lambda^3 V^*$ a \emph{\gtstr} %
(or \emph{\gtwo-form}) if there is a linear isomorphism $V \cong \bbr^7$
identifying $\varphi$ with $\varphi_0$.
\end{definition}

Since $\gtwo \subset \sorth{7}$, a \gtstr{} on $V$ induces an inner product and
an orientation. This is equivalent to how a \gtwo-form defines an inner product
and orientation in \eqref{eq:g2metric}. We will often find it convenient to
restrict attention to those \gtstr s that agree with a given orientation.

\begin{definition}
\label{def:positive_3forms}
Let $V$ be a real oriented $7$-dimensional vector space.
We call $\varphi \in \Lambda^3 V^*$ a \emph{positive 3-form}
if there is an oriented linear isomorphism $V \cong \bbr^7$
identifying $\varphi$ with $\varphi_0$.
Let $\Lambda^3_+ V^* \subset \Lambda^3 V^*$ denote the set of positive forms.
\end{definition}

\noindent
Note that $\Lambda^3_+ V^*$ is open in $\Lambda^3 V^*$ by
\ref{p:g2group}\ref{it:openorbit}. By Remark \ref{rem:psi}, we
could study \gtstr s on an oriented vector space equivalently in terms of the
Hodge duals of the positive 3-forms.

\begin{remark*}
Our definition of `positive' agrees with that of Joyce
\cite{joyce:holonomybook}, while Hitchin \cite{hitchin00} uses `positive' where
we use `\gtwo-form'.
\end{remark*}

\subsubsection*{\sunitaryn-structures}
Let $z^1, \ldots, z^n$ be standard coordinates on $\bbc^n$, and
\begin{subequations}
\begin{eqnarray*}
\Omega_0 &=& \phantom{ab} dz^1 \wedge \cdots \wedge dz^n,\\
\omega_0 &=& {\textstyle\frac{i}{2}}(dz^{1} \wedge d \bar z^{1} + \cdots +
dz^{n} \wedge d \bar z^{n}).
\end{eqnarray*}
\end{subequations}
These are, respectively, the standard complex volume form and Kähler form,
and are invariant under the action of \sunitaryn. In fact, their stabiliser
in $\glr{2n}$ is precisely \sunitaryn. For $\Omega_0$ on its
own determines $\Lambda^{1,0}_\bbc(\bbc^n)^*$ (as the kernel of
$\alpha \mapsto \Omega_0 \wedge \alpha$) and hence the complex structure on
$\bbc^n$, so the stabiliser of $\Omega_0$ in $\glr{2n}$ is precisely
$\SL{n}$.

By analogy with Definition \ref{def:g2form}, we can think of any complex
$n$-form $\Omega$ that is $\glr{2n}$-equivalent to $\Omega_0$ (\ie any
decomposable form) as defining an $\SL{n}$-structure, and any pair
$(\Omega,\omega)$ of a decomposable complex $n$-form and a non-degenerate real
2-form such that
\begin{subequations}
\begin{align}
\label{eq:omega:11}
\Omega \wedge \omega &= 0 , \\
\label{eq:sun:vol}
(-1)^{\frac{n(n-1)}{2}} \left(\tfrac{i}{2}\right)^n \Omega \wedge
\overline{\Omega} &= \frac{\omega^{n}}{n!},
\end{align}
\end{subequations}
as an $\sunitaryn$-structure. \eqref{eq:omega:11} encodes that $\omega$
is $(1,1)$ with respect to the complex structure defined by $\Omega$, while 
\eqref{eq:sun:vol} is a normalisation condition that the natural volume forms
defined by $\omega$ and $\Omega$ are equal, or equivalently that
$\norm{\Omega} = 2^n$ (see Hitchin \cite[\S 2]{hitchin97}).

\subsubsection*{\sunitary{3}-structures}
We have a particular interest in the case of complex dimension three since
\sunitary{3} is the stabiliser in \gtwo of a vector in $\bbr^7$. Let us now
give the previously promised proof of this fact.

\begin{proof}[Proof of Proposition \ref{p:g2group}\ref{it:su3}]
Let $S$ be the stabiliser of the basis vector $e_1 \in \Sph^{6}\subset \R^{7}$. Since $\gtwo\subset\sorth{7}$,
$S$ maps the orthogonal complement $e_{1}^{\perp}$ to itself. $e_{1}^{\perp}$
can be identified with $\bbc^3$ by introducing complex coordinates
$z^{1} = x^{2}+ix^{3}$, $z^{2} = x^{4}+ix^{5}$, $z^{3} = x^{6}+ix^{7}$.
The action of $S$ on $\bbc^3$ evidently preserves the forms
\begin{subequations}
\begin{align*}
e_1 \lrcorner \varphi_0 = dx^{23} + dx^{45} + dx^{67} &= \omega_{0}, \\
\varphi_0|_{e_1^\perp} = dx^{246} - dx^{257} - dx^{347} - dx^{356} &=
\Real\Omega_0 , \\
-e_1 \lrcorner \psi_0 = -dx^{247} - dx^{256} - dx^{346} + dx^{357} &=
\Imag\Omega_0 ,
\end{align*}
\end{subequations}
so $S$ is contained in \sunitary{3}. Conversely
\begin{equation}
\label{eq:g2productmodel}
\varphi_0 = dx^1 \wedge \omega_0 + \Real \Omega_0
\end{equation}
implies that \sunitary3 preserves $\varphi_0$, so $S$ is precisely \sunitary3.
\end{proof}

It follows that any \sunitary{3}-structure on a real vector space $V$ of 
dimension 6 (together with a covector $dt$ on $\bbr$ defining orientation and
length) determines a \gtstr{} on $\bbr \oplus V$. Moreover we see from the
proof how to express the relationship between the structures in terms of the
forms. If the \sunitary{3}-structure on $V$ is defined by $(\Omega,\omega)$
then the induced \gtstr{} on $\bbr \oplus V$ has \gtwo-form
\begin{equation}
\label{eq:cylg2form}
\varphi = dt \wedge \omega + \Real \Omega.
\end{equation}
Similarly, the Hodge dual $4$-form $\psi$ of $\varphi$ takes the form 
\begin{equation}
\label{eq:4form:six:dims}
\psi = \half \omega^2 - dt \wedge \Imag \Omega.
\end{equation}
Another way to think of the relationship is that the orthogonal complement
to a unit vector $u$ in a vector space with \gtstr{} inherits (in addition to
the metric) two structures from the cross product:
using Lemma \ref{L:cross:product}, $I_u : v \mapsto u \times v$ defines an
orthogonal complex structure on~$u^\perp$, while the restriction/projection of
the cross product to $u^\perp$ defines a $I_u$-antilinear cross product which
is equivalent to a complex volume form (because the complex dimension is 3).
See also p. \pageref{ss:ass_vs_complex}.

\begin{remark}
\label{r:SL3_open_orbit}
Complex volume forms in dimension three have some special properties.
Hitchin \cite[\S 2]{hitchin00} explains that the stabiliser of
$\Real \Omega_0$ alone in $\mathrm{GL}_+(6, \bbr)$ is $\SL{3}$.
The $\glr{6}$-orbit of $\Real\Omega_0$ in $\Lambda^3(\bbr^6)^*$ is therefore
open by dimension counting: $\dim \glr{6} - \dim \SL{3} = 36-16 = 20 = \dim
\Lambda^3(\bbr^6)^*$.
For any 3-form $\alpha$ in this open set there is a unique real \mbox{3-form}
$\beta$ such that $\alpha + i\beta$ is decomposable and
the induced $\SL{3}$-structure has the standard orientation.
For a real vector space of dimension 6, an $\SL{3}$-structure is therefore
equivalent to a choice of orientation together with a 3-form equivalent to
$\Real\Omega_0$ (reversing the orientation while keeping the 3-form fixed
corresponds to replacing the complex structure by its conjugate).
\end{remark}

\subsubsection*{\sunitary{2}-structures}
\label{sss:su2}
The case of complex dimension two also plays an important role in the paper. 
Let $\omega^I_0 := \omega_0$ be the standard Kähler form on $\bbc^2$, and write
the holomorphic volume form $\Omega_0$ as $\omega^J_0 + i\,\omega^K_0$.
As suggested by the notation, $\omega^J_0$ and $\omega^K_0$ define
$g_{0}$-orthogonal complex structures $J$ and $K$ on $\bbr^4$ by the relations
$\omega^J_0(x,y) = g_0(Jx,y)$ and $\omega^K_0(x,y) = g_0(Kx,y)$.
In real coordinates $(x^i)$ where $z^1 = x^1 +ix^2$, $z^2 = x^3+ix^4$
\[ \omega^I_0 = dx^{12} + dx^{34}, \quad \omega^J_0 = dx^{13} - dx^{24},
\quad \omega^K_0 = dx^{14} + dx^{23}. \]
When we identify $\bbc^2$ with the quaternions $\quat$ by
$(x^1 + ix^2, x^3 + ix^4) \mapsto x^1 + ix^2 + jx^3 + kx^4$
the complex structures $I, J, K$ correspond to left multiplication
by the standard orthonormal triple $i, j, k$ of imaginary quaternions. This
identifies \sunitary2 with the automorphism group \Sp 1 of~$\quat$.
Furthermore, any unit imaginary quaternion defines an orthogonal complex
structure, so \sunitary2 preserves a whole $\Sph^2$ of complex structures.

We can therefore think of an \sunitary2-structure on a 4-dimensional vector
space in two different ways: either as a pair $(\omega, \Omega)$ as before,
or as a choice of an ordered triple of 2-forms $(\omega^I, \omega^J, \omega^K)$
equivalent to $(\omega^I_0, \omega^J_0, \omega^K_0)$, \ie satisfying
\begin{gather*}
(\omega^I)^2 = (\omega^J)^2 = (\omega^K)^2, \\
\omega^I \wedge \omega ^J = \omega^J \wedge \omega^K = \omega^K \wedge \omega^I
= 0 .
\end{gather*}
These two definitions of \sunitary2-structures are equivalent, setting
$\omega = \omega^I$ and $\Omega = \omega^J + i\,\omega^K$. However, the first
highlights a preferred complex structure $I$, while the second emphasises
the two-sphere of complex structures. We will switch back and
forth between these two points of view.

If we want to choose an \sunitary2-structure compatible with a particular
inner product and orientation we first choose $\omega^I$ in the $\Sph^2$ of
2-forms such that $(\omega^I)^2 =2 \vol$, and then $\omega^J$ among the
$\Sph^1$ of such forms that are perpendicular to $\omega^I$ (and $\omega^K$
is then determined by $K = IJ$). All in all, there is therefore an
\sorth3-family of \sunitary2-structures inducing the same inner product and
orientation.

\begin{remark*}
Any non-degenerate complex 2-form with square 0 on a real vector space of
dimension 4 is decomposable, and thus determines an $\SL{2}$-structure.
\end{remark*}

\subsection*{Calibrations in $\R^7$}
\label{subs:calibrations}
Let $(V,g)$ be an inner product space. A $k$-form $\alpha \in \Lambda^k V^*$
is said to be a \textit{calibration} if, for every oriented $k$-plane $\pi$
in $V$, we have $\alpha_{|\pi}\leq vol_\pi$. The oriented $k$-planes $\pi$ for
which $\alpha_{|\pi}= vol_\pi$ are said to be \textit{calibrated}.

A $\gtwo$-form $\varphi$ and its Hodge dual $\psi$ define calibrations with respect to the metric $g_{\varphi}$.
\begin{lemma}
\label{L:calibrations}
\hfill
\begin{enumerate}
\item The $3$-form $\varphi_{0}$ and the $4$-form $\psi_0=*\varphi_{0}$ defined
in \eqref{eq:3form} and \eqref{eq:4form} respectively are calibrations on
$(\R^7,g_{0})$.
\item If $u, v, w$ is an orthonormal triple of vectors in $\R^7$, then
$\varphi_{0}(u,v,w)=1$ if and only if $w=u \times v$.
\item If $u,v, w, x$ is an orthonormal quadruple of vectors in $\R^7$ then
$\psi_{0}(u,v,w,x)=1$ if and only if $u=\tfrac{1}{2} \chi_{0}(v,w,x)$.
\end{enumerate}
\end{lemma}

\begin{proof}
For any orthonormal quadruple $u, v, w, x \in \R^7$ using Cauchy-Schwarz,
\eqref{eq:cross_product_vol} and \eqref{eq:phi'n'chi} we have
\begin{equation}
\label{E:phi:comass}
\varphi_{0}(u,v,w) = g_{0}(u\times v, w) \le \norm{u\times v} \norm{w} = 1,
\end{equation}
and
\begin{equation}
\label{E:psi:comass}
\abs{\psi_{0}(u,v,w,x)} = \abs{g_{0}(u, \tfrac{1}{2} \chi_{0}(v,w,x))} \le
\norm{u} \norm{\tfrac{1}{2} \chi_{0}(v,w,x)} \le 1.
\end{equation}
If $w= u \times v$ then $\varphi_{0}(u,v,w)=g_{0}(u\times v, u\times v)=1$.
Conversely, if $\varphi_{0}(u,v,w)=1$, then equality must hold throughout
\eqref{E:phi:comass} and in particular in the Cauchy-Schwarz inequality. Hence
$w=\lambda u \times v$ for some $\lambda \in \R$. But
$1 = \varphi_{0}(u,v,\lambda u \times v) = \lambda g_{0}(u\times v, u \times v)
= \lambda$, hence we must have $w = u \times v$.

Similarly we have equality in \eqref{E:psi:comass} 
if and only if $u = \lambda \tfrac{1}{2} \chi_{0}(v,w,x)$ for some $\lambda \in
\R$ and $\norm{\tfrac{1}{2} \chi_{0}(v,w,x)}=1$. Hence equality holds in
\eqref{E:psi:comass} if and only if $u= \pm \tfrac{1}{2} \chi_{0}(v,w,x)$, and
clearly we have $\psi_{0}(\pm \tfrac{1}{2} \chi_{0}(v,w,x),v,w,x)=\pm 1$.
\end{proof}
\begin{definition}
An oriented $3$-plane $\pi$ in $\R^7$ calibrated by $\varphi_{0}$ is called an
\textit{associative} plane.
An oriented $4$-plane $\pi$ in $\R^7$ calibrated by $\psi_{0}$ is called a
\textit{coassociative} plane.
\end{definition}

\begin{lemma}\label{l:ass_characterization}\hfill
\begin{enumerate}
\item \label{it:ass_char}
A 3-plane $\pi$ is associative (for one choice of orientation) if and only
if $\chi_{0|\pi}=0$.
\item \label{it:unique_ass_ext}
Any 2-plane is contained in a unique associative 3-plane.
\end{enumerate}
\end{lemma}

\begin{proof}
\ref{it:ass_char} follows directly from \eqref{eq:phi'n'chi} and the fact that
$\varphi_0$ is a calibration. \\
\ref{it:unique_ass_ext} Let $\{u,v\}$ be an orthonormal basis for the
$2$-plane. Then $\{u,v,u\times v\}$ is an oriented orthonormal basis for an
associative $3$-plane. Suppose $\pi$ is any associative $3$-plane containing
the $2$-plane $\gen{u,v}_\R$. Then we can choose an oriented
orthonormal basis $\{u,v,w\}$ for $\pi$ extending $\{u,v\}$. Hence by Lemma
\ref{L:calibrations} we must have $w= u\times v$.
\end{proof}

\subsubsection{Relation to calibrations on $\bbc^3$}
There are also standard calibrations on $\bbc^n$, given by powers of the
standard Kähler form and real parts of normalised $(n,0)$-forms. The fact
that $\omega_0^k$ is a calibration for each $k$, and that the calibrated
subspaces are precisely the complex $k$-planes, is known as \emph{Wirtinger's
inequality}. The other type of calibration is described by the following lemma.

\begin{lemma}
\label{L:sl_calibrations}
\hfill
\begin{enumerate}
\item \label{it:sl_cal}
The $n$-forms $\Real (e^{i\theta}\Omega_0)$ are calibrations on
$(\bbc^n,g_{0})$ for each $\theta \in \bbr$.
\item \label{it:sl_char}
A real $n$-plane $L \subset \bbc^n$ is calibrated by $\Real \Omega_0$ (for one
choice of orientation) if and only if $\omega_{0|L} = \Imag \Omega_{0|L} = 0$.
\end{enumerate}
\end{lemma}

\begin{proof}
It is easy to see that $\norm{\Omega_{0|L}} \leq 1$ for any real $n$-plane
$L \subset \bbc^n$, with equality if and only $\omega_{0|L} = 0$. Thus
$\norm{\Real(e^{i\theta}\Omega_0)_{|L}} \leq 1$ with equality if and only if
$L$ is Lagrangian and $\norm{\Imag(e^{i\theta}\Omega_0)_{|L}} = 0$.
\end{proof}

\noindent
Note that for each Lagrangian plane $L \subset \bbc^n$ there is a $\theta$
(unique modulo $2\pi$) such that $L$ is calibrated by
$\Real(e^{i\theta}\Omega_0)$.

\begin{definition}
We call the planes calibrated by $\Real(e^{i\theta}\Omega_0)$ \emph{special
Lagrangian with phase $\theta$}, or simply \emph{special Lagrangian} if
$\theta = 0$.
\end{definition}

Now consider $\bbc^3$ with its standard \sunitary3-structure
$(\Omega_0,\omega_0)$ as a hyperplane in
$\bbr^7 \cong \gen{e_1} \oplus \bbc^3$ with the standard product \gtstr{}
$\varphi_0 = dx^1 \wedge \omega_0 + \Real \Omega_0$ given in \eqref{eq:g2productmodel}.

\begin{lemma}
\label{L:linear:ass_vs_complex}
\hfill
\begin{enumerate}
\item \label{it:ass_vs_complex}
Let $\ell \subset \bbc^3$ be a real 2-plane. Then $\gen{e_1} \oplus \ell$ is
associative in $\bbr^7$ if and only if $\ell$ is a complex line.

\item \label{it:ass_vs_sl}
Let $L \subset \bbc^3$ be a real 3-plane. Then $L$ is associative in $\bbr^7$
if and only if $L$ is special Lagrangian.
\end{enumerate}
\end{lemma}

\begin{proof}
\ref{it:ass_vs_complex} $\varphi_{0|{\gen{e_1} \oplus \ell}} = \omega_{0|\ell}$,
so $\gen{e_1} \oplus \ell$ is calibrated by $\varphi_0$ if and only if
$\ell$ is calibrated by $\omega_0$.
\ref{it:ass_vs_sl} $\varphi_{0|L} = \Real \Omega_{0|L}$, so $L$ is calibrated by
$\varphi_0$ if and only if $L$ is calibrated by $\Real \Omega_0$.
\end{proof}

We can also think of \ref{it:ass_vs_complex} the following way.
Let $V$ be a 7-dimensional vector space with a \gtstr, $u \in V$ a unit
vector, and consider the orthogonal complement $u^\perp$ with its induced
$\sunitary3$-structure \eqref{eq:cylg2form}.
The complex structure on $u^\perp$ is $I_u : v \mapsto u \times v$.
So for $v \in u^\perp$, the unique associative 3-plane in $V$ containing
both $u$ and $v$ is $\gen{u, v, I_u v}_\R$, which is the direct sum
of $\gen{u}$ and the unique complex line in $u^\perp$ containing $v$.

\subsection*{$G$-structures and manifolds with special holonomy}
\label{ss:G_on_manifold}
Let $M$ be a smooth $n$-dimensional manifold. Let $GL(M)$ denote the principal
$\glr{n}$-bundle of linear frames on $M$.

\begin{definition}
Let $G$ be a subgroup of $\glr n$. A $G$-structure on $M$ is a $G$-subbundle of
$GL(M)$. Equivalently, it is a smooth section of the quotient bundle $GL(M)/G$.
\end{definition}
\noindent
The $G$-structures of interest to us can equivalently be defined in terms of a
choice of special algebraic structure on $M$.

\subsubsection*{\gtstr s and manifolds with holonomy \gtwo}

\begin{definition}\label{def:G2_structure}
For an oriented manifold $M$ of dimension 7,
let $\Lambda^3_+T^*M \subset \Lambda^3T^*M$ be the smooth subbundle of positive
3-forms, in the sense of Definition \ref{def:positive_3forms}.
A \emph{\gtstr} on $M$ (compatible with its orientation) is a smooth
section of $\Lambda^3_+T^*M$, \ie a smooth 3-form $\varphi$ such that for
each $x \in M$ there is an oriented isomorphism
$(T_xM, \varphi) \cong (\bbr^7, \varphi_0)$.
\end{definition}

It follows from Proposition \ref{p:g2group}\ref{it:su3} that
$\Lambda^3_+T^*M$ is an open subset of $\Lambda^3T^*M$; in particular, any
small perturbation of a \gtstr{} $\varphi$ is again a \gtstr.

\begin{remark}
\label{rmk:g2exist}
The existence of \gtstr s on a manifold is a topological question. \gtwo is
simply connected by Proposition \ref{p:g2group}, so $\gtwo \into \sorth7$ lifts
to $\gtwo \into \spin7$, and any \gtstr{} induces a spin structure.
In fact, the converse also holds: a 7-manifold $M$ admits \gtstr s if and
only if it is orientable and spin (\cf Gray \cite[Theorem 3.2]{gray69}).

The reason is that \gtwo is the stabiliser of a non-zero element
in the (unique, real, rank 8) spin representation of $\spin7$. Therefore
on an oriented Riemannian 7-manifold with a spin structure, \gtstr s
compatible with the metric, orientation and spin structure (in other words,
\gtwo-subbundles of the given $\spin7$-bundle) correspond to
non-vanishing sections of the spinor bundle. Non-vanishing sections always
exist because the rank of the spinor bundle is greater than the dimension of
the base.
\end{remark}

A \gtstr{} $\varphi$ induces a Riemannian metric $g_{\varphi}$ on $M$,
and hence also a Levi-Civita connection $\nabla_{\varphi}$ and a Hodge star
$*_{\varphi}$.
We may drop the subscripts if the \gtstr{} is clear from the context.
The canonical 4-form $\psi = *\varphi$ is also important.

\begin{definition}
\label{torsionfreedef}
A \gtstr{} defined by a positive 3-form $\varphi$ is \emph{torsion-free}
if $\nabla_{\varphi} \varphi = 0$. 
\end{definition}

\begin{remark*}
There is a notion of the \emph{intrinsic torsion} of a $G$-structure on $M$ for
a general Lie subgroup $G \subseteq \glrn$ (see \eg Joyce
\cite[\S $2.6$]{joyce:holonomybook}).
A \gtstr{} has zero intrinsic torsion in this sense if and only if it
is torsion-free according to Definition \ref{torsionfreedef}.
\end{remark*}

It follows immediately from the definition of holonomy that
if $(M^{7},g)$ is a Riemannian manifold, then $\Hol(g)$ is a subgroup
of \gtwo if and only if there is a torsion-free \gtstr{}
$\varphi$ on $M$ such that $g = g_{\varphi}$.

\begin{definition}
A \emph{\gtmfd} is a manifold $M^{7}$ equipped with a torsion-free
\gtstr{} $\varphi$ and the associated Riemannian metric $g_{\varphi}$.
We say that $(M,\varphi)$ is a manifold \textit{with holonomy $\gtwo$} or
\textit{has holonomy $\gtwo$} if $\Hol(g_\varphi)=\gtwo$.
\end{definition}

Holonomy \gtwo is a much stronger condition on $M$ than the existence of a
\gtstr, involving the metric. For example, any such metric is
Ricci-flat (Salamon \cite[{Proposition~11.8}]{salamon:holonomybook}).
On the basis of Berger's classification of holonomy groups one can prove the
following, see Joyce \cite[p. 245]{joyce:holonomybook}.

\begin{prop}\label{prop:G2_topology}
A compact \gtmfd{} has holonomy \gtwo if and only if $\pi_1(M)$ is finite.
\end{prop}

Using Hodge theory and the decomposition of the exterior algebra of any
$\gtwo$-manifold into irreducible $\gtwo$-representations one can prove the
following additional restrictions on the topology of any compact
$\gtwo$-manifold $(M,\varphi,g)$ manifold with $\Hol(g)=\gtwo$.

\begin{prop}[\mbox{\cite[p. 246]{joyce:holonomybook}}]
\label{prop:p1}
Let $(M,\varphi,g)$ be a compact $\gtwo$-manifold with $\Hol(g)=\gtwo$, and $p_{1}(M) \in H^{4}(M;\Z)$ 
the first Pontrjagin class. Then 
\begin{enumerate}
\item
$(\alpha \cup \alpha \cup [\varphi])[M] <0$ for every nonzero $\alpha \in H^{2}(M;\R)$.
\item
$(p_{1}(M) \cup [\varphi])[M] < 0$. In particular $p_{1}(M) \neq 0$.
\end{enumerate}
\end{prop}

By considering how $d\varphi$ and $d\psi$ are obtained
algebraically from $\nabla_\varphi \varphi$ one can deduce the following
characterisation of torsion-free \gtstr s.

\begin{theorem}[{\cite[Lemma $11.5$]{salamon:holonomybook}}]
\label{thm:gray}
A smooth positive 3-form $\varphi$ is torsion-free if and only if
$d\varphi = 0$ and $d^{*}_{\varphi}\varphi = 0$ (or equivalently $d\psi = 0$).
\end{theorem}

\begin{remark*} 
Given a Riemannian manifold whose holonomy is contained in the group \gtwo,
there may be several compatible torsion-free \gtstr s. But if the holonomy
is exactly \gtwo, then the torsion-free \gtstr{} is unique.

From the discussion of Remark \ref{rmk:g2exist}, a Riemannian manifold has
holonomy contained in \gtwo if and only if it admits a parallel spinor for
some spin structure.
Wang \cite{wang89} gives an explicit way to construct
a parallel positive $3$-form from a parallel spinor.
\end{remark*}

\begin{remark}
\label{rmk:joyce}
We call a $\gtwo$-structure defined by a closed positive $3$-form $\varphi$ a
\emph{closed $\gtwo$-structure}.
Joyce \cite[Thm. 11.6.1]{joyce:holonomybook} gave sufficient conditions under
which a closed $\gtwo$-structure with small enough torsion can be perturbed to
a torsion-free $\gtwo$-structure within its cohomology class.
\end{remark}

\subsubsection*{\sunitaryn-structures and Calabi--Yau manifolds}
Let $M$ be a real $2n$-dimensional manifold with an \sunitaryn-structure.
Then $M$ is equipped with an almost complex structure $I$, a real non-degenerate
2-form $\omega$ equivalent to a hermitian metric $g$, and an $(n,0)$-form
$\Omega$ of constant norm $2^n$.

If $d\Omega = 0$ then the complex structure is integrable, and $\Omega$ is
holomorphic. In particular, the canonical bundle of $M$ is trivial,
so $c_1(M) = 0 \in H^2(M;\bbz)$.
If also $d\omega=0$, then $M$ is a Kähler manifold. In particular
$\nabla\omega = 0$, so $\Hol(g) \subseteq \unitaryn$.
The fact that $\Omega$ is holomorphic of constant norm forces that also
$\nabla\Omega = 0$, so actually the holonomy must reduce further to
$\Hol(g) \subseteq \sunitaryn$.

\begin{definition}
We call an \sunitaryn-structure \emph{torsion-free} or a \emph{Calabi--Yau
structure} if $\nabla\Omega = \nabla\omega = 0$ with respect to the induced
metric. We call $M^{2n}$ equipped with a torsion-free \sunitaryn-structure
$(\Omega, \omega)$ and its associated metric a \emph{Calabi--Yau manifold}.
We say that $(M^{2n},\Omega,\omega)$ is a manifold \textit{with holonomy
\sunitaryn} or \textit{has holonomy \sunitaryn} if its holonomy is
exactly \sunitaryn.
\end{definition}

\begin{remark}
Yau's proof \cite{yau:CY} of the Calabi conjecture shows that any compact
Kähler manifold $M$ with $c_1(M) = 0 \in H^2(M; \bbr)$ admits Ricci-flat Kähler
metrics. Ricci-flat Kähler manifolds are also often referred to as Calabi--Yau
manifolds, which is not quite equivalent to our definition: the vanishing of
the Ricci curvature implies that the canonical bundle is flat so that the
restricted holonomy (\ie the group generated by parallel transport around
contractible closed curves in $M$, or equivalently the identity component of
$\Hol(g)$) is contained in \sunitaryn, but if $M$ is not simply connected then
there need not be any global holomorphic section.
\end{remark}

Now let $M^6$ be a manifold with an \sunitary3-structure $(g,I,\omega,\Omega)$.
Then the product manifold $\Sph^1 \times M$ has a natural product \gtstr.
The pointwise model \eqref{eq:cylg2form} shows that in terms of the forms
the \gtstr{} is given by
\begin{equation}
\label{eq:3form_vs_CY}
\varphi = d\anglex \wedge \omega + \Real \Omega,
\end{equation}
where $\theta$ is the natural variable on $\Sph^1$. The induced metric is the
product metric, and for any $v\in TM$, $\contra{\theta} \times v=Iv$.

\begin{lemma}
\label{L:G2_vs_CY}
If $(M^6,g,I,\omega,\Omega)$ is a Calabi-Yau 3-fold then the product manifold
$\Sph^1\times M$ with the above \gtstr{} is a \gtmfd.
\end{lemma}

\noindent
Observe that $\Sph^1\times M$ is not a manifold with holonomy \gtwo: its
holonomy equals $\Hol(M) \subseteq \sunitary3 \subset \gtwo$.

\subsubsection*{Hyper-Kähler $K3$ surfaces}
Recall that a \textit{K3 surface} is a smooth compact complex surface $(S,I)$ which is
simply connected and whose canonical bundle is holomorphically trivial,
\textit{i.e.} $\pi_1(S)=0$ and $K_S \simeq \mathcal{O}_S$.
By definition, $S$ has a non-vanishing holomorphic 2-form $\Omega$.
Siu \cite{Siu:K3_kaehler} proved that any K3 surface admits Kähler metrics,
and by Yau's solution to the Calabi conjecture there exists a unique Ricci-flat
Kähler metric $\omega$ in every K\"ahler class and thus Calabi--Yau structures
$(\omega, \Omega)$. \label{page:k3}
The pointwise considerations on p. \pageref{sss:su2} show that a
manifold with holonomy \sunitary2 = \Sp1 has an $\Sph^2$ of integrable complex
structures. A Calabi--Yau structure $(\omega, \Omega)$ compatible with the
metric corresponds to a choice of ordered oriented orthonormal triple $I, J, K$ in
this $\Sph^2$, \ie complex structures satisfying the usual quaternionic relations.
The structure, including the metric, can be recovered from the associated Kähler forms
$\omega^I, \omega^J, \omega^K$ via 
\[
\omega = \omega^I, \qquad 
\Omega = \omega^J +i\,\omega^K.
\] 
We call a K3 surface $S$ with the
structure $(\omega^I, \omega^J, \omega^K)$ a \emph{\hk K3 surface}.

Any two K3 surfaces are related by complex deformation. In particular,
there is up to diffeomorphism a unique K3 surface $S$.
$b_2(S) = 22$, and we will often refer to $H^2(S;\bbz)$ with its intersection
form as the \emph{K3 lattice} $L$. It is the unique even unimodular lattice 
of signature $(3,19)$, \ie
\begin{equation}
\label{eq:k3lattice}
L = 2E_8(-1) \perp 3\,U ,
\end{equation}
where $E_{8}$ denotes the unique even unimodular positive definite lattice of rank $8$ 
and $U$ the standard hyperbolic lattice.
We denote by $\orth{L}$ the group of isometries of the K3 lattice $L$.
A \emph{marking} of a complex K3 surface $(S, I)$ is an isometry $L \cong H^2(S; \ZZ)$.

%% file: sec4.tex
In this section we describe the main steps of our construction of compact \gtmfd s.
Starting from suitable algebraic varieties we first construct asymptotically
cylindrical Calabi-Yau $3$-folds. Given a suitably compatible pair of such manifolds
we then form a ``twisted connected sum'' $7$-manifold by gluing. The procedure is
essentially the same as used by Kovalev \cite{kovalev:connectsums}, but as we
will describe we change the algebraic starting point to use
\emph{semi-Fano 3-folds} rather than Fano 3-folds. The issue of how to satisfy
the compatibility condition between the ACyl Calabi-Yau manifolds is discussed
in detail in \S \ref{sec:k3}.
Throughout this section all homology and cohomology groups
are over $\ZZ$ unless explicitly stated otherwise.

\subsection{Asymptotically cylindrical Calabi--Yau $3$-folds}
We begin with a review of the definition of asymptotically cylindrical
Calabi-Yau 3-folds and an analytic existence result; the latter reduces the
analytic problem of finding asymptotically cylindrical Calabi-Yau 3-folds to a
problem purely in complex projective geometry.

\begin{definition}
Let $(S^4,I_S,g_S,\omega_S,\Omega_S)$ be a \hk K3 surface. 
We call the complex $3$-fold $V_\cyl:= \bbrp \times \Sph^1 \times S$ endowed with the $\bbrp$-translation invariant 
Calabi-Yau structure 
\begin{equation}
\label{eq:cy3cyl}
\begin{aligned}
I_\cyl &:= I_{\C} + I_S, \\
g_\cyl &:= dt^2 + d\anglen^2 + g_S, \\ 
\omega_\cyl &:= dt \wedge d\anglen + \omega_S, \\
\Omega_\cyl &:= (d\anglen - i dt) \wedge \Omega_S ,
\end{aligned}
\end{equation}
(where $t$ and $\anglen$ denote the standard variables on $\R^+$ and $\Sph^1$)
a \emph{Calabi-Yau cylinder}.
The phase in the expression for $\Omega_\cyl$ is unimportant but has been
chosen to put \eqref{eq:hkcyl} in a convenient form. 
\end{definition}

\begin{definition} \label{def:ACCY3}
Let $(V,g,I,\omega,\Omega)$ be a complete Calabi-Yau $3$-fold. We say that $V$
is an \textit{asymptotically cylindrical} (or \emph{ACyl} for short)
Calabi-Yau $3$-fold if there exist (i) a compact set $K\subset V$, (ii) a
Calabi--Yau cylinder $V_\cyl$ and (iii) a diffeomorphism
$\eta: V_\cyl \rightarrow V {\setminus} K$ such that for
all $k\geq 0$, for some $\lambda>0$ and as $t\rightarrow \cyl$,
$$\eta^*\omega-\omega_\cyl=d\varrho,\ \mbox{ for some $\varrho$ such that }
|\nabla^k\varrho|=O(e^{-\lambda t})$$
$$\eta^*\Omega-\Omega_\cyl=d\varsigma,\mbox{ for some $\varsigma$ such that }
|\nabla^k\varsigma|=O(e^{-\lambda t})$$
where $\nabla$ and $|\cdot|$ are defined using the metric $g_\cyl$ on
$V_\cyl$. We will refer to $V_\cyl= \R^+ \times \Sph^1 \times S$ as the
\textit{asymptotic end} of~$V$ and to the \hk K3 surface
$(S,I_S,g_S,\omega_S,\Omega_S)$ as the \emph{asymptotic K3 surface} of $V$.
\end{definition}

\begin{remark*}
Our definition asks that $\eta^*\omega$ be cohomologous to $\omega_\cyl$ on
the asymptotic end of $V$. However, as long as
$|\eta^*\omega-\omega_\cyl| \rightarrow 0$, this is automatic. The main point
of the definition is thus to impose the existence of specific $\varrho$ and
$\varsigma$ with the stated rate of decay.

Since the complex structures on both $\R^+\times\Sph^1\times S$ and $V$ are
determined by the corresponding complex volume forms, similar estimates
automatically hold for $|\nabla^k(\eta^*I-I_\cyl)|$. The same is true for the
metrics. 
\end{remark*}

\begin{remark*}
We could consider a more general definition of an ACyl Calabi-Yau $3$-fold
in which the cross-section of the asymptotic cylinder is not a priori assumed to split as a product $\Sph^1 \times S$. 
Such ACyl Calabi-Yau $3$-folds do exist, but we are not yet able to use them to 
construct compact \gtwo-manifolds. See \cite{hhn} for further discussion of this and other related issues.
\end{remark*}

Our examples of ACyl Calabi-Yau 3-folds arise by application of the
following ACyl version of the Calabi-Yau theorem sharpening 
an earlier result of Tiau-Yau \cite[Thm 5.2]{tian:yau}.
The statement is taken from
\cite[Theorem 2.6]{chnp1}. For details of the proof see \cite{hhn}.

\begin{theorem}
\label{thm:acyl_limits}
Let $Z$ be a closed Kähler 3-fold with a morphism $f : Z \to \bbp^1$, with
a reduced smooth K3 fibre $S$ that is an anticanonical divisor, and
let $V = Z \setminus S$. If $\Omega_S$ is a non-vanishing holomorphic
2-form on $S$, $\omega_S$ a Ricci-flat Kähler metric
satisfying the normalisation condition \eqref{eq:sun:vol}, and
$[\omega_S] \in H^{1,1}(S)$ is the restriction of a Kähler class on $Z$,
then there is an ACyl Calabi-Yau structure $(\omega, \Omega)$ on $V$ whose
asymptotic limit on $\bbrp \times \Sph^1 \times S$ is the product structure
\eqref{eq:cy3cyl}.
\end{theorem}

\begin{remark*}
Arguments similar to Lemma \ref{lem:zpi1} below show that the hypotheses of Theorem \ref{thm:acyl_limits} 
imply $H_1(Z)$ finite and $H^{2,0}(Z) = 0$, so $Z$ must be projective.
\end{remark*}

In the statement above we use the fact that the fibration structure of $Z$ implies that $V:=Z \setminus S$ has an obvious
topological end $\bbrp \times \Sph^1 \times S$.
We call $(Z,S)$ a \emph{building block} if it satisfies some additional
topological conditions. These assumptions will simplify the calculation of the topological
invariants of $V$ in \S \ref{sec:top}.

\begin{definition}
  \label{dfng:BLOCK}
A \emph{building block} is a nonsingular algebraic \mbox{3-fold} $Z$ together
with a projective morphism $f\colon Z\to \PP^1$ satisfying the following
assumptions: 
\begin{enumerate}[leftmargin=*]
\item the anticanonical class $-K_Z\in H^2(Z)$ is
  primitive.
\item $S=f^\star (\infty)$ is a nonsingular K3 surface and $S\sim -K_Z$. 
\end{enumerate}
Identify $H^2(S)$ with the K3 lattice $L$ \eqref{eq:k3lattice}
(\ie choose a marking for $S$), and let $N$ denote the image of
$H^2(Z) \to H^2(S)$.
\begin{enumerate}[resume]
\item The inclusion $N\hookrightarrow L$ is primitive, that is,
  $L/N$ is torsion-free.
\item The group $H^3(Z)$---and thus also $H^4(Z)$---is torsion-free.
\end{enumerate}
\end{definition}
\begin{lemma}
\label{lem:zpi1}
\label{lem:zh20}
If $Z$ is a building block then
\hfill
\begin{enumerate}
\item $\pi_1(Z) = (0)$. In particular, $H^*(Z)$ and $H_*(Z)$ are torsion-free.
\item $H^{2,0}(Z) = 0$, so $N \subseteq \Pic S$.
\end{enumerate}
\end{lemma}

\begin{proof}
(i) is \cite[Lemma 5.2]{chnp1}. For (ii), Serre duality
implies $H^{2,0}(Z) \cong H^1(K_Z)^*$, which vanishes by the long exact
sequence of
$0 \to K_Z \to \oo_Z \to \oo_S \to 0$ together with the fact that
$H^1(\oo_Z) \cong H^{1,0}(Z) = 0$.
\end{proof}

\begin{remark*}
$N \subset L$ inherits the structure of a lattice from the K3 lattice $L$. 
Because of \ref{lem:zh20}(ii) we call $N$ the \emph{polarising lattice} of the building block $Z$. 
The lattice $N$ plays a key role in this paper as we explain shortly. 
\end{remark*}

Further topological properties of building blocks are recalled in Section \ref{sec:top}.
For now it suffices to remark that $V=Z\setminus S$ is always simply-connected, 
so that any ACyl Calabi-Yau metric on $V$ has holonomy exactly $\sunitary3$. 
Most of the building blocks we use in this paper arise from \emph{semi-Fano}
3-folds, as we discuss below in Proposition \ref{prop:block_from_sf}. We say
such ACyl Calabi-Yau $3$-folds are of \emph{semi-Fano type}; see Definition
\ref{d:accy:sf} for a precise definition.

Examples of ACyl Calabi-Yau 3-folds have been constructed previously by
similar methods, using building blocks obtained from genuine Fanos by 
Kovalev \cite{kovalev:connectsums} or from K3s with non-symplectic involution 
by Kovalev-Lee \cite{kovalev:lee} (see Remark \ref{rmk:kl-def}).
We will call these ACyl Calabi-Yau $3$-folds of \emph{Fano type} and
\emph{non-symplectic type} respectively. 
While there are $105$ deformation families of smooth Fano $3$-folds and $75$
deformation classes of K3 surfaces with non-symplectic involution, 
deformation families of semi-Fano $3$-folds are much more plentiful and 
therefore so are ACyl Calabi-Yau $3$-folds of semi-Fano type.

\subsection{The gluing procedure}
We can now outline Kovalev's construction of compact \gtwo-manifolds by 
combining a pair of \emph{compatible} asymptotically cylindrical Calabi-Yau $3$-folds.
We call this the \emph{twisted connected sum} construction of compact \gtwo-manifolds 
and refer to the resulting \gtwo-manifolds as twisted connected sums. 
We emphasise at the outset that finding compatible pairs of asymptotically cylindrical Calabi-Yau $3$-folds
is perhaps the most involved part of the whole construction.

Let $V_\pm$ be two asymptotically cylindrical Calabi-Yau 3-folds with
structures $(g_\pm,I_\pm,\omega_\pm,\Omega_\pm)$. Then as in \eqref{eq:cy3cyl}
the asymptotic end of $V_\pm$ is of the form
$V_{\infty,\pm} = \bbrp \times \Sph^1 \times S_\pm$ where $S_\pm$ is the
asymptotic \hk K3 surface of $V_\pm$.
Using maps $\eta_\pm$ as in Definition \ref{def:ACCY3} to identify the ends
$V_{\cyl,\pm}$ with $\R^+\times\Sph^1\times S_\pm$, on each end we can write
\begin{gather*}
\omega_\pm = \omega_{\cyl,\pm} + d\varrho_\pm,\\
\Omega_\pm = \Omega_{\cyl,\pm} + d\varsigma_\pm.
\end{gather*}

Let $\rho=\rho(s):\R\rightarrow [0,1]$ denote a smooth function satisfying
$\rho(s) \equiv 0$ for $s \leq 0$ and $\rho(s) \equiv 1$ for $s\geq 1$.
For fixed $T \gg 0$, consider the same manifolds $V_\pm$ endowed with forms 
$\omega_{T,\pm}$, $\Omega_{T,\pm}$
obtained by the following perturbation on the ends:
\begin{subequations}
\label{eq:omega:T}
\begin{align}
\omega_{T,\pm} & := \omega_\pm-d(\rho(t-T+1)\varrho_\pm), \\
\Omega_{T,\pm} & := \Omega_\pm-d(\rho(t-T+1)\varsigma_\pm) . 
\end{align}
\end{subequations}
Both forms are closed and in the interval $t\in [T-1,T]$ they interpolate
between the ACyl $\sunitary{3}$-structure ($\omega_{\pm},\Omega_{\pm}$) on
$V_\pm$ and the product $\sunitary{3}$-structure
$(\omega_{\cyl,\pm},\Omega_{\cyl,\pm})$ on the ends $V_{\cyl,\pm}$.
The $C^k$ norms of $\omega_{T,\pm} - \omega_\pm$ and
$\Omega_{T,\pm} - \Omega_\pm$ are $O(e^{-\lambda T})$.

Now consider the product (asymptotically cylindrical) $7$-manifolds
$M_\pm=\Sph^1\times V_\pm$. We let $\anglex$ denote
the standard variable on the new $\Sph^1$ factor, reserving the notation
$\anglen$ for the copy of $\Sph^1$ contained in the ends of $V_\pm$. We endow
$\Sph^1\times V_\pm$ with the $3$-forms (\cf \eqref{eq:3form_vs_CY})
\[ \varphi_{T,\pm} := d\anglex\wedge\omega_{T,\pm} + \Real \Omega_{T,\pm} . \]
For $T$ large the forms $\varphi_{T,\pm}$ are small perturbations of the
\gtstr s on $\Sph^1\times V_\pm$ defined by the original Calabi-Yau structures
on $V_\pm$ as in \eqref{eq:3form_vs_CY}, so they are again \gtstr s. 

To form the twisted connected sum of $M_+$ and $M_-$ 
we require a certain compatibility condition of the pair of asymptotic %
K3 surfaces~$S_\pm$ of $V_\pm$.
The asymptotic limit of $V_\pm$ defines a Calabi-Yau structure $(\omega_\pm,
\Omega_\pm)$ on $S_\pm$ and a preferred complex structure
$I_\pm$ on $S_\pm$. However, recall from p.\pageref{page:k3} that $S_\pm$
admits an $\Sph^2$ of complex structures, and that setting
\begin{equation}
\label{eq:hkstr}
\omega_\pm = \omega^I_\pm, \quad \Omega_\pm = \omega^J_\pm + i \,\omega^K_\pm,
\end{equation}
defines a \hk structure $(\omega^I_\pm, \omega^J_\pm, \omega^K_\pm)$.
These are Kähler forms with respect to complex structures $I_\pm$, $J_\pm$ and
$K_\pm$ respectively; the special status of $I_\pm$ is reflected by the
ordering.
The compatibility condition we need for our pair of ACyl Calabi-Yau $3$-folds $V_\pm$ 
is the existence of the following special 
type of map between their asymptotic \hk K3 surfaces.
\begin{definition}
\label{def:hkrot}
  Consider two \hk K3 surfaces $S_\pm$. A map
  $\hkr : S_+ \rightarrow S_-$ is a \textit{\hk rotation} if $\hkr^*g_- = g_+$,
  $\hkr^*I_- = J_+$ and $\hkr^*J_- = I_+$; the \hk relationship $IJ=K$ then
  implies that $\hkr^*K_- = -K_+$.
  Equivalently, $\hkr^*\omega^I_- = \omega^J_+, \,\hkr^*\omega^J_- = \omega^I_+$
  and $\hkr^*\omega^K_- = -\omega^K_+$.
\end{definition}

As soon as we are given a pair of ACyl Calabi-Yau $3$-folds $V_{\pm}$ for which we can establish the existence of 
a \hk rotation $\hkr$   between the asymptotic \hk K3 surfaces $S_{\pm}$ then 
we can glue the two $7$-manifolds $M_\pm = \Sph^1\times V_\pm$
together by their ends, as follows. On the region defined by $t\in (T,T+1)$
consider the diffeomorphism
\begin{equation}
\label{eq:gluemap}
\begin{aligned}
F : \; \Sph^1 \times V_{\cyl,+} \cong
\Sph^1 \times \R^+ \times \Sph^1 \times S_+ \; & \longrightarrow \;
\Sph^1 \times \R^+ \times \Sph^1\times S_- \cong \Sph^1 \times V_{\cyl,-}, \\
(\anglex, \, t, \, \anglen, \, x) \; &
\longmapsto (\anglen, \, T+1-t, \, \anglex, \, \hkr(x)) .
\end{aligned}
\end{equation}
Notice that by \eqref{eq:omega:T} we are working on regions where
$(\Omega_{T,\pm},\omega_{T,\pm})$ are the standard product structures
\eqref{eq:cy3cyl}.
Thus, using \eqref{eq:hkstr}, the \gtstr s on these regions can be written
\begin{equation}
\label{eq:hkcyl}
\begin{aligned}
\varphi_{T,\pm} &= d\anglex\wedge\omega_{\cyl,\pm} + \Real \Omega_{\cyl,\pm} \\
& = d\anglex\wedge dt \wedge d\anglen + d\anglex\wedge\omega^I_\pm +
d\anglen\wedge\omega^J_\pm + dt \wedge \omega^K_\pm .
\end{aligned}
\end{equation}
The compatibility condition for $\hkr$ given in \ref{def:hkrot} implies immediately that
$F^*\varphi_{T,-} = \varphi_{T,+}$. Now truncate each $\Sph^1\times V_\pm$ at
$t = T+1$ to form a pair of compact manifolds $M_\pm(T)$ with boundaries
$\Sph^1 \times \Sph^1 \times S_\pm$. Using $F$ we can glue these manifolds
together at the boundary to form a `twisted connected sum'
$M_\hkr = M_+(T) \cup_F M_-(T)$.
This is a smooth compact $7$-manifold (independent of $T$ up to diffeomorphism
but depending on the choice of the \hk rotation $\hkr$), which admits a closed
\gtwo-structure $\varphi_{T,\hkr}$ defined by setting its restriction
to $M_\pm(T)$ to equal $\varphi_{T,\pm}$. 
With respect to the metric of $\varphi_{T,\hkr}$, $M_\hkr$ contains an
approximately cylindrical neck of length $\sim 2T$.
The torsion of $\varphi_{T,\hkr}$ (which is measured by $d^*\varphi_{T,\hkr}$
according to Theorem \ref{thm:gray}) is $O(e^{-\lambda T})$.
Kovalev \cite[Theorem 5.34]{kovalev:connectsums} uses this to prove that for
$T$ sufficiently large there are nearby torsion-free \gtstr s (one could also
apply more general results of Joyce, see Remark \ref{rmk:joyce}).

\begin{theorem}
\label{thm:g2glue}
Let $(V_\pm,\omega_\pm,\Omega_\pm)$ be two asymptotically cylindrical
Calabi-Yau 3-folds whose asymptotic ends are of the form
$\bbrp \times \Sph^1 \times S_\pm$ for a pair of \hk K3 surfaces $S_\pm$,
and suppose there exists  a \hk rotation $\hkr : S_+ \to S_-$.
Define closed \gtstr s $\varphi_{T,\hkr}$ on the twisted connected sum
$M_\hkr$ as above. For sufficiently large $T$ there is a torsion-free
perturbation of $\varphi_{T,\hkr}$ within its cohomology class.
\end{theorem}

Whenever the $V_\pm$ in the theorem have holonomy $\sunitary3$,
\cite[Proposition 2.15]{hhn} implies that their fundamental groups are finite
and generated by the $\Sph^1$ factors in the cylindrical ends. The Van Kampen
theorem implies $\pi_1(M_\hkr) \cong \pi_1(V_+) \times \pi_1(V_-)$ is finite,
so the holonomy of the metric defined by the torsion-free \gtstr{} on $M_\hkr$
is exactly \gtwo by Proposition \ref{prop:G2_topology}. 
Any ACyl Calabi-Yau $3$-fold $V$ of semi-Fano or Fano type is
simply connected and therefore twisted connected sums using them are also simply connected.
The 74 deformation families of ACyl Calabi-Yau $3$-folds of non-symplectic
type are also simply connected \cite[Lem 4.2]{kovalev:lee}.

\subsection{ACyl Calabi-Yau $3$-folds from semi-Fano $3$-folds}
It remains to explain how we can construct ACyl Calabi-Yau $3$-folds suited 
to the twisted connected sum construction from \emph{semi-Fano} $3$-folds. 
To this end we now recall from \cite[\S 4]{chnp1} the definition and a few of the basic properties of 
semi-Fano $3$-folds; we refer the reader to \cite{chnp1} for proofs of the facts recalled here 
and for a much more comprehensive treatment 
of semi-Fano $3$-folds, including relevant algebro-geometric background.

A semi-Fano $3$-fold is a particular type of \emph{weak Fano} $3$-fold, a generalisation 
of a Fano $3$-fold in which the positivity of $-{K_{Y}}$ is replaced with a sufficiently strong 
notion of semi-positivity.
\begin{definition}
\label{dfng:weak_fano}
A \emph{weak Fano \mbox{$3$-fold}} is a nonsingular projective complex
\mbox{$3$-fold} $Y$ such that the
anticanonical sheaf $-K_Y$ is a nef and big line bundle, 
\ie $-K_{Y} \cdot C \ge 0$ for any compact algebraic curve $C \subset Y$ 
and $(-K_{Y})^{3}>0$. 
For any weak Fano $3$-fold $Y$ the integer $(-K_{Y})^{3}$ is an even integer which we write $2g-2$; 
$(-K_{Y})^{3}=2g-2$ is called the \emph{anticanonical degree} of $Y$ 
and $g$ the \emph{genus} of $Y$.

The \emph{index} of a weak Fano $3$-fold $Y$ is the integer $r=\gdiv{c_1(Y)}$, 
\ie the greatest divisor of $c_{1}(Y) \in H^{2}(Y)$.
\end{definition}
From the classification of Fano $3$-folds we know that there are exactly $105$ deformation families 
of smooth Fano $3$-folds. For weak Fano $3$-folds we still know that there are only finitely many deformation 
families. However,  there are many more deformation 
families of weak Fano $3$-folds as explained in \cite{chnp1} 
and a classification of all weak Fano $3$-folds looks a long way off.

If $Y$ is a weak Fano $3$-fold then for $n$ sufficiently large the linear system $\abs{-nK_{Y}}$ is basepoint-free. 
It follows that 
\[
R(Y,-K_Y):=\oplus_{n\geq 0} H^0(Y;-nK_Y)
\] 
is a finitely generated ring called the \emph{anticanonical ring} of $Y$. 
We call the birational morphism $\varphi\colon Y \to X:=\Proj R(Y,-K_Y)$ attached
to $|{-}K_Y|$ the \emph{anticanonical morphism} of $Y$ and $X$ the
\emph{anticanonical model} of~$Y$.
$X$ is a singular Fano $3$-fold with mild (at worst Gorenstein canonical)
singularities and $\varphi \colon Y \to X$ is a crepant resolution of $X$, \ie $\varphi^*K_X=K_Y$.

Conversely, if $Y$ is a projective crepant resolution
$\varphi\colon Y \to X$ of a Fano $3$-fold $X$ with Gorenstein canonical
singularities then $Y$ is a weak Fano $3$-fold whose anticanonical model
is~$X$. In other words, one way to exhibit weak Fano $3$-folds is to find
projective crepant resolutions of Gorenstein canonical Fano $3$-folds.
For instance a sufficiently general quartic $X \subset \CP^{4}$ that contains a
projective plane $\Pi$ is a suitable singular Fano $3$-fold; 
$X$ has exactly $9$ singular points, all ordinary nodes contained in $\Pi$
and admits a projective crepant (in fact small) resolution $\varphi\colon Y
\to X$, obtained by blowing up the plane $\Pi$: $Y$ is a weak Fano $3$-fold which we use 
later in the paper---see Example \ref{exg:quartic_w_plane}.

A key fact about any smooth weak Fano 3-fold $Y$ is that a general anticanonical divisor
$S \in \abs{-K_{Y}}$ is a nonsingular K3 surface.  From now on we make the following extra assumption about all the 
weak Fano $3$-folds we will use in this paper.
\smallskip

\noindent
\textbf{Assumption:}
the linear system $\abs{-K_Y}$ contains 
two nonsingular members $S_0, S_\infty$ intersecting transversally.

\smallskip
\noindent
The few weak Fano $3$-folds for which this assumption is not satisfied are classified: 
see \cite[\S 4]{chnp1} and references therein for further details.

Under the assumption above a generic pencil in $\abs{-K_Y}$ has a base locus which is a smooth curve 
(of genus $g=g(Y)$).
Hence from $Y$ we can construct a smooth projective $3$-fold $Z$ fibred over $\CP^{1}$ by 
(generically) smooth anticanonical K3 fibres by blowing up the base locus of a generic
pencil $\abs{S_{0},S_{\infty}} \subset \abs{-K_{Y}} $:
see \cite[Proposition 4.24]{chnp1}.
Therefore by Theorem \ref{thm:acyl_limits} we can construct ACyl Calabi-Yau structures on $V:=Z\setminus S$. 

However, for the purposes of this paper  it is convenient to restrict to ACyl Calabi-Yau structures obtained from 
a subclass of weak Fano $3$-folds which we call \emph{semi-Fano} $3$-folds. 
There is still a large number of deformation families of semi-Fano $3$-folds.
\begin{definition}
\label{dg:semifano}
Let $Y$ be a weak Fano $3$-fold and $\varphi\colon Y \ra
  X$ its anticanonical morphism. If $\varphi$ is \emph{semi-small}, we
  call $Y$ a \emph{semi-Fano} $3$-fold, \ie the anticanonical morphism 
  $\varphi \colon Y \to X$ can contract divisors to curves, or curves to points, but not divisors to points.
\end{definition}
\noindent
From any semi-Fano $3$-fold $Y$ satisfying our assumption above 
we can obtain a building block.
\begin{prop}[{\cite[Props. 4.24 \& 5.7]{chnp1}}] 
\label{prop:block_from_sf}
Let $Y$ be a semi-Fano 3-fold with $H^{3}(Y)$ torsion-free, $|S_0,S_\infty| \subset |{-}K_Y|$ a generic
pencil with (smooth) base locus~$C$, $S \in |S_0,S_\infty|$
generic, and $Z$ the blow-up of $Y$ at $C$.
Then $S$ is a smooth K3 surface, its proper transform in $Z$ is isomorphic
to $S$, and $(Z,S)$ is a building block in the sense of Definition
\ref{dfng:BLOCK}. Furthermore
\begin{enumerate}
\item \label{it:polar}
the image $N$ of $H^2(Z) \to H^2(S)$ equals that of $H^2(Y) \to H^2(S)$;
\item \label{it:cones}
$\Amp_Y \subseteq \Amp_Z$, where $\Amp_Y$ and $\Amp_Z$ denote the \emph{images}
in $N_\bbr \subseteq H^{1,1}(S)$ of the Kähler cones of $Y$ and $Z$;
\item $H^2(Y) \to H^2(S)$ is injective (and $K = 0$ in \eqref{eq:k_def}).
\end{enumerate}
\end{prop}

\begin{definition}
\label{d:accy:sf}
We will refer to a building block $(Z,S)$ arising from Proposition \ref{prop:block_from_sf} as a 
\emph{building block of semi-Fano type}. By Theorem \ref{thm:acyl_limits} we can obtain ACyl Calabi-Yau structures $(\omega,\Omega)$ on $V:=Z \setminus S$ and we call $(V,\omega,\Omega)$ an ACyl Calabi-Yau $3$-fold of \emph{semi-Fano type}.
\end{definition}

\begin{remark}
\label{rmk:cones}
The significance of \ref{prop:block_from_sf}\ref{it:cones} is that
Theorem \ref{thm:acyl_limits} ensures that exactly %
the classes in $\Amp_Z$ can be represented by the asymptotic limit of
an ACyl Calabi-Yau Kähler form $\omega$ on~$V$.
Note that $\Amp_Y$ and $\Amp_Z$ are typically proper subcones of the Kähler
cone of $S$, even when $Y$ is Fano (\cf Remark \ref{rmk:ampfamily}).
We need to pay attention to this in the
matching argument in \S \ref{sec:k3}.
\end{remark}

Sometimes one can get different building blocks from the same semi-Fano by blowing
up base loci of non-generic anticanonical pencils (\cf Examples
\ref{exg:P3_deg}, \ref{exg:2conics}, \ref{exg:toricV22_b}). In this case extra
work is required both to check that the topological conditions of a building
block are satisfied, and to apply the matching arguments from \S \ref{sec:k3}.
To avoid ambiguity, the term semi-Fano type will always refer to blow-ups
of generic pencils as in Proposition \ref{prop:block_from_sf}, and we will
warn explicitly in the few cases where we use non-generic pencils.

\begin{remark*}
We do not know any example of a semi-Fano $3$-fold $Y$ with torsion in $H^{3}(Y)$, 
but cannot in general prove $H^{3}(Y)$ is torsion-free: see \cite[\S 5]{chnp1} for further remarks in this direction.
This assumption is used to prove that $H^{3}(Z)$ is torsion-free as required in Definition \ref{dfng:BLOCK}(iv).
Note that this condition is only used in order to simplify the calculation of
the full integral cohomology. Dropping it would not affect the
more crucial matching arguments, but in the absence of known examples with
torsion in $H^3(Z)$ we do not concern ourselves with this generality.

If $Z$ is obtained---in the manner of Proposition \ref{prop:block_from_sf}---from a weak Fano $Y$ which is not semi-Fano 
then the natural map $H^{2}(Y) \to H^{2}(S)$ cannot be injective, since
the class of any contracted divisor lies in the kernel.
It is also not clear that the map has to have primitive image; in particular
$Z$ might not be a building block in the sense of Definition \ref{dfng:BLOCK}
because property (iii) could fail. 
\end{remark*}

By varying the choice of semi-Fano $3$-fold $Y$ within its deformation type  
$\sff$,  the choice of generic pencil $\abs{S_0,S_\infty} \subset \abs{-K_Y}$ and the choice of a 
generic $S \in \abs{S_0,S_\infty}$ we can obtain families of ACyl Calabi-Yau structures on the 
same smooth $6$-manifold $V$. Varying the ACyl Calabi-Yau structure on $V$ this way allows us to obtain 
different asymptotic \hk K3 surfaces~$S$.

This observation is crucial when we come to construct pairs of compatible ACyl Calabi-Yau $3$-folds. 
Given a \emph{fixed} pair of ACyl Calabi-Yau $3$-folds $V_\pm$ in general it will not be possible to construct 
any \hk rotation $\hkr$ between the asymptotic K3 surfaces $S_\pm$. However, for ACyl Calabi-Yau $3$-folds 
of semi-Fano type we will prove that in many cases it is possible to deform the pair of ACyl Calabi-Yau 
structures on the $3$-folds $V_\pm$ as above, so that within these deformation families a compatible pair does exist.
To achieve this it is important to understand which \hk K3 surfaces can arise 
as the asymptotic K3 surface of ACyl Calabi-Yau $3$-folds of semi-Fano type.

If $Y$ is a semi-Fano 3-fold then Proposition
\ref{prop:block_from_sf}\ref{it:polar} shows that the polarising lattice $N$ of
a semi-Fano type block obtained from $Y$ is isomorphic to $H^2(Y)$.
Therefore the asymptotic K3 surface $S$ of an ACyl Calabi-Yau $3$-fold of
semi-Fano type obtained from any deformation of $Y$ has a primitive 
sublattice isomorphic to $\Pic{Y}=H^{2}(Y)$ in $\Pic{S}$. 

So $\rk{\Pic{S}} \ge b^{2}(Y)$, whereas a generic (projective) K3 surface $S$
has $\rk{\Pic{S}}=1$.  In other words, 
the larger $b^2(Y)$ is the more special the K3 surfaces that can arise 
as asymptotic K3 surfaces obtained from a fixed deformation type $\sff$ of semi-Fanos via Proposition \ref{prop:block_from_sf}.
The moduli theory of K3 surfaces whose Picard group contains a given sublattice $N$---so-called
 \emph{lattice polarised K3 surfaces}---is well-understood and was reviewed in our previous paper 
\cite{chnp1}. We will need to know that the generic lattice polarised K3 surface of a given type occurs
as the asymptotic \hk K3 surface of some ACyl Calabi-Yau structure obtained from 
a given deformation type $\sff$ of semi-Fano $3$-folds via Proposition \ref{prop:block_from_sf}.
The proof of this fact relies on semi-Fano $3$-folds enjoying a better deformation theory 
than general weak Fano $3$-folds: see \cite[\S 6]{chnp1}. 
The improved deformation theory uses the stronger cohomology vanishing theorems available for semi-Fano $3$-folds.

We will explain the above more precisely when we explain how to construct
compatible pairs of ACyl Calabi-Yau $3$-folds $V_{\pm}$ by orthogonal matching in Section \ref{sec:k3}. 

\begin{remark}
\label{rmk:kl-def}
Another kind of building blocks was defined by Kovalev and Lee
\cite{kovalev:lee} from K3 surfaces $S$ with non-symplectic involution,
\ie with an involution $\tau$ acting as $-1$ on $H^{2,0}(S)$.
In Nikulin's classification, %
1 of the 75 families of non-symplectic involutions
acts freely. In the other 74 cases, resolving the singular set of
$(S \times \CP^1)/(\tau,-1)$ by blow-up defines a simply-connected building
block $Z$, which we say is of \emph{non-symplectic type}.

The polarising lattice of $Z$ is the $\tau$-invariant part $N$ of $H^2(S)$.
$N$ characterises $\tau$ in the sense that a generic
$N$-polarised K3 admits an equivalent involution. The matching
arguments we will use for families of semi-Fano blocks can therefore also
be used for families of non-symplectic blocks.
The image $\Amp_Z \subset H^{1,1}(S)$ of the Kähler cone of $Z$
is the full Kähler cone of $S$ \cite[Prop.~4.1]{kovalev:lee}.
$\rk K$ (as defined in \eqref{eq:k_def}) is twice the number of fixed
components of $\tau$, so at least 2 \cite[(4.3)]{kovalev:lee}.
\end{remark}

\subsection{Semi-Fano $3$-folds from nodal Fano $3$-folds}
While our general theory will allow us to find compatible pairs of ACyl Calabi-Yau $3$-folds 
of semi-Fano type, most of the specific semi-Fano $3$-folds we use to build concrete \gtwo-manifolds in this paper
satisfy additional properties which we now describe.

An important special class of semi-Fano $3$-folds are those for which 
the anticanonical morphism $\varphi \colon Y \to X$ is not just semi-small but \emph{small}, \ie 
contracts only finitely many curves.
A special case---and for this paper by far the most important case---is when $X$ is a \emph{nodal Fano} $3$-fold, 
\ie $X$ has only finitely many singular points 
each (locally analytically) equivalent to the $3$-fold ordinary double point. 
In this case any small resolution $Y$ of $X$  replaces each node with a smooth rational curve $\CP^{1}$ 
with normal bundle $\oo(-1) \oplus \oo(-1)$. Most of the semi-Fano $3$-folds
$Y$ we consider in detail in this paper will arise from projective small
resolutions of nodal Fano $3$-folds.
\begin{remark*}
When the semi-Fano $Y$ arises as a projective small resolution $\varphi \colon Y \to X$ 
of a nodal Fano $X$ then each exceptional curve $C$ of $\varphi$
gives rise to a \emph{compact} rigid curve $C$ in the ACyl Calabi-Yau $3$-fold $V=Z\setminus S$ constructed 
from $Y$ using Proposition \ref{prop:block_from_sf}.
These compact rigid curves in $V$ will allow us to construct compact rigid associative $3$-folds 
in twisted connected sum \gtwo-manifolds built using $V$.
\end{remark*}
Whenever the anticanonical morphism $\varphi$ of a semi-Fano $3$-fold $Y$ is small we have the following additional features:
\begin{enumerate}
\item
The anticanonical model $X$ is a Fano $3$-fold with Gorenstein \emph{terminal} (and therefore isolated) singularities.
\item
The small projective morphism $\varphi \colon Y \to X$ can be \emph{flopped}.  
Flopping yields other smooth semi-Fano $3$-folds $Y'$ with the same anticanonical model $X$ 
and whose anticanonical morphism $\varphi' \colon Y' \to X$ is also small.
\item
$X$ is smoothable by a flat deformation and hence is a degeneration of a nonsingular Fano $3$-fold $X_{t}$. 
In particular, the Picard ranks and the Fano indices of $X$ and $X_{t}$ are equal.
\end{enumerate}

The topologies of the smooth $3$-folds $Y$ and $X_t$ and the singular $3$-fold $X$ are closely related. 
The following is explained in much greater detail in our previous paper \cite{chnp1}. 
In the current paper we will need some of the facts below in our discussion of \emph{\gtwo-transitions} in 
Section \ref{sec:close} but not elsewhere in the paper.

Since $X$ is singular in general it need not satisfy Poincar\'e duality.  One way to define the \emph{defect} of $X$ is 
as the following measure of failure of Poincar\'e duality on $X$, 
\begin{equation}
\label{e:defect}
\sigma(X):=\rk H_4(X)-\rk H^2(X).
\end{equation}
The existence of a \emph{projective} small resolution $\varphi \colon Y \to X$ can be shown to force the defect $\sigma(X)$ 
to be positive. Also for any small resolution $\varphi \colon Y \to X$ we have
\begin{equation}
\label{e:b2:res}
b^2(Y) = b^2(X) + \sigma(X) = b^2(X_t) + \sigma(X).
\end{equation}

In particular, if we start from a smooth Fano $X_t$, degenerate to the singular
Fano $X$ and then resolve to obtain the smooth semi-Fano $3$-fold $Y$ then
necessarily $b^2(Y) > b^2(X_t)$.
For instance, if $Y$ is a small resolution of a 
generic quartic containing a plane $\Pi$ then one can show that $\sigma(X)=1$,  $b^2(X_t)=b^2(X)=1$ and hence $b^2(Y)=2$.
Therefore %
the asymptotic K3 surfaces in ACyl Calabi-Yau $3$-folds
of semi-Fano type $Y$ are more special than those in
ACyl Calabi-Yau $3$-folds of Fano type $X_t$. 
One can interpret this as saying that finding an ACyl Calabi-Yau $3$-fold compatible with 
an ACyl Calabi-Yau $3$-fold of semi-Fano type $Y$ should be harder than 
finding one compatible with an ACyl Calabi-Yau $3$-fold of Fano type~$X_{t}$.

If $X$ is a nodal $3$-fold with $e$ nodes and defect $\sigma$ 
one can show that the third Betti numbers $b^3$ of $Y$, $X$ and $X_t$ are related as follows
\begin{equation}
\label{e:b3:res}
b^3(X) =  b^3(X_t) + \sigma - e, \qquad b^3(Y) = b^3(X_t) - 2e + 2 \sigma.
\end{equation}
Since one always has $\sigma \le e$ the second equation shows that $b^{3}(Y)\le b^{3}(X_{t})$.

To summarise, in passing from the smooth Fano $X_{t}$ to the smooth semi-Fano $Y$ 
$b^{2}$ must increase whereas $b^{3}$ typically decreases. We will discuss the significance of these facts 
for \mbox{\gtwo-manifolds} arising as twisted connected sums of ACyl Calabi-Yau $3$-folds of semi-Fano type 
in Section \ref{sec:close}.

%% file: sec_top.tex
In this section, we collect some tools to compute topological
invariants of \gtmfd s that are obtained by gluing
asymptotically cylindrical Calabi-Yaus. All homology and cohomology
groups in this section are over $\ZZ$ unless explicitly stated
otherwise. 

\subsection{Cohomology of the building blocks}

Here we recall notation and some computations of cohomology
groups from \cite[\S 5]{chnp1}. 
First recall the definition of a building block from \ref{dfng:BLOCK}.
We denoted there by $N$ the image of $H^2(Z) \to H^2(S) = L$.
We regard $N$ as a lattice with the quadratic form inherited from $L$.
In examples, $N$ is almost never unimodular,
so the natural inclusion $N\hookrightarrow N^\ast$ is not an isomorphism.
We write
\[
T = N^\perp = \{ l \in L \mid \inner{l, n} = 0 \ \  \forall\; n\in N \} .
\]
($T$ stands for ``transcendental''; in examples, $N$ and $T$ are the
Picard and transcendental lattices of a lattice polarized K3 surface.) Using $N$ primitive and $L$ unimodular we find \mbox{$L/T\simeq N^*$}.

Let $V=Z\setminus S$, and write $H=H^2(V)$.
Since the normal bundle of $S$ in $Z$ is trivial, there is an inclusion
$S \into V$ well-defined up to homotopy. We let
\begin{equation}
\label{eq:k_def}
\rho \colon H \to L \quad
\text{the natural restriction map, and}
\quad
K=\ker (\rho) .
\end{equation}
It follows from (ii) of the following lemma that the image of $\rho$
equals $N$.

\begin{lemma}
  \label{lemg:Z&V}
Let $f\colon Z\to \PP^1$ be a building block. Then:
\begin{enumerate}
\item $\pi_1(V)=(0)$ and $H^1(V)=(0)$;
\item \label{itg:h2}
the class $[S]\in H^2(Z)$ fits in a split exact sequence
\[(0)\to \ZZ\overset{[S]}{\longrightarrow} H^2(Z)\to H^2(V)\to (0),\]
hence $H^2(Z)\simeq\ZZ[S]\oplus H^2(V)$, and the restriction homomorphism $H^2(Z)\to L$ factors through
$\rho\colon H^2(V)=H\to L$;
\item 
there is a split exact sequence
\[
(0) \to H^3(Z) \to H^3(V) \to T \to (0),
\]
hence $H^3(V)\simeq H^3(Z)\oplus T$;
\item \label{itg:h4}
there is a split exact sequence
\[
(0) \to N^\ast \to H^4(Z)\to H^4(V)\to (0),
\]
hence $H^4(Z)\simeq H^4(V)\oplus N^\ast$;
\item
$H^5(V) = (0)$. 
\end{enumerate}
\end{lemma}

\begin{corollary}
\label{corg:V&S^1(S)}
 Let $f\colon Z \to \PP^1$ be a building block. Since the normal %
 bundle of $S$ in $Z$ is trivial, we get a natural
 inclusion of $S\times \bbS^1\subset V$. Denote by $\bfa^0\in
 H^0(\bbS^1)$, $\bfa^1\in H^1(\bbS^1)$ the standard generators. 
The natural restriction homomorphisms:
\[
\beta^m\colon H^m(V) \to H^m(S\times \bbS^1)=\bfa^0H^m(S)\oplus \bfa^1H^{m-1}(S)
\] 
are computed as follows:
\begin{enumerate}
\item $\beta^1 =0$;
\item $\beta^2 \colon H^2(V) \to H^2(S\times \bbS^1)=\bfa^0H^2(S)$ is the
  homomorphism $\rho \colon H \to L$;
\item $\beta^3\colon H^3(V)\to H^3(S\times \bbS^1)=\bfa^1H^2(S)$ is
  the composition $H^3(V) \twoheadrightarrow T \subset L$;
\item \label{itg:4} the natural surjective restriction homomorphism
  $H^4(Z)\to H^4(S)=\ZZ$ factors through
  $\beta^4\colon H^4(V)\to H^4(S\times \bbS^1)=\bfa^0H^4(S)=\ZZ$, and
  there is a split exact sequence:
\[
(0) \to K^\ast \to H^4(V)\overset{\beta^4}{\longrightarrow} H^4(S)\to (0) .
\]
\end{enumerate}
\end{corollary}

Lemma \ref{lemg:Z&V} and Corollary \ref{corg:V&S^1(S)} are closely related
to the long exact sequences for cohomology of $Z$ relative to $S$ and $V$
relative to its boundary $\Sph^1 \times S$, respectively.
\begin{gather}
\label{eq:zrel}
H^k_{cpt}(V) \to H^k(Z) \to H^k(S) \to H^{k+1}_{cpt}(V)\\
\label{eq:vrel}
H^k_{cpt}(V) \to H^k(V) \stackrel{\beta^k}{\to} H^k(\Sph^1 \times S)
\stackrel{\partial}{\to} H^{k+1}_{cpt}(V)
\end{gather}
In particular note that $H^4_{cpt}(V) \into H^4(Z)$.
Also $H^4_{cpt}(V) \cong N^* \oplus K^*$, where the term $N^* \cong L/T$ is
the image of $H^3(\Sph^1 \times S)$ under $\partial$. Its image in $H^4(Z)$ is
precisely the $N^*$ appearing in \ref{lemg:Z&V}\ref{itg:h4}.

\subsection{Cohomology of the 7-manifolds}
\label{sub:g2coh}

We are interested in smooth $7$-manifolds $M$ constructed as follows. Start
with two building blocks $(Z_+, S_+)$, $(Z_-, S_-)$ and a \hk rotation
$\hkr \colon S_+ \to S_-$.
Let $\bbS(S_\pm) = S_\pm \times \bbS_\pm^1 \subset V_\pm$ denote the unit
normal bundles of $S_\pm$ in $Z_\pm$. We glue $M_+ = V_+ \times \bbS_-^1$
with $M_- = V_- \times \bbS_+^1$ identifying the ends via the diffeomorphism
of $\bbS(S_+)\times \bbS_-^1= S_+ \times \TT^2$ with
$\bbS(S_-) \times \bbS_+^1 = S_- \times \TT^2$ that identifies $S_+$ with
$S_-$ by the \hk rotation $\hkr$ and exchanges the two factors of $\TT^2$
(see \eqref{eq:gluemap}). For the purposes of this section $\hkr$ is fixed and,
using this identification, we let $S$ denote $S_+ = S_-$.

We now compute the cohomology groups of $M$ in terms of the cohomology
groups of $Z_\pm$, the restrictions $\rho_\pm \colon H_\pm \to L$, their
kernels $K_\pm$ and their images $N_\pm \subset L$, which are primitive
sublattices by assumption. Our main tool is the
Mayer-Vietoris exact sequence for the decomposition $M = M_+ \cup M_-$ along
the common intersection $W = S \times \bbS_+^1\times \bbS_-^1$:
\begin{equation}
  \label{eq:mayer_vietoris}
H^{m-1}(M_+) \oplus H^{m-1}(M_-) \to H^{m-1}(W) \stackrel{\delta}{\to} H^m (M)
\stackrel{\rho^m}{\to} H^m(M_+)\oplus H^m(M_-) \stackrel{\gamma^m}{\to} H^m (W)
\end{equation}
We write $\gamma^m =
\gamma_+^m \oplus \gamma_-^m \colon H^m(M_+) \oplus H^m(M_-)\to H^m(W)$. 

Lemma \ref{lemg:Z&V} shows that $H^m(M_\pm)$, thus $\mbox{Im}(\rho^m)$, is
torsion-free. Sequence \eqref{eq:mayer_vietoris} thus yields isomorphisms
\begin{equation}\label{eq:M_splits}
H^m(M)\simeq \mbox{Im}(\rho^m)\oplus \ker(\rho^m)\simeq \ker(\gamma^m)\oplus \coker(\gamma^{m-1}).
\end{equation} 
The key task is thus to describe the homomorphisms $\gamma^m$.
\begin{lemma}
\label{lem:mvmaps}
  Let $Z_\pm \to \PP^1$ be building blocks; let $M_\pm$ and $M$ be as above.
  We use the self-explanatory notation:
  \begin{align*}
    H^m(M_+) &= \bfa_-^0 H^m(V_+) \oplus \bfa_-^1H^{m-1}(V_+)\\
    H^m(M_-) &= \bfa_+^0 H^m(V_-) \oplus \bfa_+^1H^{m-1}(V_-)
  \end{align*}
and
\[
H^m(W) = \bfa_+^0 \bfa_-^0H^m(S) \oplus \bfa_+^1 \bfa_-^0 H^{m-1}(S)
\oplus \bfa_+^0 \bfa_-^1 H^{m-1}(S) \oplus \bfa_+^1 \bfa_-^1 H^{m-2}(S).
\]
 The homomorphisms $\gamma^m \colon H^m(M_+) \oplus H^m(M_-) \to H^m(W)$
 that occur in the Mayer-Vietoris sequence are computed as follows:
 \begin{enumerate}
 \item $H^1(M_+) \oplus H^1(M_-)=\bfa_-^1 H^0(V_+)\oplus \bfa_+^1 H^0(V_-)$, \\
	$H^1(W) = \bfa_+^0 \bfa_-^1 H^0(S)\oplus \bfa_+^1 \bfa_-^0 H^0(S)$, and 
\[
\gamma^1=
\begin{pmatrix}
  \mathbf{1} & 0\\0 & \mathbf{1}
\end{pmatrix}
\colon H^0(V_+) \oplus H^0(V_-) \to H^0(S) \oplus H^0(S)
\]
   is the natural isomorphism.
 \item $H^2(M_+)\oplus H^2(M_-)=\bfa_-^0 H_+ \oplus \bfa_+^0H_-$, \\
   $H^2(W)=\bfa_+^0\bfa_-^0 H^2(S)\oplus
   \bfa_+^1\bfa_-^1H^0(S)=L\oplus \ZZ[S]$, and 
\[
\gamma^2=
\begin{pmatrix}
  \rho_+ & \rho_-\\0 & 0
\end{pmatrix}
\colon H_+\oplus H_- \to L\oplus \ZZ[S].
\] 
 \item \label{it:mv3} $H^3(M_+)\oplus H^3(M_-)=\bfa_-^0H^3(V_+)\oplus \bfa_-^1
   H^2(V_+)\oplus \bfa_+^0H^3(V_-)\oplus \bfa_+^1H^2(V_-)$, \\
   $H^3(W)=\bfa_+^1\bfa_-^0H^2(S)\oplus \bfa_+^0\bfa_-^1H^2(S)$, and
\[
\gamma^3=
\begin{pmatrix}
  \beta^3_+&0&0&\rho_-\\0&\rho_+&\beta^3_-&0
\end{pmatrix}
\colon H^3(V_+)\oplus H_+\oplus H^3(V_-) \oplus H_- \to L\oplus L;
\]
\item $H^4(M_+)\oplus H^4(M_-)=\bfa_-^0H^4(V_+)\oplus
  \bfa_-^1H^3(V_+)\oplus \bfa_+^0H^4(V_-)\oplus \bfa_+^1 H^3(V_-)$, \\
  $H^4(W)=\bfa_+^0\bfa_-^0H^4(S)\oplus \bfa_+^1\bfa_-^1
  H^2(S)=H^4(S)\oplus L$, and
\[
\gamma^4=
\begin{pmatrix}
  \beta_+^4&0&\beta_-^4&0&\\0&\beta_+^3&0&\beta_-^3
\end{pmatrix}
\colon H^4(V_+)\oplus H^3(V_+) \oplus H^4(V_-)\oplus H^3(V_-)\to
H^4(S)\oplus L.
\]
 \end{enumerate}
\end{lemma}

\begin{proof}
  Straightforward.
\end{proof}

\begin{theorem}
\label{thm:g2topology}
\hfill
  \begin{enumerate}
  \item \label{it:h1m} $\pi_1(M)=(0)$ and $H^1(M)=(0)$;
  \item \label{it:h2m}
    $H^2(M) = \ker \Bigl[H_+\oplus H_-\to N_+ + N_-\Bigr]$, in
    other words, %
    $H^2(M)\simeq (N_+\cap N_-)\oplus K_+\oplus K_-$;
  \item \label{it:h3m}
    $H^3(M) \simeq \ZZ[S]\oplus (L/_{N_+ + N_-}) \oplus (N_-\cap T_+)
    \oplus (N_+\cap T_-)\oplus H^3(Z_+)\oplus H^3(Z_-)\oplus K_+\oplus K_-;$
  \item \label{it:h4m}
    $H^4(M)\simeq H^4(S)\oplus (T_+\cap T_-)\oplus
     (L/_{N_- + T_+})\oplus (L/_{N_+ + T_-}) \oplus H^3(Z_+) \oplus H^3(Z_-)
     \oplus K_+^\ast \oplus K_-^\ast$. 
  \end{enumerate}
\end{theorem}

\begin{proof}
Since $\pi_1(V_\pm) = (0)$, the van Kampen theorem for the decomposition
$M=M_+\cup M_-$ along the common intersection $W=S\times \TT^2$
immediately implies that $\pi_1(M)=(0)$.

We know that $\gamma^0$ is surjective and $\gamma^1$
  injective, hence \ref{it:h1m}. Since $\gamma^1$ is surjective, $H^2(M)=\ker
  (\gamma^2)=\ker \left(H_+{\oplus}H_- \to N_+ {+} N_- \right)$.
  Thus, we have an exact sequence:
\[
0\to K_+\oplus K_-\to H^2(M)\to N_+\cap N_- \to (0),
\]
 which is split since $N_+\cap N_-$ is torsion-free \ref{it:h2m}.
 To show \ref{it:h3m} note first that, from the description of $\gamma^2$, it
 is clear that
\[
\coker (\gamma^2)=\ZZ[S]\oplus (L/_{N_+ + N_-}).
\]
 Now $\ker (\gamma^3)$ is a direct sum of two pieces
\[
\ker_\pm =
\ker \Bigl[   \begin{pmatrix} \beta^3_\pm & \rho_\mp
\end{pmatrix}\colon H^3(V_\pm)\oplus H_\mp \to L\Bigr] .
\]
 Each of these kernels is computed by a split exact sequence:
\[
(0) \to H^3(Z_\pm)\oplus K_\mp \to \ker_\pm \to N_\mp\cap T_\pm \to (0)
\]
 and \ref{it:h3m} follows from \eqref{eq:M_splits}.
To show \ref{it:h4m} note first that, from the
 description of $\gamma^3$, it is clear that
\begin{equation}
\label{eq:mvimage}
\coker (\gamma^3) = (L/_{N_+ + T_-})\oplus (L/_{N_- + T_+}).
\end{equation}
Now $\ker (\gamma^4)$ is the direct sum of two pieces
\[
\ker \Bigl[   \begin{pmatrix}
    \beta^4_+ & \beta^4_-
  \end{pmatrix}\colon H^4(V_+)\oplus H^4(V_-) \to H^4(S)\Bigr]\oplus
\ker \Bigl[   \begin{pmatrix}
    \beta^3_+ & \beta^3_-
  \end{pmatrix}\colon H^3(V_+)\oplus H^3(V_-) \to L\Bigr].
\]
 The first of these kernels is isomorphic to $H^4(S)\oplus K_+^\ast
 \oplus K_-^\ast$; the second is isomorphic to $(T_+\cap T_-)\oplus
 H^3(Z_+)\oplus H^3(Z_-)$, and \ref{it:h4m} again follows from
 \eqref{eq:M_splits}.
 \end{proof}

\begin{remark}
\label{rmk:mvfactor}
If $[\alpha] \in H^k_{cpt}(V_\pm)$ then
$\bfa_\mp^0[\alpha] \in H^k_{cpt}(M_\pm)$ can be pushed forward to a class
in $H^4(M)$. Denote this map by $i_\pm : H^k_{cpt}(V_\pm) \to H^k(M)$.
Then the restriction of $\delta : H^3(W) \to H^4(M)$ to
$\bfa_\pm^1 \bfa_\mp^0 H^2(S)$ equals
$i_\pm \circ \partial_\pm : H^3(\Sph^1 \times S) \to H^4(M)$.
In the expression \eqref{eq:mvimage} for $\imag \delta$, we can identify
$L/_{N_\mp + T_\pm}$ as the image of $i_\pm \circ \partial_\pm$.
\end{remark}

\subsubsection{Torsion of the cohomology}

From Theorem \ref{thm:g2topology} we can immediately identify the torsion
part of the cohomology.

 \begin{corollary}
 \label{cor:torsion}
 \hfill
   \begin{enumerate}
   \item \label{it:h3t} $\Tor H^3(M) \simeq  \Tor (L/_{N_+ + N_-})$;
   \item \label{it:h4t}
   $\Tor H^4(M) \simeq  \Tor (L/_{N_- + T_+})\oplus \Tor (L/_{N_+ + T_-})$.
   \end{enumerate}
 \end{corollary}

While it is not of central importance to our application, let us also determine
the \emph{torsion linking form} on $\Tor H^4(M)$. Recall that for a closed
oriented manifold $M$ of odd dimension $2n-1$, the torsion linking form can be
interpreted as follows: Let $\alpha$, $\beta$ be cocycles representing
torsion elements in $H^n(M)$. Then $k\alpha = d\gamma$ for some
$(n{-}1)$-cocycle $\gamma$ and $k \geq 2$. The mod~$k$ reduction of $\gamma$
represents a class in $H^{n-1}(M; \cg{k})$, independent of the choices
up to addition of integral classes
(it is a pre-image of $[\alpha]$ under the Bockstein map associated to
$0 \to \ZZ \to \ZZ \to \cg{k} \to 0$). Therefore %
the linking form
$b([\alpha], [\beta]) = \frac{1}{k}(\gamma \cup \beta)[M] \in \QQ/\ZZ$
is well-defined. $b$~is symmetric if $n$ is even, and skew-symmetric
if $n$ is odd.

\begin{lemma}
In \ref{it:h4t}, the two summands are isotropic with respect to the
torsion-linking form.
\end{lemma}

\begin{proof}
Using Remark \ref{rmk:mvfactor}, we can describe the elements of the
$L/_{N_- + T_+}$ summand in $H^4(M)$ as follows.
Given an element in $L$, take a closed cochain $\alpha \in C^2(S;\ZZ)$
representing it. Extend $\bfa_+^1 \alpha$ to $\alpha_+ \in C^3(V_+;\ZZ)$.
Then $d\alpha_+$ has compact support in $V_+$, and represents
$\partial_+ [\bfa_+^1 \alpha] \in H^4_{cpt}(V_+)$.
Further, $\bfa_-^0 d\alpha_+$ can be extended to a cochain on $M$.
The extension $S\alpha \in C^4(M;\ZZ)$ is
closed, and the image of $[\alpha]$ in the $L/_{N_+ + T_-}$ summand
in $H^4(M)$ is $[S\alpha]$.

Now suppose that $[\alpha]$ represents a torsion element of $L/_{N_- + T_+}$,
\ie $k[\alpha] \in N_- + T_+$ for some $k \geq 2$. To compute the torsion
linking of the image $[S\alpha] \in H^4(M)$ with some other torsion
element we need to identify a prederivative of $kS\alpha$
that is defined on all of $M$. By Lemma \ref{lem:mvmaps}\ref{it:mv3},
$k[\alpha] \in N_- + T_+$ means that there
are $[\psi_+] \in H^3(V_+)$ and $[\psi_-] \in H^2(V_-)$ such that
$\bfa_+^1 \bfa_-^0 k[\alpha] =
\bfa_+^1 \bfa_-^0 [\psi_+]_{|W} + \bfa_-^0[\psi_-]_{|W}$.
The representatives $\psi_+ \in C^3(V_+;\ZZ)$ and $\psi_- \in C^2(V_-;\ZZ)$
can be chosen so that
$\bfa_+^1 \bfa_-^0 k\alpha =
\bfa_+^1 \bfa_-^0 \psi_{+|W} + \bfa_-^0\psi_{-|W}$.
Then we can define $P\alpha \in C^3(M;\bbz)$ by setting the restrictions to
$M_+$ and $M_-$ to be $\bfa_-^0(k\alpha_+ - \psi_+)$ and $\bfa_+^1\psi_-$
respectively.
Since $k \bfa_-^0 d \alpha_+ = d(\bfa_-^0(k\alpha_+ - \psi_+))$ on $M_+$
while $\psi_-$ is closed, $d(P\alpha) = kS\alpha$. 

Thus, if $[\alpha']$ represents some other torsion element of $L/_{N_+ + T_-}$
then the torsion linking of their images $[S\alpha]$ and $[S\alpha']$ in
$H^4(M)$ is given by the integral of the cup product of the mod~$k$ cocycles
$P\alpha$ and $S\alpha'$, %
which vanishes: $S\alpha'$ is zero on $M_-$, while the restrictions of both
cochains to $M_+$ are pull-backs from $V_+$. Hence the $\Tor L/_{N_+ + T_-}$
summand is isotropic.
\end{proof}

\noindent
In fact, one can see from the proof that the torsion linking form is given by
the following natural pairing
\begin{align*}
\Tor (L/_{N_- + T_+}) \times& \Tor (L/_{N_+ + T_-}) \to \QQ/\ZZ \\
([\alpha] ,& [\beta]) \mapsto b([\alpha], [\beta])  :
\end{align*}
We can write $k\alpha = n + t$ for some $k \geq 1, \; n \in N_-, \; t \in T_+$. 
$\inner{t, \beta} = - \inner{n, \beta} \mod k$ is independent of choice
of $n$ and $t$ because $\beta \perp N_- + T_+$, and it is also independent
of the choices of representatives $\alpha, \beta \in L$. Therefore
$b([\alpha], [\beta]) = \frac{1}{k}\inner{t,\beta} \in \QQ/\ZZ$ is
well-defined.

\begin{remark*}
If $H^3(Z)$ is not torsion-free, then Corollary \ref{corg:V&S^1(S)} remains
true, except that
$0 \to \bar K \to H^4(V) \overset{\beta^4}{\longrightarrow} H^4(S) \to 0$,
where there are natural isomorphisms $K \cong \Hom(\bar K, \ZZ)$ and
$\Tor \bar K \cong \Tor H^4(Z)$.
Theorem \ref{thm:g2topology} remains true too except that appearances of
$K_\pm^*$ should be replaced by $\bar K_\pm$ (but proving that the short exact
sequences used in the proof split becomes a bit more complicated).
Then $\Tor H^4(M)$ contains summands $\Tor H^4(Z_\pm) \oplus \Tor H^3(Z_\pm)$.
The torsion linking form on these terms is the obvious one.
\end{remark*}

\begin{remark}
\label{rmk:tor_form}
The fact that the torsion of $H^4(M)$ always splits as a direct sum of two
subgroups isotropic under the torsion-linking form means that for manifolds
of this twisted connected sum type, the isomorphism class of the
torsion-linking form is determined by the isomorphism class of $H^4(M)$.
\end{remark}

\subsubsection{Gluing classes in $H^4(Z_\pm)$}

The Mayer--Vietoris theorem says that if we try to glue a pair of classes in
$H^4(M_+)$ and $H^4(M_-)$ having the same image in $H^4(W)$ to a class in
$H^4(M)$ then there is an ambiguity given by the image of
$\delta : H^3(W) \to H^4(M)$. However, in this particular construction there
\emph{is} an unambiguous way to glue a pair of classes in $H^4(Z_+)$ and
$H^4(Z_-)$, which will be important for describing the characteristic classes
of $M$.
Define
\[ H^4(Z_+) \oplus_0 H^4(Z_-) \; = \;
\left\{ ([\alpha_+], [\alpha_-]) \in H^4(Z_+) \oplus H^4(Z_-) :
[\alpha_+]_{|S} = [\alpha_-]_{|S} \in H^4(S) \right\} . \]

\begin{definition}
\label{def:y}
We define a map
\[ Y : H^4(Z_+) \oplus_0 H^4(Z_-) \to H^4(M) \]
as follows. Recall that $S = f_\pm^{-1}(\infty)$ for a fibration
$f_\pm \colon Z_\pm \to \PP^1$.
Let $\Delta \subset \PP^1$ be a trivialising neighbourhood of $\infty$ for
$f_\pm$,
and let $U_\pm = f_\pm^{-1}(\Delta) \cong \Delta \times S \subset Z_\pm$.
($\Delta^* \times S$ correspond to the cylindrical ends
$\bbrp \times \Sph^1 \times S$ of $V_\pm$, mapping 
$\bbrp \times \Sph^1 \to \Delta^*$ by
$(t,\anglen) \mapsto z = e^{-t-i\anglen}$).
Let $p_\pm : U_\pm \to S$ be the projection for the local trivialisation.
For $([\alpha_+], [\alpha_-]) \in H^4(Z_+) \oplus_0 H^4(Z_-)$, let $[\beta]$ be
their common image in $H^4(S)$. Then we may choose the cocycles 
$\alpha_\pm \in C^4(Z_\pm;\bbz)$ so that the restriction of $\alpha_\pm$ to
$U_\pm$ equals $p_\pm^*\beta$. The pull-backs of $\alpha_\pm$ to
$\Sph^1 \times V_\pm$ have the same restriction to the cylindrical end, and
patch to a cocycle on $M$. We set $Y([\alpha_+], [\alpha_-])$ to be the
class represented by that cocycle.
\end{definition}

Let $N'_\mp$ be the image of $N_\mp$ in $N_\pm^* = L/T_\pm$.
Recall from Lemma \ref{lemg:Z&V}\ref{itg:h4} that $N_\pm^* \into H^4(Z)$.
The image lies in the kernel of restriction to $V$ and hence also of
restriction to~$S$, so $N_\pm^* \into H^4(Z_+) \oplus_0 H^4(Z_-)$.

\begin{lemma}
\label{lem:ymv}
$Y : H^4(Z_+) \oplus_0 H^4(Z_-) \to H^4(M)$ maps onto the terms
\[ H^4(S) \oplus (L/_{N_- + T_+}) \oplus (L/_{N_+ + T_-})
\oplus K_+^\ast \oplus K_-^\ast  \subseteq H^4(M) . \]
in the expression \ref{thm:g2topology}\ref{it:h4m} for $H^4(M)$,
with kernel $N'_+ \oplus N'_-$.
\end{lemma}

\begin{proof}
It follows from \eqref{eq:zrel} that
$0 \to H^4_{cpt}(V_\pm) \to H^4(Z_\pm) \to H^4(S) \to 0$ is split exact.
Hence so is 
\[ 0 \to H^4_{cpt}(V_+) \oplus H^4_{cpt}(V_-) \to H^4(Z_+) \oplus_0 H^4(Z_-)
\to H^4(S) \to 0. \]
Recall that $H^4_{cpt}(V_\pm) \cong N_\pm^* \oplus K_\pm^*$, where
$N_\pm^* = \imag \partial_\pm$. The result follows from identifying
$N_\pm'$ as the kernel of $i_\pm : \imag \partial_\pm \to \imag \delta$ in
Remark \ref{rmk:mvfactor} and considering the commutative diagram

\xymatrix@C=-8mm{
H^3(\Sph^1 \times S) \oplus H^3(\Sph^1 \times S)
\ar[rr]^{\hspace{5mm}\partial_+ \oplus \partial_-} \ar[dd]_{\cong} & &
H^4_{cpt}(V_+) \oplus H^4_{cpt}(V_-) \ar[rr] \ar[dd]_{i_+ + i_-}
\ar@{_{(}->}[rd] & & H^4(V_+) \oplus H^4(V_-) \ar@{^{(}->}[dd] \\
& \hspace{2.5cm} & & H^4(Z_+) \oplus_0 H^4(Z_-) \ar[ur] \ar[ld]_Y \\
H^3(T^2 \times S) \ar[rr]^{\delta} & & H^4(M) \ar[rr]^{\rho^4} &
& H^4(M_+) \oplus H^4(M_-)%
}

\noindent where the top row is the direct sum of the sequences \eqref{eq:vrel}
of relative cohomology for $V_+$ and~$V_-$, and the bottom row is the
Mayer--Vietoris sequence \eqref{eq:mayer_vietoris}.
\end{proof}

\subsection{Characteristic classes of twisted connected sums}
\label{ss:p1calc}

We now consider how to determine the characteristic classes of a twisted
connected sum in terms of related data on the building blocks $Z_\pm$.
We begin with a summary of the characteristic classes of 
relevance for a closed oriented spin 7-manifold $M$.

\subsubsection{Oriented characteristic classes}

The characteristic classes of (the tangent bundle of) an oriented 7-manifold
$M$ are the Stiefel--Whitney classes $w_2(M), \ldots, w_7(M)$ and the first
Pontrjagin class $p_1(M)$.
First we want to show that all the Stiefel--Whitney classes vanish for any 
oriented spin 7-manifold and hence that 
the only oriented characteristic class of interest for a \gtmfd{} is $p_1(M)$.
We will use some standard facts about characteristic classes, which can
be found in Milnor--Stasheff \cite{MS}.
First of all, for any vector bundle $E \to M$ the Stiefel--Whitney class
$w_k(E) \in H^k(M;\cg{2})$ can be determined from $\{w_{2^i}(E) : 2^i \leq k\}$
using the Steenrod square operations
$\Sq^k : H^i(M; \cg{2}) \to H^{i+k}(M;\cg{2})$ \cite[8B]{MS}, \eg
\begin{equation}
\label{eq:w3}
w_3 = \Sq^1 w_2 + w_1 w_2 .
\end{equation}
Hence all Stiefel--Whitney classes of an oriented rank 7 bundle are determined
algebraically by $w_2$ and $w_4$.
Further, Wu's formula \cite[Theorem 11.14]{MS} expresses the Stiefel--Whitney
classes of a closed $n$-dimensional manifold $M$ as
\begin{equation}
\label{eq:wu}
w_k = \sum_{i=0}^k \Sq^{k-i}v_i ,
\end{equation}
where the Wu class $v_k(M)$ can be defined as the Poincar\'e dual to
$\Sq^k \colon H^{n-k}(M; \cg{2}) \to H^n(M; \cg{2})$.
Applying \eqref{eq:wu} recursively, we find for any closed oriented
manifold that ${v_1 = w_1 = 0}$, $v_2 = w_2$, combining with \eqref{eq:w3} gives
$v_3 = 0$, and
\begin{equation}
v_4 = w_4 + w_2^2
\end{equation}
since $\Sq^2 a = a^2$ for any $a \in H^2(M; \cg{2})$. Because $\Sq^k$ vanishes
on $H^i(M; \cg{2})$ for $i < k$, Wu classes above the middle dimension always
vanish (\cf \cite[page 132]{MS}), so $w_4 = w_2^2$ for any closed orientable
7-manifold $M$. If $M$ is spin, then $w_2 = 0$, and hence all other Stiefel--Whitney
classes vanish too.

\subsubsection{Spin characteristic classes}

The Stiefel--Whitney and Pontrjagin classes are all stable, \ie they are
invariant under addition of trivial bundles. The stable characteristic classes
of an oriented vector bundle are pull-backs of elements of $H^*(\bso)$ under a
classifying map $M \to \bso$, where $\bso$ is the classifying space for the
stable special orthogonal group $\sso = \lim_{n \to \infty} \sorth{n}$.
If the vector bundle is spin then the classifying map can be lifted to
$\bspin$, and we can possibly define further characteristic classes by
considering $H^*(\bspin)$. 
$\bso$ and $\bspin$ have isomorphic cohomology groups over $\Q$ or mod $p$ with
$p$ an odd prime, but over $\Z$ and mod~2 there is extra subtlety that was
first studied by Thomas \cite{thomas62}; however, for our purposes we will only
be interested in the first nonzero (integral) class in $H^{4}(\bspin)$ 
and we can give a direct and elementary description of this class as follows.
The stable spin group $\sspin$ is \mbox{2-connected}, with
$\pi_3(\sunitary{2}) \cong \pi_3(\sspin)$, realised by the inclusion
$\sunitary{2} \into \spin{4}$ given by the standard rank 2 complex
representation. The homotopy long exact sequence of the fibration
$\sspin \into \espin \to \bspin$, where $\espin$ is contractible, shows that
$\bspin$ is 3-connected with $\pi_4(\bspin) \cong \ZZ$. By the Hurewicz theorem
$H^4(\bspin) \cong \ZZ$, and we denote the stable characteristic class
corresponding to a choice of generator by $\posp$.

The following well-known lemma implies that if there is no 2-torsion in
$H^4(M)$ then $\posp(M)$ is determined from the Pontrjagin class $p_{1}(M)$.
Since we are mostly concerned with the case when $H^4(M)$ is torsion-free,
for simplicity we choose to phrase our subsequent main discussion in terms of
$p_1(M)$, addressing the refinements concerning $\posp(M)$ in supplementary
remarks.

\begin{lemma}[{\cf \cite[(1.5),(1.6)]{thomas62}, \cite[Lemma 2.4]{cadek08}}]
\label{lem:posp_div}
For any spin bundle, 
$p_1 = 2\posp$ and $w_4 = \posp \mod 2$.
\end{lemma}

\begin{proof}
Because $H^4(\bspin) \cong H^4(\bsu(2))$, it suffices to prove that the
relations hold for $SU(2)$-bundles; these are complex rank 2 bundles
with $c_1 = 0$, so they are spin because $w_2 = c_1 \mod 2$.
The generator for $H^4(\bsu(2))$ is $c_2$ so, fixing the sign of $\posp$, any
$\sunitary{2}$-bundle has $\posp = -c_2$. On the other hand,
$p_1 = -2c_2 + c_1^2$ \cite[Corollary 15.5]{MS} and
$w_4 = c_2 \mod 2$ \cite[14B]{MS} for any complex vector bundle. 
\end{proof}
Since as we explained above $w_{4}=0$ for any closed spin 7-manifold we deduce the following:
\begin{corollary}
\label{cor:p_half}
If $M$ is a closed spin 7-manifold then $\posp(M)$ is even, and hence $p_1(M)$
is divisible by 4.
\end{corollary}

\begin{remark*}
For an oriented vector bundle $E \to M$ with $w_2(E) = 0$,
$\posp(E)$ is independent of the choice of spin structure
\cite[page 4]{cadek08}.
\end{remark*}

\begin{remark*}
We used above that $w_1(E) = 0$ if and only $E$ is orientable, and that if also
$w_2(E) = 0$ then $E$ is spin. Another interpretation is that $w_1$ and $w_2$
are the successive obstructions to trivialising $E$ over the 1- and
2-skeleta, respectively, of the base manifold. 
There is a similar obstruction-theoretic interpretation of the vanishing of $\posp$ 
which we now explain, though we will make no subsequent use of this fact.
The relation of $\posp(E)$ of a spin vector bundle
$E$ to the generator of $\pi_3(\sspin)$
means that it is the primary obstruction to stable trivialisability of $E$ in
the sense of Steenrod \cite[Definition 35.3]{steenrod51}: that $\sspin$ is 2-connected implies that $E$ is always stably trivialisable over the 3-skeleton of the
base manifold, and $\posp(E)$ is the obstruction to stable trivialisability
over the 4-skeleton. Because $\sspin$ has $\pi_4 = \pi_5 = \pi_6 = 0$, there are
no further obstructions until degree 8, so the tangent bundle of a spin
7-manifold $M$ is stably trivial if and only if $\posp(M) = 0$. The relation
between obstruction classes and Pontrjagin classes was determined by
Kervaire \cite{kervaire59}.
\end{remark*}

There are no further spin characteristic classes of interest in degree
below 8, not even if we consider unstable classes of a rank 7 spin bundle, 
\ie we consider classes obtained by pullback from $H^{*}(\bspin(7))$ and not only $H^{*}(\bspin)$.
Since we are studying manifolds with \gtstr s we could also consider
characteristic classes defined by elements of $H^*(\bgtwo)$, but there are
no classes of interest beyond $\posp$ of the associated spin bundle.
$H^*(\bspin(7))$ and $H^*(\bgtwo)$ are both described by Gray
\cite[Theorem 3.4]{gray69}.

\subsubsection{Computing $p_1$ of twisted connected sums}

The restrictions $p_1(\Sph^1 \times V_\pm)$ of $p_1(M)$ to
$\Sph^1 \times V_\pm$ do not determine $p_1(M)$ since the Mayer-Vietoris
boundary map $H^3(W) \to H^4(M)$ is non-trivial. Another point of view is that
the isomorphism class of a vector bundle on $M$ is not determined by the
isomorphism classes of its restrictions to $V_+$ and $V_-$: it also depends
on (the homotopy class of) the isomorphism one uses to glue the bundles
together on the overlap.
However, it turns out that we can determine $p_1(M)$ from $p_1(Z_\pm)$,
using the map $Y$ from Definition \ref{def:y}.

Recall %
that $p_1(Z)= -2c_2(Z) + c_1(Z)^2$ for any complex manifold $Z$. 
If $Z$ is a building block then $c_1(Z)^2 = 0$, so $p_1(Z) = -2c_2(Z)$.
The image of $c_2(Z_\pm)$ in $H^4(S)$ is $c_2(S)$, so in particular
$(c_2(Z_+), c_2(Z_-)) \in H^4(Z_+) \oplus_0 H^4(Z_-)$,
and $Y(c_2(Z_+), c_2(Z_-))$ is defined.

\begin{prop}
\label{prop:p1y}
Let $M$ be a twisted connected sum of the building blocks $Z_+$ and $Z_-$.
Then
\[ p_1(M) = -2Y(c_2(Z_+), c_2(Z_-)) . \]
\end{prop}

\begin{proof}
We need to find a suitable cocycle representing $p_1(M)$.
Let $E_7(\bbr)$ be the tautological bundle over
$\bso(7) = \ogr_7(\bbr^{\infty})$, the Grassmannian of oriented 7-planes.
A classifying map for $TM$ is a $g : M \to \ogr_7(\bbr^{\infty})$ such that
there is a vector bundle isomorphism $G : TM \to g^* E_7(\bbr)$.
By definition, there is a cocycle
$\wp_1 \in C^4(\ogr_7(\bbr^{\infty}); \ZZ)$ such that
$p_1(M) = [g^* \wp_1]$ for any classifying map $g$.

Consider $Z_\pm$ as the union of $V_\pm = Z_\pm \setminus S$ and
$U_\pm \cong \Delta \times S$, and define a vector bundle $R_\pm$ over $Z_\pm$
by gluing $TV_\pm$ and $TU_\pm$ as follows: on the overlap
$\bbrp \times \Sph^1 \times S \cong \Delta^* \times S$,
$(t,\anglen) \mapsto z = x+iy = e^{-t-i\anglen}$, 
map $TS$ to $TS$ by the
identity, and $T(\bbr \times \Sph^1)$ to $T\Delta^*$ by
$\contra{t} \mapsto \contra{x}$, $\contra{\anglen} \mapsto \contra{y}$.
Identifying a complex vector bundle with the $(1,0)$-part of its
complexification, this is the complex linear map
$\contra{t} - i \contra{\anglen} \mapsto \contra{z}$.
In contrast, $TZ_\pm$ is formed by gluing $TV_\pm$ and $TU_\pm$ by the
derivative of $(t,\anglen) \mapsto z$, which maps
$\contra{t} - i \contra{\anglen} \mapsto z\contra{z}$.
For comparison, if we glue the trivial complex line bundle $\trivc$ over
$V_\pm$ to $\trivc$ over $U_\pm$ by $u \mapsto z^{-1}u$ over
$\Delta^* \times S$, then the result is $[-S]$, the line bundle over $Z_\pm$
with divisor $-S$. Now
\[ R_\pm \oplus \trivc \; \cong \; TZ_\pm \oplus [-S] , \]
because both bundles are the result of gluing $TV_\pm \oplus \trivc$ to
$TU_\pm \oplus \trivc$ by homotopic maps; at $(z,p) \in \Delta^* \times S$, the
difference of the gluing maps
sends $(v,w,u) \in TS \oplus T\Delta \oplus \trivc$ to $(v, zw, z^{-1}u)$, and
any $\Delta^* \to SU(2)$ is homotopic to a constant since $SU(2)$ is
simply-connected.
Because $p_1$ is additive and $p_1([-S]) = [-S]^2= 0$, we find
\[ p_1(R_\pm) = p_1(Z_\pm) = -2c_2(Z_\pm) . \]

Let $f : S \to Gr_2(\bbc^\infty)$ be a classifying map for the complex vector
bundle $TS$, with an isomorphism $F : TS \to f^*E_2(\bbc)$. Identifying
$\bbc \oplus \bbc^\infty \cong \bbc^\infty$ embeds
$Gr_2(\bbc^\infty) \into Gr_3(\bbc^\infty)$, and
$f^* E_3(\bbc) \cong \trivc \oplus f^*E_2(\bbc)$.
Let $f_\pm : Z_\pm \to Gr_3(\bbc^\infty)$ be a classifying map for $R_\pm$ such
that $f_{\pm|U_\pm} = f \circ p_\pm$ and the restriction of the isomorphism
$F_\pm : R_\pm \cong f_\pm^*E_3(\bbc)$ to $U_\pm$ maps the $TS$ factor
of $R_{\pm|U_\pm} = TU_\pm$ to $f^*E_2(\bbc)$ by $F$, and the $T\Delta$ factor
to $\trivc$ by the identity map.

Let $g_\pm : V_\pm \to Gr_3(\bbc^\infty)$ and
$G_\pm : TV_\pm \to g_\pm^*E_3(\bbc)$ be the restriction of $f_\pm$ and
$F_\pm$. The gluing map in the definition of $R_\pm$ has been chosen
precisely to ensure that,
restricted to the cylindrical end $\bbrp \times \Sph^1 \times S$,
$G_\pm$ maps $T(\bbrp \times \Sph^1) \to \trivc$ by
$u (\contra{t} - i \contra{\anglen}) \mapsto u$
(and $TS \to f^*E_2(\bbc)$ by $F$).
Finally, define $g : M \to \ogr_7(\bbr^\infty)$ by patching the
compositions
\[ \Sph^1 \times V_\pm \to V_\pm \stackrel{g_\pm}{\to}
Gr_3(\bbc^\infty) \to \ogr_7(\bbr^\infty) ; \]
this is possible because on the neck region $\bbr \times T^2 \times S$, both
$g_+$ and $g_-$ equal the composition of $f$ with projection to $S$.
Further, on the neck region the restrictions
$G_\pm : T(\bbr \times T^2 \times S) \to g^*E_7(\bbr) \cong
\trivr^3 \oplus f^* E_4(\bbr)$ are both translation-invariant, and differ by
a \emph{constant} rotation of the $\trivr^3$ factor. By picking a
path from this rotation to the identity we can interpolate between $G_+$ and
$G_-$ to define an isomorphism $G : TM \to g^*E_7(\bbr)$.
Hence $g$ is a classifying map for $TM$, and
\[ p_1(M) = [g^*\wp_1] =
Y([f_+^*\wp_1], [f_-^*\wp_1]) = Y(p_1(Z_+), p_1(Z_-)) . \qedhere \]
\end{proof}

\begin{remark*}
If $K_\pm = 0$, then Corollary \ref{corg:V&S^1(S)}(iv) implies that
$H^4(V_\pm) \cong H^4(S) \cong \ZZ$, and $c_2(V_\pm)$ is completely
determined by the fact that the restriction of $c_2(V_\pm)$ to $S$ is
$c_2(S) \cong \chi(S) = 24$. Thus considering
$p_1(\Sph^1 \times V_\pm) = -2c_2(V_\pm)$ instead of $c_2(Z_\pm)$ really
does lose a lot of information. 
\end{remark*}

\begin{remark}
Note that $c_1(R_\pm) = 0$; indeed the gluing map in the construction matches
the non-vanishing complex 3-forms $\Omega_\pm$ and 
$dz \wedge (\omega_\pm^J + i\omega_\pm^K)$ over $V_\pm$ and $U_\pm$.
In particular $R_\pm$ is a spin bundle, and its spin characteristic class
$\posp(R_\pm)$ equals $-c_2(R_\pm) = -c_2(Z_\pm)$. Carrying out the proof
of Proposition \ref{prop:p1y} using classifying maps to appropriate versions
of $\bspin$ and $\bsu$ therefore proves the more refined statement that
$\posp(M) = -Y(c_2(Z_+), c_2(Z_-))$.
\end{remark}

\begin{remark*}
If we work with real coefficients then the relation
$p_1(M) = Y(p_1(Z_+), p_1(Z_-))$ is more conveniently proved using
Chern--Weil theory. It is clear how to define $Y$ as a map on de Rham
cohomology $H^4_{dR}(Z_+) \oplus_0 H^4_{dR}(Z_-) \to H^4_{dR}(M)$.
For a Riemannian metric $g$ on~$M$, a certain quadratic polynomial function of
the curvature of $g$ defines a differential form $p_1(g) \in \Omega^4(M)$
representing $p_1(M) \in H^4_{dR}(M)$.

Let $g_S$ be a metric on $S$, and $g_\pm$ a metric on $V_\pm$ that equals
$dt^2 + d\anglen^2 + g_S$ on the cylindrical end.
Let $g'_\pm$ be a metric on $Z_\pm$ that equals $g_\pm$ outside a neighbourhood
of $S$ and is a product metric on $\Delta \times S$, equal to $|dz^2| + g_S$
near $S$. Then $p_1(g'_\pm) = p_1(|dz^2|) + p_1(g_S) = p_1(g_S)$ on
$\Delta \times S$, so the differential forms $p_1(g_\pm)$ and
$p_1(g'_\pm)_{|V_\pm}$ are equal.
Finally let $g$ on $M$ be a patching of $d\anglex^2 + g_\pm$ on
$\Sph^1 \times V_\pm$. Then
$p_1(M) = [p_1(g)] = Y([p_1(g'_+)], [p_1(g'_-)]) = Y(p_1(Z_+), p_1(Z_-))$.
\end{remark*}

\subsection{Smooth type of connected-sum \gtmfd s}

Many of the \gtmfd s we construct in this paper are $2$-connected; 
in this case we can compute classifying topological invariants and in many 
cases determine the diffeomorphism type of the underlying smooth $7$-manifold.
These are the first compact manifolds with holonomy $\gtwo$ for which 
the diffeomorphism type of the underlying $7$-manifold has been determined.
We will see in \S \ref{sec_g2mfds} that in many cases we can get $7$-manifolds
with the same invariants by taking the twisted connected sum of completely
unrelated pairs of building blocks, and can thus construct different metrics
with holonomy $\gtwo$ on the same underlying smooth $7$-manifold.
Judicious choices of pairs of building blocks allow us to vary the number of
compact associative $3$-folds we can exhibit in different \gtwo-holonomy
metrics on the same smooth $7$-manifold.

Let us first review the classification theory of smooth 2-connected
7-manifolds; we concentrate on the simplest case, namely where the cohomology is torsion-free.
Lemma \ref{l:2connected:g2} gives sufficient conditions on a twisted connected sum manifold $M$
to ensure that $M$ is $2$-connected with torsion-free cohomology, and therefore the 
classification theory discussed below applies to $M$.

\subsubsection{Almost-diffeomorphism classification of smooth closed
2-connected 7-manifolds}
Two smooth manifolds $M, N$ are \emph{almost-diffeomorphic} if there is a
homeomorphism $M \to N$ that is smooth away from a finite set of points; 
this is equivalent to $M$ being diffeomorphic to $N \# \Sigma$ for some
homotopy sphere~$\Sigma$. Recall that by the $h$-cobordism theorem, any
homotopy sphere in dimension $n > 4$ is an exotic sphere, \ie a smooth manifold
homeomorphic to $\Sph^n$; under connected sums the homotopy spheres form a finite abelian group denoted $\Theta_n$.
The group $\Theta_7$ of exotic 7-spheres is $\cg{28}$. 
It turns out that the classification of smooth $2$-connected $7$-manifolds is
the same up to homeomorphism as up to almost-diffeomorphism; in particular
there are at most 28 smooth structures on any 2-connected topological
7-manifold.

Let $M$ be a smooth connected closed 7-manifold that is 2-connected, \ie
$\pi_1(M)$ and $\pi_2(M)$ are trivial. Then $H_1(M) \cong H_2(M) = 0$ by the
Hurewicz theorem, so \mbox{$H^1(M) = H^2(M) = 0$} by universal coefficients and
$H^5(M) = H^6(M) = 0$ by Poincar\'e duality. So apart from
\mbox{$H^0(M) \cong H^7(M) \cong \bbz$} the only non-vanishing cohomology groups are $H^3(M)$, which is
torsion-free, and $H^4(M)$, whose free part is isomorphic to $H^3(M)$.
If $H^4(M)$ is torsion-free then all the information about
the cohomology of $M$ reduces to the integer $b^{3}(M) = b^{4}(M)$.

Another invariant of $M$ is the first Pontrjagin class $p_1(M) \in H^4(M)$.
If $H^4(M)$ is torsion-free %
then the position of $p_1(M)$ in $H^4(M)$ up to isomorphism is determined by
the greatest divisor $\gdiv{p_1(M)} \in \NN$; recall from
Corollary \ref{cor:p_half} that this always divisible by 4.
For our purposes, the following special case of the classification results
of Wilkens \cite[Theorem 3]{wilkens72} will suffice.

\begin{theorem}
\label{thm:torfreeclass}
Smooth closed 2-connected 7-manifolds $M$ with $H^4(M)$ torsion-free are
classified up to almost-diffeomorphism by the isomorphism class of the pair
$(H^4(M), p_1(M))$,
or equivalently by the non-negative integers $b^{4}(M)$ and $\gdiv{p_1(M)}$. 
Moreover, any pair of non-negative integers of the form $(k,4m)$ is realised as $k=b^{4}(M)$ and $4m=\gdiv{p_{1}(M)}$ 
for some smooth closed $2$-connected $7$-manifold $M$.
\end{theorem}

By Novikov \cite{novikov65} rational Pontrjagin classes are natural under
homeomorphisms. In the absence of torsion in $H^4$, so are the integral
classes, \ie $p_1(M) = f^*p_1(N)$ for any homeomorphism $f \colon M \to N$.
Since the classifying almost-diffeomorphism invariants are
also invariant under homeomorphism, it follows that the classification
up to homeomorphism is the same.

\begin{remark}
\label{R:h4:torsion}
When $H^4(M)$ has torsion, the invariants in Theorem \ref{thm:torfreeclass}
need to be amended. Instead of $p_1(M)$, one should use the spin characteristic
class $\posp(M) \in H^4(M)$. %
The torsion-linking form
$b : TH^4(M) \times TH^4(M) \to \QQ/\ZZ$ (defined following \ref{cor:torsion})
is another obvious invariant; Wilkens showed that the isomorphism class of the
triple $(H^4(M), b, \posp(M))$ classifies $M$ up to almost-diffeomorphism when
$H^4(M)$ has no 2-torsion.
Crowley \cite[Theorem~B]{crowley01} showed that when $H^4(M)$ has 2-torsion
one obtains classifying invariants by replacing $b$ with one of 2 possible
``families of quadratic refinements''.
(All triples of invariants are realised subject only to the constraint that
$\posp(M)$ is divisible by 2.)
\end{remark}

\begin{remark}
\label{R:homotopy_class}
\cite[Theorem 6.11]{crowley01} can be paraphrased as stating that
smooth closed 2-connected 7-manifolds are classified up to homotopy equivalence
by the mod 24 reduction of the invariants in the almost-diffeomorphism
classification. When $H^4(M)$ has no 2-torsion this is the tensor
product of the triple $(H^4(M), b, \posp(M))$ with $\ZZ/24\ZZ$. 
In the case when $H^4(M)$ is torsion-free, this means that $M$ is classified
up to homotopy equivalence by the pair $b^4(M)$ and $\gdiv p_1(M) \mod 48$.
\end{remark}

\subsubsection*{Concrete realisations of $2$-connected smooth $7$-manifolds}
We can give concrete descriptions of many $2$-connected smooth $7$-manifolds
using $\Sph^{3}$-bundles over $\Sph^{4}$ and connected sums thereof.
The trivial bundle $\Sph^{3}\times \Sph^{4}$ gives a $2$-connected $7$-manifold
with torsion-free cohomology; clearly, it has $H^{3}(M)=H^{4}(M)= \Z$ and
vanishing first Pontrjagin class $p_{1}(M)$ (since $\Sph^{3}\times \Sph^{4}$ is
parallelisable). The $k$-fold connected sum $k(\Sph^{3}\times \Sph^{4})$ gives a
$2$-connected $7$-manifold with $H^{3}(M)=H^{4}(M)=\Z^{k}$ with $p_{1}(M)=0$ 
(since connected sums of stably parallelisable manifolds are stably
parallelisable, and Pontrjagin classes are stable). 

Via the usual `clutching' construction for bundles over a sphere, equivalence
classes of linear $\Sph^{3}$-bundles over $\Sph^{4}$ are in one-to-one
correspondence with $\pi_{3}(\sorth{4}) \cong \Z \oplus \Z$. Convenient generators for $\pi_{3}(\sorth{4})$ are given by
\[
\rho(u) v = uvu^{-1}, \quad \sigma(u)v = uv;
\]
here we have identified $\Sph^{3}$ with the unit quaternions and composition denotes quaternionic multiplication.
Identifying the pair of integers $(m,n)$ with the element $m \rho +n \sigma \in \pi_{3}(\sorth{4})$  hence determines
a real rank $4$ vector bundle $\xi_{m,n}$ over $\Sph^{4}$ and its corresponding $3$-sphere bundle 
$M_{m,n}:=S(\xi_{m,n}) \ra \Sph^{4}$, with projection map $\pi$.

By the homotopy long exact sequence of a fibration, any $\Sph^{3}$-bundle over
$\Sph^{4}$ is $2$-connected. Together with the fact that
$H^{4}(M_{m,n}) = \ZZ/n\ZZ$ (using the Gysin sequence and that the Euler number
of the bundle is $e(\xi_{m,n})=n$) this determines all the homology groups of
the bundle.
For the $\Sph^{3}$-bundles $M_{m,0}$ with Euler number $0$ we have
(\cf~Crowley and Escher \cite[Fact 3.1]{crowley03}) 
\[
H^{3}(M_{m,0}) \cong H^{4}(M_{m,0}) \cong  \pi^{*}H^{4}(\Sph^{4}) \cong \Z; \qquad \quad p_{1}(M_{m,0}) = 4m \kappa_{4} \in \Z;
\]
where $\kappa_{4}:= \pi^{*}\iota_{4}\in \pi^{*}H^{4}(\Sph^{4})$ is the generator of $H^{4}(M_{m,0}) \cong \Z$ 
and $\iota_{4}$ denotes a generator of $H^{4}(\Sph^{4}) \cong \Z$.

\begin{remark}
\label{r:2connected:7mfd}
The connected sum $M^{k}_{m}:= M_{m,0} \, \# \, (k{-}1)(\Sph^{3}\times \Sph^{4})$
is a $2$-connected smooth $7$-manifold with torsion-free cohomology,
$b^{3}(M_{m}^{k})= b^{4}(M_{m}^{k})=k$ and $\gdiv{p_{1}(M_{m}^{k})} = 4m$; 
taking a further connected sum with any exotic $7$-sphere $\Sigma \in \Theta_{7} \cong \ZZ/28\ZZ$ 
yields another $2$-connected $7$-manifold with the same invariants which may or may not be (oriented) 
diffeomorphic to~$M^{k}_{m}$.
\end{remark}

\subsubsection{Almost-diffeomorphism to diffeomorphism classification}
In general, finding the number of (oriented) diffeomorphism classes in the
almost diffeomorphism class of a $2$-connected $7$-manifold $M$ is equivalent to identifying the \emph{inertia subgroup}
\[
I(M) \subseteq \Theta_7 := \{\Sigma \in \Theta_{7} | \,
M \# \Sigma \text{\ is oriented-diffeomorphic to}\ M\} . \] 
\begin{theorem}[{\cite[Theorem 1]{wilkens75}}]
\label{thm:inertia}
Let $M$ be a closed 2-connected 7-manifold. If $H^4(M)$ has no 2- or 7-torsion
and $d$ is the greatest divisor of $p_1(M)$, then the inertia subgroup
$I(M) \subseteq \Theta_{7}$ consists of the
elements of $\Theta_7$ divisible by $d/8$. (If $p_1(M)$ is a torsion element
then we interpret $d$ to be 0, and $I(M)$ is trivial.)
\end{theorem}

\noindent
So, for example, if $\gcd(p_1(M), 8{\cdot}28)$ divides 8 then $I(M) = \Theta_7$
and any manifold almost-diffeomorphic to $M$ is actually diffeomorphic to $M$.
If there is torsion in $H^4(M)$ then one can still say that
$I(M) \subseteq (d_\pi/4)\Theta_7$ where $d_\pi$ is the greatest divisor of
$p_1(M)$ modulo torsion \cite[Corollary to Proposition 5]{wilkens75},
but the precise value of $I(M)$ may depend on the torsion linking form
\cite{7class}.

If $M$ has holonomy \gtwo then, by Proposition \ref{prop:p1}(ii),
$p_1(M)$ is never a torsion class even if $H^4(M)$ has torsion.
\subsubsection{Application to twisted connected sums}

We now consider compact \gtmfd s $M$ constructed as a
twisted connected sum from a pair of building blocks $Z_+, Z_-$ 
from the point-of-view of their diffeomorphism and almost-diffeomorphism type.
To begin with we deduce from our results on the cohomology of twisted connected
sum manifolds a simple sufficient condition for $M$ to be 2-connected and for
$H^4(M)$ to be torsion-free. Combined with our calculation of $p_1(M)$ we can
then apply the classification Theorem \ref{thm:torfreeclass}.

\begin{lemma}[$2$-connected twisted connected sums with torsion-free $H^{4}$]\hfill
\label{l:2connected:g2}
\begin{enumerate}
\item If $K_\pm = 0$ (\ie $H^2(V_\pm) \to H^2(S)$ is injective; recall
\eqref{eq:k_def}), $N_+ \cap N_- = 0$ and the inclusion
$N_+ + N_- \subset L$ is primitive then $M$ is 2-connected.
\item If $N_+ \perp N_-$, then $H^4(M)$ is torsion-free.
\end{enumerate}
\end{lemma}

\begin{proof}\hfill
\begin{enumerate}
\item We know from Theorem~\ref{thm:g2topology} that $\pi_1(M) = 0$.
Theorem \ref{thm:g2topology}\ref{it:h2m} implies that $H^2(M) = 0$
and Corollary \ref{cor:torsion}\ref{it:h3t} that $H^3(M)$ is torsion-free.
So $\pi_2(M) \cong H_2(M) = 0$.

\item Follows from \ref{cor:torsion}\ref{it:h4t}. \qedhere
\end{enumerate}
\end{proof}

The twisted connected sum construction relies on being able to find pairs of suitably compatible 
ACyl Calabi-Yau $3$-folds $V_{\pm}=Z_{\pm}\setminus S_{\pm}$. 
We will often refer to finding such compatible 
pairs as \emph{solving the matching problem}. 
We will see (\cf Proposition \ref{prop:orth_gluing}) that the easiest way to
find solutions to the matching problem involves
\begin{itemize}
\item
using building blocks of semi-Fano type  which automatically (%
Proposition \ref{prop:block_from_sf}) have $K = 0$;
\item
applying %
results of Nikulin \cite{nikulin:quadratic}
to embed the orthogonal direct sum $N_{+} \perp N_{-}$ primitively in the K3
lattice $L$ (``primitive perpendicular gluing'').
\end{itemize}
This will allow us to obtain a large class of examples of compact \gtmfd s
that are 2-connected and have $H^4(M)$ torsion-free.
When $K_\pm = 0$ and $N_+ \perp N_-$, Theorem \ref{thm:g2topology} implies that
$$b^{3}(M) = b^{4}(M) = b^{3}(Z_+) + b^{3}(Z_-) + 23.$$
So by Theorem \ref{thm:torfreeclass}, to understand the almost-diffeomorphism
type of such $M$ it remains only to determine the divisibility of $p_1(M)$.

\begin{remark*}
If $M$ is $2$-connected but $H^{4}(M)$ has torsion then we could still apply
the almost-diffeomorphism classification theory of Wilkens and Crowley as in
Remark \ref{R:h4:torsion}.
Recall from Remark \ref{rmk:tor_form} that the isomorphism class of the torsion-linking form of 
a twisted connected sum \gtwo-manifold is determined by the isomorphism class of $H^{4}(M)$.
Hence for $2$-connected twisted connected sums the 
isomorphism class of the pair $(H^4(M), p_1(M))$ is sufficient to determine the
almost-diffeomorphism class, except possibly when $H^{4}(M)$ has 2-torsion.
\end{remark*}

\begin{remark}
\label{R:joyce:g2}
Of all the \gtmfd s constructed by Joyce's orbifold desingularisation methods
\cite{joyce:g2, joyce:holonomybook} only one example has $b^{2}=0$;
in particular, none of the other Joyce \gtmfd s are $2$-connected.
Since the diffeomorphism classification of general simply-connected smooth
$7$-manifolds is still unsolved, the determination of the diffeomorphism type
of Joyce's \gtmfd s remains a challenge.
The example with $b^2 = 0$ is found in \cite[Thm 12.5.7]{joyce:holonomybook},
and has $b^{3}=215$. It is in fact topologically a twisted connected sum
of blocks of the type described in the following remark.
\end{remark}

\begin{remark}
\label{R:k3:involution:g2}
The non-symplectic type blocks described in Remark \ref{rmk:kl-def} 
always have $\rk K \ge 2$ in \eqref{eq:k_def}.
Hence by Theorem \ref{thm:g2topology}(ii)
any twisted connected sum \gtmfd{} $M$ constructed using at least one
such building block has $b^{2}(M)\ge 2$; in particular, the
diffeomorphism classification of such twisted connected sum \gtmfd s also remains
open. The non-trivial $K$ of these blocks arises from resolving singularities by
blow-ups; in some cases it is possible to desingularise by smoothing
instead to obtain blocks with $K = (0)$.
While the details are beyond the present scope, Joyce's example with $b^2 = 0$
can be seen to be recovered topologically by using such blocks
from K3s with non-symplectic involution with fixed lattice $U(2)$, \ie double
covers of $\PP^1 \times \PP^1$ branched over a smooth curve of bidegree $(4,4)$.
\end{remark}

Let $N'_\mp$ be the image of $N_\mp$ in $N_\pm^* = L/T_\pm \subset H^4(Z)$
as before.
From Proposition \ref{prop:p1y} and Lemma \ref{lem:ymv} we immediately deduce

\begin{corollary}
\label{cor:p1gcd}
Let $M$ be a twisted connected sum of the building blocks $Z_+$ and $Z_-$.
Then
\[ \gdiv{p_1(M)} = 2\gcd(c_2(Z_+) \mod N'_-, \; c_2(Z_-) \mod N'_+) . \]
In particular any common divisor of $2c_2(Z_+)$ and $2c_2(Z_-)$ also divides
$p_1(M)$, and if $N_+ \perp N_-$ then
\[ \gdiv{p_1(M)} = 2\gcd(c_2(Z_+), c_2(Z_-)) . \]
\end{corollary}

\noindent
Here the `greatest common divisor' of $c_2(Z_+)$ and $c_2(Z_-)$ should simply
be interpreted as the greatest integer by which both are divisible in the
respective $\ZZ$-modules $H^4(Z_\pm)$ (and $H^4(Z_\pm) \mod N'_\mp$).

For the building blocks used in this paper, we already computed the greatest
divisors of $c_2(Z)$ in \cite{chnp1}, see Table \ref{tableg:blocks}.
In examples of twisted connected sums where $N_+$ and $N_-$ are not
perpendicular (so that $N'_\pm$ are non-trivial), we need more detailed
information about $c_2(Z)$.
When $N'_\pm$ is primitive, corresponding to $H^4(M)$ being torsion-free,
\cite[Lemma 5.18]{chnp1} can be applied to give the information we
need for building blocks constructed from semi-Fanos using Proposition
\ref{prop:block_from_sf}. (In general, it is a little easier to compute
$\gdiv p_1(M)$ modulo the torsion in~$H^4(M)$.)

\begin{lemma}
\label{lem:p1range}
$\gdiv p_1(M) \in \{ 4, 8, 12, 16, 24, 48 \}$
for any twisted connected-sum \gtmfd{}~$M$.
\end{lemma}

\begin{proof}
Since $M$ is spin, $p_1(M)$ is divisible by 4 according to
Corollary \ref{cor:p_half} (we can also deduce this from
Corollary \ref{cor:p1gcd} and $c_2(Z)$ being even for any building
block $Z$ \cite[Lemma 5.10]{chnp1}).
On the other hand, $M$ contains a K3 surface $S$ with trivial normal bundle,
so the image of $p_1(M)$ in $H^4(S) \cong \bbz$ is 
$p_1(S) = -2c_2(S) \cong -2\chi(S) = -48$.
\end{proof}

\begin{remark*}
The examples in Table \ref{tab:rank1} show that the restrictions in
Lemma \ref{lem:p1range} are the only constraints on the possible greatest
divisors of $p_1$ of twisted connected sum \gtmfd s.
\end{remark*}

\begin{remark*}
Together with Remark \ref{R:homotopy_class}, the lemma above implies that a
pair of 2-connected twisted connected sum \gtmfd s with torsion-free $H^4$ are
almost-diffeomorphic if and only if they are homotopy equivalent.
\end{remark*}

\begin{corollary}
\label{c:inertia:g2}
For a $2$-connected twisted connected sum \gtmfd{} $M$ with $H^{4}(M)$
torsion-free either 
\begin{enumerate}
\item
The inertia group $I(M)=\Theta_{7}$ and hence the almost diffeomorphism class of $M$ consists of a single 
diffeomorphism class; this holds when $\gdiv{p_{1}(M)} \in \{4,8,12,24\}$; or 
\item
The inertia group $I(M)$ consists of all even elements in $\Theta_{7} \simeq \ZZ/28\ZZ$ and hence 
the almost diffeomorphism class of $M$ contains exactly two diffeomorphism classes; 
this holds when  $\gdiv{p_1(M)} \in \{16, 48\}$.
\end{enumerate}
In particular, knowing only $b^{4}(M)$ determines the diffeomorphism type of $M$ up to $8$ possibilities.
\end{corollary}

\begin{proof}
Follows immediately from  Theorem \ref{thm:inertia} and Lemma \ref{lem:p1range}.
\end{proof}

\begin{remark}
\label{rmk:EK}
Eells and Kuiper \cite{eells62} defined a $\cg{28}$ valued invariant for
(in particular) closed simply-connected spin 7-manifolds $M$ with $b^4(M) = 0$
(\ie $H^4(M)$ finite). This invariant classifies the elements of $\Theta_7$,
and can be used to detect the connected sum action of $\Theta_7$ and thus
distinguish between the diffeomorphism types in an almost-diffeomorphism class.
This invariant can be generalised to the case when $b^4(M) > 0$, in such a way
that it distinguishes between all smooth structures on $M$ when $M$ is
2-connected, but is more complicated to define and compute when $p_1(M)$ is
not a torsion class \cite{7class}.
\end{remark}

%% file: sec6.tex
Let $(M,g)$ be a Riemannian manifold. A $k$-form $\alpha$ on $M$ is 
a \textit{calibration} if $d\alpha=0$ and, for all $x\in M$ and every
oriented $k$-plane $\pi$ in $T_xM$, we have $\alpha_{|\pi}\leq \vol_\pi$. An
oriented submanifold $i:\asm\hookrightarrow M$ is
\textit{calibrated} if, for all $x\in\asm$, $\pi_x:=i_*(T_x\asm)$  
attains the equality:  $\alpha_{|\pi_x}= vol_{\pi_x}$.
The fundamental property of any calibrated submanifold is that 
it minimises volume in its homology class
\cite[Thm. II.4.2]{harveylawson:calgeometry}.

It follows from Lemma \ref{L:calibrations} %
that, on any \gtwo-manifold $(M,\varphi)$, the (parallel) $3$-form $\varphi$ is
a calibration. The corresponding calibrated $3$-dimensional submanifolds are
known as \textit{associative}.

In this section we explain that if the ACyl Calabi--Yau 3-folds $V_\pm$ used
in the twisted connected-sum construction of \gtmfd s $M$ described in
\S \ref{sec:twisted_kovalev} contain appropriate \emph{compact} calibrated
submanifolds, then these will give rise to associative submanifolds of $M$.
More precisely, if $C \subset V_\pm$ is a holomorphic curve, then
$\Sph^1 \times C$ is an associative in $\Sph^1 \times V_\pm$, and if
$L \subset V_\pm$ is special Lagrangian, then $\{\anglex\} \times L$ is
associative. We will prove that under certain conditions it is possible to
perturb these to manifolds that are associative with respect to the
torsion-free \gtstr{} on $M$, when the neck-length parameter in the
construction is sufficiently large.

\subsection*{Geometry of associative submanifolds}
\label{ss:ass_subs}
This subsection recalls basic features of the geometry of associative
submanifolds.
Let $\asm$ be an associative submanifold in a \gtwo-manifold $(M,\varphi)$.
Let $N\asm$ denote the normal bundle of $A$. 
\begin{lemma}\label{l:normal_bundle_trivial}
  The normal bundle $N\asm$ of an associative submanifold $\asm$
  is differentiably trivial.
\end{lemma}
\begin{proof}
Since $NA$ is a rank $4$ real vector bundle over a $3$-dimensional base 
it admits at least one global nowhere-vanishing section $\nv$.
Since $\asm$ is 3-dimensional and orientable its tangent bundle is differentiably
  trivial, \textit{i.e.} it admits three global linearly independent
  sections $e_i$. Lemma \ref{L:calibrations}
shows that the cross product on $M$ defines an operation
$T\asm\times N\asm \rightarrow N\asm$. Thus $\nv,
  e_1\times\nv,e_2\times\nv,e_3\times\nv$ are global
  linearly independent sections of $N\asm$.
\end{proof}
Let $\nabla$ denote the
Levi-Civita connection defined by the metric $g$ on $M$. Recall that for $x\in
\asm$ the projections $T_xM\rightarrow T_x\asm$ and
$T_xM\rightarrow N_x\asm$ corresponding to the orthogonal splitting
$T_xM=T_x\asm\oplus N_x\asm$ define connections on the bundles
$T\asm$, $N\asm$. When necessary we will distinguish these via the
notation $\nabla^\top$, $\nabla^\perp$. 

The cross product $T\asm \times N\asm \to N\asm$ gives the normal bundle
a Clifford bundle structure. Together with the connection $\nabla^\perp$
this defines a natural \emph{Dirac operator}
$\D:\Gamma(N\asm)\rightarrow \Gamma(N\asm)$.
For $\nv \in \Gamma(N\asm)$ we can express $\D \nv$ as follows.
For any $x\in\asm$ let $e_1,e_2,e_3$ denote a positive orthonormal basis
of $T_xM$, and let 
\begin{equation}
\label{eq:dirac}
\D\nv(x) := \sum_{i=1}^3 \, e_i \times (\nabla^\perp_{e_i}\nv) .
\end{equation}
$\D$ is a first order differential operator. One
can check that it is elliptic and formally self-adjoint, \textit{i.e.}
$$\int_\asm \inner{\D\nv,\nw} \, dvol = \int_\asm \inner{\nv, \D\nw} \, dvol.$$

\begin{remark*}
$\D$ is in fact a twisted Dirac operator, in the sense that
$N\asm \otimes_\bbr \bbc$ is isomorphic as a Clifford bundle to a twisted
spinor bundle $S \otimes E$. For the relation $T\asm \cong \Lambda^2_+N\asm$
implies that for any spin structure $P$ on $\asm$ (which exists because
$\asm$ is 3-dimensional and orientable) there is a lift of the
$\sorth4$-structure of $N\asm$ to a $\spin4$-structure, so that $P$ is
associated via the projection of $\spin4 \cong \spin3 \times \spin3$
to one factor. Then $N\asm \otimes_\bbr \bbc$ is the tensor product
of the two vector bundles associated to the spin representations of $\spin4$,
and one of these is the spinor bundle $S$ associated to $P$.
See McLean \cite[\S 5]{mclean} and Lawson-Michelsohn \cite{lawson89}.
\end{remark*}

The Dirac operator plays an important role in the deformation theory of
associative submanifolds.
Given the \gtstr{} $\varphi$, we can define a global vector-valued $3$-form
$\chi$ on $M$ modelled on \eqref{E:triple:cross}:
\[ g(u, \half \chi(v,w,z)) = \psi(u,v,w,z) \quad \textrm{for all }
u, v, w, z \in T_xM , \]
where $\psi = *\varphi$. Then
$\asm \subset M$ is associative if and only if the normal vector field
$F(\asm,\varphi) = \chi(T\asm) \in \Gamma(N\asm)$ vanishes, where $T\asm$ is
interpreted as a simple unit norm section of $\Lambda^3 TM$ over $\asm$.
Recall that we can parametrise the deformations of $\asm$ as follows. Let $exp$
denote the exponential map on $M$. Then all (small) deformations of $\asm$, up
to reparametrisation, can be obtained as $\asm_\nv = i_\nv(\asm)$
for some $\nv \in \Gamma(N\asm)$ close to the zero section, where
$ i_\nv:\asm\rightarrow M$ is defined by 
\[ \ i_\nv(x):=exp_x(\nv(x)).  \]
Given $\nv$, $F(\asm_\nv)$ defines a section of $N\asm_\nv$.
In other words, if we let $\caln$ be the vector bundle over $\Gamma(N\asm)$
whose fibre over $\nv$ is $\Gamma(N\asm_\nv)$, then $F$ is a
section of $\caln$.

The associative deformations of $\asm$ are parametrized by the zero set of
$F$ in a small neighbourhood $U$ of the zero section in $\Gamma(N\asm)$.
We say that $\asm$ is \textit{isolated} if $F^{-1}(0)=\{0\}$, \textit{i.e.}
if there do not exist other associative submanifolds attainable as small
deformations of $\asm$.

Because $F(0) = 0$, the differential $DF_0 : \Gamma(N\asm) \to \Gamma(N\asm)$
is defined naturally (without any connection on $\caln$), and it is precisely
equal to $\D$ (see \cite[\S 5]{mclean} or \cite[Theorem 2.1]{gayet10}).
We call the kernel of $\D$ the \emph{infinitesimal deformation space} of
$\asm$, and say that $\asm$ is \emph{rigid} if this space vanishes.

We could attempt to study the set $F^{-1}(0)$ via the Implicit Function Theorem.
It is first necessary to pass to the Banach space completions of the relevant
spaces and maps, \eg using Sobolev spaces. If $\asm$ is closed then the
standard theory of elliptic operators shows that $\D$ extends to a Fredholm
operator. It follows that if $\asm$ is rigid then it is also isolated.
As $\D$ is formally self-adjoint it has index 0, and the
\emph{obstruction space} $\coker \D$
vanishes if and only if $\asm$ is rigid. Therefore we can use the Implicit
Function Theorem to prove smoothness of the deformation space of a closed
associative only when the space is in fact discrete.

\subsection*{Persistence of associatives}

We prove that any rigid associative submanifold $A$ will persist under small
deformations of the ambient \gtstr. %

\begin{theorem}
\label{perturbthm}
Let $\asm$ be a closed associative in a \gtmfd{} $(M,\varphi)$.
If $\ker \D = 0$ then for any small deformation of the \gtstr{},
there is a unique small deformation of $\asm$
which is associative with respect to the new \gtstr.
\end{theorem}

When we apply Theorem \ref{perturbthm} we will often first replace $M$ by
an open neighbourhood of $A$ in order to avoid regions where the \gtstr{}
has torsion.
Even if the obstruction space $\ker \D$ is non-zero, $\asm$ may be
``unobstructed in a family''.
Infinitesimal deformations of the \gtstr{} on $M$ correspond simply to
3-forms, and so the derivative of $\varphi' \mapsto F(\asm,\varphi')$
at $\varphi$ is a map $\Omega^3(M) \to \Gamma(N\asm)$.
Let $R_{A,\varphi} : \Omega^3(M) \to \ker \D$ denote the composition with
the projection $\Gamma(N\asm) \to \coker \D \cong \ker \D$.

\begin{theorem}
\label{thm:genperturb}
Let $\asm$ be a closed associative in a \gtmfd{} $(M,\varphi)$,
and $\{\varphi_s : s \in \calg\}$ an \mbox{$m$-dimensional}
family of deformations of $\varphi$ such that
$R_{A,\varphi} : T\calg \to \ker \D$
is an isomorphism. %
Then there is a
ball $B \subset \bbr^m$, a family of perturbations $A_b$ of $A$ parametrised
by $b \in B$ (a smooth function $A \times B \to M$) and
$f : B \to \calg$, such that each $A_b$ is associative with respect to~$f(b)$.
The same conclusion holds with $\calg$ replaced by any sufficiently small
deformation to a family of \gtstr s $\calg'$.
\end{theorem}

The perturbation $A_b$ is rigid as an $f(b)$-associative unless $b$ is a
critical point of $f$.
Theorem \ref{perturbthm} is of course a special case of Theorem 
\ref{thm:genperturb}. It can be proved with less cumbersome notation.

\begin{proof}[Proof of Theorem \ref{perturbthm}]
As above, mapping $v \in \sob{k+1}(N\asm)$ to the image $A_v = i_v(A)$
identifies a neighbourhood $U$ of $\asm$ in the space of
$\sob{k+1}$-submanifolds of $M$ with a neighbourhood of the origin in
$\sob{k+1}(N\asm)$. Choose a trivialisation of the bundle $\caln$ over $U$,
\ie isomorphisms $\Gamma(NA_v) \cong \Gamma(NA)$ for each $v$.
Let $\{\varphi_t : t \in (-\epsilon,\epsilon)\}$ be a 1-parameter family of
\gtstr s (containing $\varphi = \varphi_0$). Consider the map
\[ U \times (-\epsilon,\epsilon) \to \sob{k}(N\asm),
\quad (\asm',t) \mapsto F(\asm', \varphi_t) . \]
By hypothesis, the derivative at $(\asm,0)$ is bijective on the first factor.
By the Implicit Function Theorem, a neighbourhood of $(\asm,0)$ in the
pre-image of $0$ is the graph of a function $t \mapsto \asm'(t)$, \ie for each
perturbation $\varphi_t$ of the \gtstr{} there is a unique
$\sob{k+1}$-perturbation $\asm_v$ of $\asm$ that is associative with respect
to $\varphi_t$.
Because the deformation operator $\D$ is elliptic, $v$ is a solution of a
non-linear elliptic equation, and is smooth by elliptic regularity.
\end{proof}

\begin{proof}[Proof of Theorem \ref{thm:genperturb}]
Let $\{\varphi_{s,t} : s \in \calg, t \in (-\epsilon,\epsilon)\}$ be
a one-parameter family of deformations of $\calg$
(with $\varphi_{s_0,0}$ corresponding to the initial \gtstr{} $\varphi$ on~$M$,
with respect to which $\asm$ is associative).
With $U$ as before, consider the map
\[ U \times \calg \times (-\epsilon,\epsilon) \to \sob{k}(N\asm),
\quad (\asm', s, t) \mapsto F(\asm', \varphi_{s,t}) \] 
The derivative
$T_\asm U \times T_{s_0}\calg \times \bbr \to \sob{k}(N\asm)$ at
$(\asm, s_0, 0)$ equals $\D$ on $T_\asm U = \sob{k+1}(N\asm)$, while the
composition of the derivative with the projection to $\coker \D$ equals
$R_{A,\varphi}$ on the $T_{s_0}\calg$ factor.
Hence the derivative is an isomorphism transverse to
$\ker \D \oplus \{0\} \oplus \bbr$. By the Implicit Function Theorem,
a neighbourhood of $(\asm, s_0,0)$ in the pre-image of 0 is
a graph over $B \times (-\epsilon',\epsilon')$, for some small ball
$B \subset \ker \D$. For each fixed
$t \in (-\epsilon',\epsilon')$, this defines a family of deformations
$\{A_b : b \in B\}$ and a map 
$f : B \to \calg' = \{\varphi_{s,t} : s \in \calg\}$.
\end{proof}

In the situation where we want to use the unobstructedness in a family,
there is an obvious family $\cala$ of initial associatives, and we can perturb
the whole family.

\begin{corollary}
\label{cor:genperturb}
Suppose that $\cala$ is a smooth compact (possibly with boundary)
$m$-dimensional family of closed associatives in a \gtmfd{}
$(M,\varphi)$, and that $\{\varphi_s : s \in \calg\}$ is an $m$-dimensional
family of deformations of $\varphi$ such that
$R_{\asm,\varphi} : T\calg \to \ker \D$
is surjective for each $\asm \in \cala$ (so $\dim \ker \D = m$).
Then for any sufficiently small deformation of $\calg$ to a family of
\gtstr s $\calg'$, there is a small deformation $\cala'$ of $\cala$ and a
smooth map $f : \cala' \to \calg'$
such that each $\asm' \in \cala'$ is associative with respect to $f(\asm')$.
\end{corollary}

\begin{proof}
For each $\asm \in \cala$, Theorem \ref{thm:genperturb} describes how to
deform a neighbourhood of $\asm$, provided that $\calg'$ is a sufficiently
small deformation of $\calg$. Because $\cala$ is compact it can be covered by
finitely many such neighbourhoods.
\end{proof}

When $\cala$ is without boundary, $f : \cala' \to \calg'$ will definitely have
some critical points, so some elements of $\cala'$ are not rigid.

\subsection*{Associative submanifolds and complex curves}
\label{ss:ass_vs_complex}
Let $(V,\Omega,\omega)$ be a Calabi--Yau 3-fold, and consider $\Sph^1 \times V$
with the torsion-free \gtstr{} $\varphi = d\anglex \wedge \omega + \Real \Omega$
as described in \eqref{eq:3form_vs_CY}. Let $C$ be a complex curve in $V$.
Then Lemma \ref{L:linear:ass_vs_complex}\ref{it:ass_vs_complex} implies that
$\Sph^1 \times C$ is an associative submanifold. The aim of this section is to
relate the properties of the associative submanifold to those of the complex
curve. 

Recall that a Calabi-Yau 3-fold $V$ carries a global holomorphic $(3,0)$-form
$\Omega$. We will denote its real part by $\alpha$ and its imaginary part by
$\beta$, \textit{i.e.} $\Omega=\alpha+i\beta$. We can define a cross product on
$TV$ via the formula
\begin{equation}\label{eq:CY_definition_cross}
g(a\times b,c) = \alpha(a,b,c),
\end{equation}
where $g$ is the Calabi-Yau metric of $V$.
Of course, this coincides with the projection onto $TV$ of the cross
product on $\Sph^1\times V$.
The fact that $\Omega$ is $I$-linear and $g$ is Hermitian implies that
the cross-product is $I$-antilinear. It has the usual property that
$a \times b$ is perpendicular to both $a$ and $b$. In particular, for
a complex curve $C\subset V$ the cross product gives a complex linear map
\begin{equation}
\label{eq:cx_cliff}
TC \times NC \rightarrow \overline{NC} .
\end{equation}
Moreover, because $\Omega$ is parallel
\begin{equation}\label{eq:CY_leibniz}
\nabla(a\times b)=\nabla a\times b+a\times \nabla b.
\end{equation}
Let us now review some well-known facts concerning holomorphic vector
fields.  Given any complex manifold $(V,I)$ with real
tangent bundle $TV$, recall the isomorphism of complex vector bundles $(TV,I)\cong T^{1,0}V$ given by 
\begin{equation}\label{eq:tangent_vs_holomorphic}
X\mapsto X-iIX.
\end{equation}
Recall also that any holomorphic bundle $E\mkern -6mu\rightarrow \mkern -6mu V$ has a natural
\textit{Cauchy-Riemann} operator $\debar:\Gamma(E)\rightarrow
\Omega^{0,1}(E)$ whose kernel consists of the holomorphic sections of $E$. A
Hermitian metric $h$ on $E$ defines a \textit{Chern connection}
$\widetilde{\nabla}:\Gamma(E)\rightarrow\Omega^1(E)$: it is uniquely
characterized by the properties $\widetilde{\nabla}h=0$ and
$\widetilde{\nabla}^{0,1}=\debar$, where
$\widetilde{\nabla}^{0,1}:=\tfrac{1}{2}(\widetilde{\nabla}+i\widetilde{\nabla}_I)$ is
the $(0,1)$-component under the splitting
$\Omega^1(E) = \Omega^{1,0}(E) \oplus \Omega^{0,1}(E)$.
Because $g$ is a Kähler metric on $V$, the Chern connection on $TV$ coincides
with the Levi-Civita connection $\nabla$. Hence the Chern connection on $NC$
coincides with the projection $\nabla^\perp$. In particular
the Cauchy-Riemann operator on $NC$ is just the $(0,1)$-part of $\nabla^\perp$.
We can use this fact and the complex Clifford structure \eqref{eq:cx_cliff}
to define an operator
\[ \D^c : \Gamma(NC) \to \Gamma(\overline{NC}) \]
whose kernel is exactly the space of holomorphic normal vector fields:
for $\nv\in\Gamma(N\asm)$ and $x\in C$ pick any unit vector
$a\in T_xC$ and set
\[ \D^c\nv(x):=a\times (\nabla_a+I\nabla_{Ia})^\perp \nv. \]
Since this is unchanged if we replace $a$ by $Ia$, it is in fact independent
of the choice of $a$. It defines a %
complex first-order linear elliptic operator, which
we will refer to as the \textit{complex Dirac operator} on $NC$.
Using \eqref{eq:CY_leibniz} and $I$-antilinearity of the cross-product we find
\[ \inner{\D^c\nv,\nw} = \mbox{div}_C(\nv\times\nw) + \inner{\nv,\D^c\nw} \]
where $\mbox{div}_C$ denotes the divergence operator on vector fields tangent
to $C$, defined via an orthonormal basis of $TC$ by
$\mbox{div}_C X=\inner{\nabla_{e_i}X,e_i}$.
Under integration, the divergence term vanishes, so 
$\D^c$ is formally self-adjoint:
\[ \int_C \inner{\D^c\nv,\nw} \, dvol = \int_C \inner{\nv,\D^c\nw} \, dvol . \]

Let us now return to the product \gtmfd{} $\Sph^1 \times V$ and the associative
submanifold $\Sph^1 \times C$. %
We can identify the normal
bundle of $\Sph^1\times C\subset \Sph^1\times V$ with the normal bundle of
$C\subset V$; notice however that any section $\nv$ will depend on both the
$\anglex$ variable and the variable on $C$. Choose a point
$(\anglex,x)\in\Sph^1\times V$. Set $e_1:=\contra{\anglex}$ and let $e_2=a$ be any
unit vector on $T_xC$ so that $e_3=Ia$. Then \eqref{eq:dirac} becomes
$$\D\nv=\contra{\anglex}\times\dot{\nv}+a\times(\nabla_a\nv)^\perp+Ia\times(\nabla_{Ia}\nv)^\perp.$$
where $\dot{\nv}$ denotes the derivative with respect to $\anglex$. As seen
at \eqref{eq:3form_vs_CY}, $\contra{\anglex}\times \dot{\nv}=I\dot{\nv}$.
Using that $\nabla^\perp$ is $I$-linear and the cross product is $I$-antilinear 
we can then rewrite $\D\nv$ as follows:
\begin{equation}
\label{eq:dirac_product}
\D\nv = I\dot{\nv}+a\times(\nabla_a\nv+\nabla_{Ia}I\nv)^\perp =
I\dot{\nv} + \D^c\nv.
\end{equation}
Normal holomorphic vector fields represent the infinitesimal deformations of
$C$ as a complex curve in $V$. The curve $C$ is said to be \textit{rigid} if
it has no infinitesimal holomorphic deformations. In the previous section we
saw that the solutions to $\D\nv=0$ (for the Dirac operator defined in
\eqref{eq:dirac}) correspond to the infinitesimal (associative) deformations
of an associative submanifold of a \gtmfd. 

\begin{lemma}\label{l:rigidity}
For the associative submanifold $\Sph^1\times C \subset \Sph^1 \times V$,
the kernel of $\D$ is the pull-back of the kernel of $\D^c$.
Thus $\Sph^1\times C$ is rigid if and only if the complex curve $C$ is rigid.
\end{lemma}

\begin{proof}
The facts that $\D^c$ is $I$-antilinear and formally self-adjoint and
that $\nv \mapsto I\dot{\nv}$ is a composition of two skew-adjoint maps imply
\[ \linner{\D^c \nv, I\dot{\nv}} = 0 . \]
Therefore \eqref{eq:dirac_product} implies that
$\lnorm{\D \nv}^2 = \lnorm{\D^c \nv}^2 + \lnorm{\dot{\nv}}^2$ .
\qedhere
\end{proof}

\subsection*{Associative submanifolds and special Lagrangians}
\label{ss:ass_vs_sl}
Let $(V,\Omega,\omega)$ be a Calabi--Yau \mbox{3-fold}, and consider as before
$\Sph^1 \times V$ with the torsion-free \gtstr{}
$\varphi = d\anglex \wedge \omega + \Real \Omega$ described in
\eqref{eq:3form_vs_CY}.
If $L \subset V$ is a special Lagrangian 3-fold then Lemma
\ref{L:linear:ass_vs_complex}\ref{it:ass_vs_sl} implies that
$L_\anglex = \{\anglex\} \times L$ is associative in $\Sph^1 \times V$ for any
$\anglex \in \Sph^1$. We assume that $L$ is closed.

We want to describe the relation between the deformation theory of the
associative $L_\anglex$ %
and the special Lagrangian $L$. Note that since we can deform $L_\anglex$
simply by changing $\anglex \in \Sph^1$, it is \emph{never} rigid, and the
obstruction space $\coker \D$ is always non-trivial. We will therefore study
the map $\Omega^3(\Sph^1 \times V) \to \coker \D$ in order to apply Theorem
\ref{thm:genperturb} later.

Let us first recall the deformation theory of a closed special Lagrangian
$L \subset V$ \cite[\S 3]{mclean}. According to
Lemma \ref{L:sl_calibrations}\ref{it:sl_char}, for $L$ to be special Lagrangian
is equivalent to $\omega_{|L} = \Imag\Omega_{|L} = 0$.
The Lagrangian condition implies that we can identify the
normal bundle $NL$ with $T^*L$ by $\sigma \mapsto \sigma \lrcorner \omega$.
We can therefore parametrise small deformations of $L$ by small
$\alpha \in \Omega^1(L)$.

Since $\omega$ and $\Imag \Omega$ are closed, the cohomology classes
represented by their restrictions to $L$ are homotopy invariant, so the
restrictions are exact for all deformations of $L$.
The special Lagrangian deformations of $L$ are therefore
parametrised by the zero set of a map
\[ \Omega^1(L) \to d\Omega^1(L) \times d\Omega^2(L) . \]
The linearisation of this map at $0$ (corresponding to $L$) is
$D_L : \alpha \mapsto (d\alpha, d{*}\alpha)$.
This is surjective, with kernel $\harm^1(L)$, the space of harmonic 1-forms
on $L$. Thus the deformations of $L$ are always unobstructed, and form a smooth
manifold near $L$ of dimension $b^1(L)$.

Now consider the associative $L_\anglex = \{\anglex\} \times L$. Its normal
bundle $NL_\anglex$ in $\Sph^1 \times V$ is a direct sum of the trivial bundle
spanned by $\contra{\anglex}$ and the normal bundle $NL$ of $L$ in $V$.
We can identify it with $\Lambda^0 T^*L \oplus \Lambda^1 T^*L$.
Then the Dirac operator
$\D : \Gamma(NL_\anglex) \to \Gamma(NL_\anglex)$ is interpreted as %
\begin{equation}
\begin{aligned}
\label{eq:sl_dirac}
\Omega^0(L) \times \Omega^1(L) &\to \Omega^0(L) \times \Omega^1(L) , \\
(f, \alpha) & \mapsto (d^*\alpha, \, df + *d\alpha)
\end{aligned}
\end{equation}
(see Gayet \cite[Proposition 4.7]{gayet10}). The kernel consists of the
harmonic forms.  In particular, the infinitesimal deformation space of
$L_\anglex$ consists of the infinitesimal special Lagrangian deformations of
$L$ in $V$ together with translations of $\anglex$.
(Note that on the second factor, \eqref{eq:sl_dirac} equals $*D_L$, which is of
course consistent with the fact that $L_\theta$ is associative if and only if
$L$ is special Lagrangian.) 

\begin{lemma}
Let $L \subset V$ be a closed special Lagrangian submanifold. For the
associative submanifold $L_\anglex \subset \Sph^1 \times V$, the kernel of $\D$
is the direct sum of the kernel of $D_L$ and the span of $\contra{\anglex}$.
\end{lemma}

Now we study the map from infinitesimal deformations of the \gtstr{},
parametrised by $\Omega^3(\Sph^1 \times V)$, to the obstruction space
$\coker \D$. In the identification of $\D$ with \eqref{eq:sl_dirac},
$\coker \D$ corresponds to $\harm^0(L) \oplus \harm^1(L)$.
The map from $\Omega^3(\Sph^1 \times V)$ to $\coker \D$ is the composition
of a point-wise map
$\Lambda^3T_x^*(\Sph^1 \times V) \to \Lambda^0 T_x^*L \oplus \Lambda^1 T_x^*L$
and the projection to the harmonic forms.
We are interested primarily in torsion-free deformations of
$\Sph^1 \times V$, and (at least for $V$ compact/ACyl with $b^1(V) = 0$) up to
diffeomorphism and rescaling of the $\Sph^1$ factor they are all products.

\begin{lemma}
\label{lem:sl_family}
Let $(\sigma, \tau)$ be an infinitesimal deformation of the
$\sunitary3$-structure $(\Omega, \omega)$, and
let $\varphi_t$ be a 1-parameter family of \gtstr s with
$\frac{d\varphi_t}{dt} = d\anglex \wedge \tau + \Real\sigma$.
Then $\frac{d}{dt}F(L_\anglex,\varphi_t)_{|t=0} \in \Gamma(NL_\anglex)
\cong \Omega^0(L) \times \Omega^1(L)$
corresponds to $(*(\Imag\sigma_{|L}), *(\tau_{|L}))$, and the image in
$\coker \D \cong \harm^0(L) \oplus \harm^1(L)$ to the de Rham projection.
\end{lemma}

\begin{proof}
$\frac{d}{dt}F(L_\anglex,\varphi_t)$ is a linear function of
$\frac{d\psi_t}{dt}$.  Without loss of generality
$\varphi_t = d\anglex \wedge \omega_t + \Real \Omega_t$
where $(\Omega_t, \omega_t)$ is a deformation of $(\Omega, \omega)$
tangent to $(\sigma, \tau)$.
Then \eqref{eq:4form:six:dims} gives
$\frac{d\psi_t}{dt} = \omega \wedge \tau - d\anglex \wedge \Imag \sigma$,
from which we can deduce the result.
\end{proof}

In particular, consider the case when $L$ is a rational homology 3-sphere,
\ie $b^1(L) = 0$, so that $L$ is rigid as a special Lagrangian.
If $(\Omega_t,\omega_t)$ is a 1-parameter family of deformations of the
Calabi--Yau structure on $V$ and $\int_L \frac{d\Imag\Omega_t}{dt} \not= 0$,
then the $\Sph^1$-family of associatives $\{L_\anglex : \anglex \in \Sph^1\}$
is unobstructed with respect to the 1-parameter family
$\varphi_t = d\anglex \wedge \omega_t + \Real\Omega_t$, in the sense of
Corollary \ref{cor:genperturb}.

\subsection{Associatives in twisted connected sums}

We now put together the results of the section to identify the data we can
use to construct associatives in twisted connected-sum \gtmfd s.
As in Theorem \ref{thm:g2glue}, let $(V_\pm,\omega_\pm,\Omega_\pm)$ be two
asymptotically cylindrical Calabi-Yau 3-folds with asymptotic ends of the
form $\bbrp \times \Sph^1 \times S_\pm$ for a pair of \hk K3 surfaces $S_\pm$,
and $\hkr : S_+ \to S_-$ a \hk rotation. Let $M_\hkr$ be the twisted connected
sum of $\Sph^1 \times V_\pm$, and $\tv_{T,\hkr}$ the torsion-free \gtstr{} with
`neck length' $2T$ defined in Theorem \ref{thm:g2glue}.

\begin{prop}
\label{prop:cx_to_assoc}
Let $C \subset V_+$ be a closed rigid holomorphic curve. Then for sufficiently
large~$T$, there is a small deformation of the image of
$\Sph^1 \times C \subset \Sph^1 \times V_+$ in $M_\hkr$ that is associative
with respect to $\tv_{T,\hkr}$, and this associative is rigid.
\end{prop}

\begin{proof}
By Lemma \ref{l:rigidity}, $\Sph^1 \times C \subset \Sph^1 \times V_+$
is a rigid associative.

Recall that the \gtstr{} $\varphi_{T,\hkr}$ with small torsion defined before
Theorem \ref{thm:g2glue} is exactly the product \gtstr{} on the complement
of $\{t > T\}$ in $\Sph^1 \times V_+$, and hence near $\Sph^1 \times C$ when
$T$ is large. The $C^k$ norms of the difference between $\varphi_{T,\hkr}$ and
$\tv_{T,\hkr}$ are of order $O(e^{-\lambda t})$, so Theorem \ref{perturbthm}
implies that $\Sph^1 \times C$ can be perturbed to an associative with respect
to $\tv_{T,\hkr}$ for any sufficiently large $T$. 
\end{proof}

Constructing associatives from closed special Lagrangians $L \subset V_\pm$
requires a little bit more work since, as pointed out above, the associatives
$L_\anglex = \{\anglex\} \times L \subset \Sph^1 \times V_\pm$ are never rigid.
We restrict our attention to the case when $b^1(L) = 0$, so that $L$
is rigid as a special Lagrangian and the obstruction space of $L_\anglex$
is 1-dimensional. We can find a 1-parameter family of torsion-free
deformations of $\Sph^1 \times V_\pm$ such that the family
$\{ L_\anglex : \anglex \in S^1 \}$ is unobstructed (in the sense required
to apply Corollary \ref{cor:genperturb}) if the class of
$L$ in the homology of $V_\pm$ relative to its boundary is non-zero.
(In this section, all homology and cohomology refers to $\bbr$ coefficients.)

\begin{lemma}
Let $L^m$ be a closed submanifold of an ACyl manifold $V^n$ with
cross-section $X$.
If $[L] \in H_m(V,X)$ is non-zero then there is an exponentially decaying
harmonic $m$-form $\beta$ on $V$ such that $\int_L \beta \not= 0$.
\end{lemma}

\begin{proof}
The image of the Poincar\'e dual of $L$ under $H^{n-m}_{cpt}(V) \to H^{n-m}(V)$
is non-zero, and represented by an exponentially decaying harmonic form
$\alpha$ (Lockhart \cite[Theorems 7.6 \&~7.9]{lockhart87}).
Take $\beta = *\alpha$.
\end{proof}

\begin{corollary}
Let $L \subset V_+$ be a compact special Lagrangian with $b^1(L) = 0$,
such that $[L] \not= 0 \in H_3(V_+, \, \Sph^1 {\times} S_+)$. Then there is a
1-parameter family of deformations $\varphi_{t,+}$ of the product \gtstr{} on
$\Sph^1 \times V_+$, all with the same asymptotic limit as $\varphi_+$,
with respect to which $\{ L_\anglex : \anglex \in \Sph^1 \}$ is unobstructed
in the sense of Corollary \ref{cor:genperturb}.
\end{corollary}

\begin{proof}
Take $\beta \in \Omega^3(V_+)$ as in the previous lemma.
There is a unique complex 3-form $\sigma$ on $V_+$ such that
$\imag \sigma = \beta$ and $\sigma$ is an infinitesimal deformation
of $\Omega$ as an $\SL 3$-structure (\cf Remark \ref{r:SL3_open_orbit}).
Because the map $\beta \mapsto \sigma$ is $\sunitary 3$-equivariant it
maps harmonic forms to harmonic forms. Because $b^1(V_+) = 0$,
$(\sigma,0)$ is an infinitesimal deformation of $(\Omega, \omega)$
as an $\sunitary 3$-structure. 

$\real \sigma$ is an infinitesimal deformation of the product \gtstr{}
$\real \Omega + d\anglex \wedge \omega$, and because it is harmonic it
can be integrated to a 1-parameter family of torsion-free deformations
$\varphi_{t,+}$ \cite[Proposition 6.18]{jn1}. It follows from Lemma
\ref{lem:sl_family} that $\{ L_\anglex : \anglex \in \Sph^1 \}$ is unobstructed
in this family.
\end{proof}

Since each $\varphi_{t,+}$ ($t \in [-\epsilon,\epsilon]$) has the same asymptotic limit as $\varphi_+$, the \hk rotation $\hkr$ matches $\varphi_{t,+}$ and
$\varphi_{-}$. Thus for $T$ sufficiently
large we can define a 1-parameter family of torsion-free
\gtstr s $\{\tv_{t, T, \hkr} : t \in [-\epsilon,\epsilon]\}$.
Corollary \ref{cor:genperturb} implies

\begin{prop}
\label{prop:lag_to_assoc}
Let $L \subset V_+$ be a compact special Lagrangian with $b^1(L) = 0$,
such that $[L] \not= 0 \in H_3(V_+, \, \Sph^1 {\times} S_+)$. Then for $T$
large enough there is a smooth map $f : \Sph^1 \to [-\epsilon,\epsilon]$
and a deformation $\{ L'_\anglex : \anglex \in \Sph^1 \}$ in $M_\hkr$ such that
each $L'_\anglex$ is associative with respect to $\tv_{f(\anglex), T, \hkr}$.
\end{prop}

\noindent
$f$ has at least 2 critical points, which correspond to associatives that are
not rigid.

\begin{remark}
\label{rmk:lag_homology}
If $V_+$ is compactifiable in the sense that $V_+ \cong Z_+ {\setminus} S_+$
for a K3 divisor $S_+$ with trivial normal bundle in a compact complex manifold
$Z_+$, then $H_3(Z_+) \into H_3(Z_+, \, \Delta {\times} S_+)
\cong H_3(V_+, \, \Sph^1 {\times} S_+)$ since $H_3(S_+) = 0$.
If $Z_+$ is in turn a blow-up of a weak Fano $Y$, then the preimage of any
closed homologically non-trivial $L \subset Y$ not meeting the blow-up locus or
$S_+$ will represent a non-trivial class in $H_3(V_+, \, \Sph^1 {\times} S_+)$.
\end{remark}

%% file: sec_k3.tex
Recall that Theorem \ref{thm:g2glue} allows us to form a twisted
connected sum \gtmfd{} from any pair of ACyl Calabi-Yau 3-folds
$V_{\pm}$ satisfying a compatibility condition on their asymptotic \hk
K3 surfaces $S_{\pm}$.  In \cite{chnp1}, we constructed large numbers
of suitable ACyl Calabi-Yau 3-folds, applying Theorem
\ref{thm:acyl_limits}---the ACyl version of the Calabi-Yau
theorem---to building blocks $Z_{\pm}$ obtained from semi-Fano
$3$-folds as in Proposition \ref{prop:block_from_sf}.  In fact, as we
already remarked by varying various choices made in the construction
we obtain families of ACyl Calabi-Yau structures on the same
underlying smooth $6$-manifold $Z\setminus S$.  To complete the
construction of \gtmfd s, it remains to explain how to find
\emph{compatible pairs} of such ACyl Calabi-Yau 3-folds; this will
require us to exploit the freedom we have to vary the ACyl Calabi-Yau
structures on both building blocks.

We reformulate the compatibility condition in terms of existence of
``matching data'' between a pair of building blocks, which are certain
triples of cohomology classes in $L_\bbr \cong H^{2}(S; \bbr)$.
The definition of the matching data is linked to the moduli theory of
algebraic K3 surfaces. This formulation will help us prove the
existence of many pairs of compatible ACyl Calabi-Yau 3-folds given
some additional algebraic geometry input.
We remark at the outset that the same pair of deformation families of
building blocks $\fbb_{\pm}$ may be matched in different ways and
hence give rise to several different twisted connected sum
\gtwo-manifolds.

In this section we describe one convenient strategy for finding
matching data which we term ``orthogonal gluing''.  Given some
additional input about the deformation theory of the building blocks
used to construct the ACyl Calabi-Yau $3$-folds, orthogonal gluing
allows us to reduce the problem of finding compatible pairs of ACyl
Calabi-Yau $3$-folds $V_{\pm}=Z_{\pm}\setminus S_{\pm}$ almost
entirely to arithmetic questions about the pair of polarising lattices
$N_{\pm}$ of the building blocks $Z_{\pm}$.  For ACyl Calabi-Yau
$3$-folds of semi-Fano type the deformation theory we need was
developed in \cite[\S 6]{chnp1}.  In \S \ref{sec_g2mfds}
we use orthogonal gluing to find many compatible pairs of ACyl
Calabi-Yau $3$-folds.

At the end of \S \ref{sec_g2mfds} we also discuss so-called ``handcrafted nonorthogonal gluing''. 
This allows matching in situations impossible to achieve using orthogonal gluing; 
the price ones pays is that the method is much more labour-intensive as it 
requires more precise information about K3 moduli spaces.

\subsection{Reformulating the existence of \hk rotations}
\label{subsec:k3_matching}

Let us first recall the set-up for the gluing Theorem \ref{thm:g2glue}.
$V_\pm$ is a pair of ACyl Calabi-Yau 3-folds with asymptotic limits
$\bbrp \times \Sph^1 \times S_\pm$. The $S_\pm$ are K3 surfaces, with preferred
complex structure $I_\pm$, Kähler form $\omega^I_\pm$ and holomorphic
volume form $\Omega_\pm$. Because this is a \hk structure, there are further
complex structures $J_\pm$ and $K_\pm$, with Kähler forms $\omega^J_\pm$ and
$\omega^K_\pm$ ($\Omega_\pm = \omega^J_\pm + i \omega^K_\pm$).
The compatibility condition for $V_+$ and $V_-$ is that $S_\pm$ are related
by a \emph{\hk rotation} as in Definition \ref{def:hkrot}:
we need an orientation-preserving isometry $\hkr\colon S_+ \to S_-$ such
that $\hkr^\ast(I_-)=J_+$ and $\hkr^\ast(J_-)=I_+$ (with the isometry
condition this implies $\hkr^\ast (K_-)=-K_+$).
Equivalently, $\hkr^*\omega^I_- = \omega^J_+$, 
$\hkr^*\omega^J_- = \omega^I_+$ and $\hkr^*\omega^K_- = -\omega^K_+$.

We use the Torelli theorem to reduce this relation to the action on cohomology. 

\begin{lemma}
Let $\hdg : H^2(S_-; \ZZ) \to H^2(S_+; \ZZ)$ be an isometry, extend it to
$H^2(S_-; \R) \to H^2(S_+; \R)$, and suppose that
\[ \hdg[\omega^I_-] = [\omega^J_+], \quad \hdg[\omega^J_-] = [\omega^I_+]
\quad \text{and} \quad \hdg[\omega^K_-] = -[\omega^K_+]. \]
Then there is a \hk rotation
$\hkr : S_+ \to S_-$ such that $\hkr^* = \hdg$.
\end{lemma}

\begin{proof}
Consider the complex structure $J_-$ on $S_-$. $\omega^I_- - i\omega^K_-$
is a holomorphic 2-form with respect to $J_-$. Therefore $\hdg$ maps
$H^{2,0}(S_-,J_-)$ to $H^{2,0}(S_+, I_+)$, \ie it is a Hodge isometry between
the complex K3 surfaces $(S_-, J_-)$ and $(S_+, I_+)$. Moreover, the Kähler
class $[\omega^J_-]$ is mapped to the Kähler class $[\omega^I_+]$.
Therefore the strong Torelli theorem implies that there is a holomorphic map
$\hkr \colon (S_+,I_+) \rightarrow (S_-,J_-)$ such that $\hkr^\ast = \hdg$.
Since the holomorphic 2-forms are uniquely determined by their de Rham
cohomology classes, $\hkr^*\omega^I_- = \omega^J_+$ and
$\hkr^*\omega^K_- = -\omega^K_+$.
Further $\hkr^*\omega^J_- = \omega^I_+$, by uniqueness of a Ricci-flat
Kähler metric in its cohomology class.
Thus $\hkr$ is a \hk rotation.
\end{proof}

It is useful to rephrase the previous lemma in the language of the moduli theory
of K3 surfaces. Recall that a \emph{marking} of a complex K3 surface $(S, I)$
is an isometry $L \cong H^2(S; \ZZ)$.
$H^{2,0}(S) \subset H^2(S;\bbc)$  can be identified with an oriented real
2-plane in $H^2(S;\bbr)$, and its image in $L_\bbr$ is the \emph{period}
of the marked K3 surface.

\begin{prop}
\label{pro:matching_k3}
Let $(k_0, k_+, k_-)$ be an orthonormal triple of positive classes in $L_\bbr$.
Let $(S_\pm, I_\pm)$ be complex K3 surfaces with markings
$\hdg_\pm : L \to H^2(S_\pm; \ZZ)$ such that $\gen{k_\mp, \pm k_0}$ is the
period point, and $\hdg_\pm(k_\pm)$ is a Kähler class on $S_\pm$.
Let $\hdg = \hdg_+ \circ \hdg_-^{-1} : H^2(S_-; \ZZ) \to H^2(S_+; \ZZ)$.
Then there exist unique \hk structures
$(\omega^I_\pm, \omega^J_\pm, \omega^K_\pm)$ on $S_\pm$ with
$[\omega^I_\pm] = \hdg_\pm(k_\pm)$ and $\omega^J_\pm +i\omega^K_\pm$
holomorphic with respect to $I_\pm$, such that there is a \hk rotation
$\hkr : S_+ \to S_-$ with $\hkr^\ast = \hdg$. 
\end{prop}

\begin{proof}
The Kähler class $\hdg_\pm(k_\pm)$ contains a unique Ricci-flat Kähler metric
$\omega^I_\pm$. Up to complex scalar multiplication, there is a unique
2-form $\omega^J_\pm + i \omega^K_\pm$ on $S_\pm$ that is holomorphic
with respect to $I_\pm$. Since $\gen{k_\mp, \pm k_0}$ is the period with
respect to the marking $\hdg_\pm$, we can normalise it so that
$[\omega^J_\pm] = \hdg_\pm(k_\mp)$ and $[\omega^K_\pm] = \hdg_\pm(\pm k_0)$.
Then $(\omega^I_\pm, \omega^J_\pm, \omega^k_\pm)$ are \hk structures.
This choice of normalisation is the only one for which
$\hdg[\omega^I_-] = [\omega^J_+]$, $\hdg[\omega^J_-] = [\omega^I_+]$ and
$\hdg[\omega^K_-] = -[\omega^K_+]$, which is equivalent to the existence 
of a \hk rotation with $\hkr^\ast = \hdg$.
\end{proof}

\subsection{Matching data for pairs of building blocks}
The asymptotic \hk K3 surfaces $S_{\pm}$ of our ACyl Calabi-Yau $3$-folds $V_{\pm}$
come with a preferred complex structure $I_{\pm}$ and Kähler form
$\omega^{I}_{\pm}$ defined by the asymptotic limit of the Calabi-Yau structure.
We need to take this fact into account when we attempt to construct \hk
rotations between $S_{+}$ and $S_{-}$.

\begin{definition}
\label{d:matching:data}
A set of \emph{matching data for a pair of building blocks $(Z_\pm, S_\pm)$} is
a triple $(k_+, k_-, k_0)$ of classes in $L_\bbr$ for which there are markings
$\hdg_\pm : L \to H^2(S_\pm; \ZZ)$ such that $\gen{k_\mp, \pm k_0}$ is the
period point of the marked K3 $(S_\pm, I_\pm, \hdg_\pm)$, and
$\hdg_\pm(k_\pm)$ is the restriction to $S_\pm$ of a Kähler class on $Z_\pm$.
\end{definition}

With this terminology, the following is an immediate consequence of Theorem
\ref{thm:acyl_limits} and Proposition \ref{pro:matching_k3}.

\begin{corollary}
\label{cor:matching}
If there is matching data for the pair of building blocks $(Z_\pm, S_\pm)$
then $V_\pm = Z_\pm \setminus S_\pm$ admit compatible ACyl Calabi-Yau structures, 
\ie there exists a \hk rotation $\hkr \colon S_{+} \to S_{-}$.
Hence the twisted connected sum $7$-manifold 
$M_{\hkr}$ formed by gluing $\Sph^{1} \times V_{\pm}$ using the 
diffeomorphism specified in \eqref{eq:gluemap} admits \gtwo-metrics 
as described in Theorem \ref{thm:g2glue}.
\end{corollary}

The markings $\hdg_\pm$ in Definition \ref{d:matching:data} are not necessarily
unique. Hence nor is $\hdg = \hdg_+ \circ \hdg_-$ or the \hk rotation $\hkr$,
and the choices can affect the topology of the \gtmfd{} $M_\hkr$. We can
relate this to the computations of \S \ref{sec:top} as follows.

Recall that it is part of Definition \ref{dfng:BLOCK} that if $(Z,S)$
is a building block and $H^2(S; \ZZ) \cong L$ is a marking, then the
image $N \subset L$ of $H^2(Z; \ZZ)$ is primitive.  According to Lemma
\ref{lem:zh20}, $N \subseteq \Pic S$. This means that $S$ is a marked
\emph{$N$-polarised K3 surface}. Since $H^{2,0}(S)$ is perpendicular
to $H^{1,1}(S)$, the period of such a marked $N$-polarised K3 surface
is perpendicular to $N$; it must lie in the \emph{Griffiths domain}
$D_N$ of oriented positive 2-planes $\Pi \subset N^\perp \subset L_\bbr$.
$D_N$ can be considered as an open subset (determined by the positivity) of
$\bbp(\Lambda^2 (N^\perp \otimes \bbc))$, and hence as a complex manifold.
For a deformation family $\fbb$ of building blocks,
all members have the same polarising lattice $N$ and the primitive
embedding $N \into L$ is well-defined up to the action of $\orth{L}$.

Let $(k_+, k_-, k_0)$ be a set of matching data for a pair of blocks
$(Z_\pm, S_\pm)$. A choice of markings $h_\pm : L \to H^2(S_\pm; \ZZ)$ in
Definition \ref{d:matching:data}
determines embeddings $N_\pm \into L$ of the polarising lattices. While each
embedding is unique up to the action of $\orth{L}$, the \emph{pair} is not,
\eg $N_+ \cap N_-$ could vary. Since $N_+ \cap N_-$ is a summand in $H^2(M; \ZZ)$, 
the choice of markings can affect the topology of the \gtmfd{} produced in
Corollary \ref{cor:matching}. We say that the matching data
is \emph{adapted} to a given pair of embeddings $N_\pm \into L$ if
the markings $h_\pm$ can be taken to be $N_\pm$-polarised.
A necessary condition is that $\gen{k_\mp, \pm k_0} \perp N_\pm$.

Given a semi-Fano 3-fold $Y$, we can blow up to get a building block
$(Z,S)$ according to Proposition \ref{prop:block_from_sf}. In this
case the polarising lattice $N \subset L$ of the block is simply given
by the primitive embedding $H^2(Y; \ZZ)\to H^2(S; \ZZ)\cong L$, isometric
with respect to the form $(-K_Y) \cdot D_{1}\cdot D_{2}$ on
$H^2(Y; \ZZ)$. By deforming the semi-Fano $3$-fold $Y$
and varying the choice of smooth section $S \in \abs{-K_{Y}}$ we obtain a
family of building blocks $\fbb$, all with
the same topology and polarising lattice $N \cong H^{2}(Y; \ZZ)$.
In this way we can obtain ACyl Calabi-Yau manifolds
$V = Z \setminus S$ with potentially different asymptotic \hk K3 surfaces~$S$.

Given a pair $\fbb_\pm$ of
such families of building blocks, we can now approach the problem of
using them to construct a compact \gtmfd{} as follows.

\begin{enumerate}[label=\arabic*.]
\item Choose embeddings $N_\pm \into L$ in the $\orth{L}$-orbits
of primitive isometric embeddings determined by $\fbb_\pm$.
Let $T_\pm = N_\pm^\perp \subset L$ denote the orthogonal complements.
\item \label{it:match}
Consider triples $(k_+, k_-, k_0)$ %
such that
$k_\pm \in N_\pm(\bbr) \cap T_\mp(\bbr)$ and $k_0 \in T_+(\bbr) \cap T_-(\bbr)$.
Then $\gen{k_\mp, \pm k_0}$ lives in the Griffiths period domain $D_{N_\pm}$
for $N_\pm$-polarised K3 surfaces.
Find a triple that forms matching data, adapted to the given embeddings
$N_\pm \into L$, %
for some $(Z_\pm, S_\pm) \in \fbb_\pm$.
\item Apply Corollary \ref{cor:matching} to construct matching ACyl Calabi-Yau
structures on $V_\pm = Z_\pm {\setminus} S_\pm$.
\item Apply Theorem \ref{thm:g2glue} to glue $\Sph^1 \times V_\pm$ to
a compact \gtmfd{} $M$.
\end{enumerate}

\begin{remark*}
In a sense this scheme reverse engineers the process described in~\S\ref{sec:twisted_kovalev}: 
in effect, we first identify what \hk K3 to aim for, and then construct ACyl
Calabi-Yau 3-folds with that asymptotic K3 (up to \hk rotation).
\end{remark*}

We can then use the results in \S \ref{sec:top} to compute topological
invariants of $M$. Note that the cohomology of $M$ depends only on the
cohomology of the building blocks and on the choice of the pair of embeddings
$N_\pm \into L$. %
In many cases we can determine the diffeomorphism type of $M$.  
Also, if either building block
$(Z_\pm, S_\pm)$ contains rigid complex curves (\eg if $Z$ is a building block of semi-Fano type 
where the semi-Fano $Y$ is obtained as a small resolution of a nodal Fano) then Proposition \ref{prop:cx_to_assoc}
shows that $M$ contains corresponding rigid associative submanifolds.

So the key problem that remains to be addressed is to find the matching data 
in Step 2.
This is a difficult problem in general and in most cases we do not currently 
understand all possible ways to match a given pair of deformation families 
of building blocks $\fbb_{\pm}$.
In this section we describe a general method which we call \emph{orthogonal gluing} 
that yields large numbers of matching ACyl Calabi-Yau structures. 

\subsection{Orthogonal gluing}
\label{sec:orthogonal-gluing}

Let $\fbb_\pm$ be a pair of families of building blocks, obtained from 
semi-Fano 3-folds $Y_\pm$ of given deformation types $\sff_\pm$.
We describe a method that provides matching data for a large class of
such pairs.
\begin{itemize}
\item Choose the pair of primitive embeddings $N_\pm \into L$ so that
$N_+$ and $N_-$ \emph{intersect orthogonally}, \ie $N_\pm(\bbr) =
(N_\pm(\bbr) \cap N_\mp(\bbr)) \oplus (N_\pm(\bbr) \cap T_\mp(\bbr))$
(in other words, the reflections in $N_\pm(\bbr)$ commute).
\item In addition, arrange that some elements of
$N_\pm(\bbr) \cap T_\mp(\bbr)$ correspond to Kähler classes of some $Y_\pm$
under some marking of $S_\pm \subset Y_\pm$ (as pointed out in Remark
\ref{rmk:cones}, this is more restrictive than asking for Kähler classes
on $S_\pm$.)
\item Show that for a generic positive $k_0 \in T_+(\bbr) \cap T_-(\bbr)$,
choosing generic positive $k_\pm \in N_\pm(\bbr) \cap T_\mp(\bbr)$ gives
$\gen{k_\mp, \pm k_0} \in D_{N_\pm}$ that are periods of some
$S_\pm \subset Y_\pm \in \sff_\pm$, and use this to prove that $k_\pm$ can be
taken to correspond to Kähler classes on $Y_\pm$.
\item Blow up $Y_\pm$ according to Proposition \ref{prop:block_from_sf} to form
building blocks $(Z_\pm, S_\pm)$. The facts that the K3 fibres $S_\pm$ are
isomorphic to the chosen K3 divisors in the semi-Fanos $Y_\pm$ and that the
image of the Kähler cone of $Z_\pm$ contains that of $Y_\pm$ imply that
$(k_+, k_-, k_0)$ is a set of matching data for this pair of building blocks.
\end{itemize}

\subsubsection{Lattice push-outs and embeddings}

To complete the first step, we first try to cook up an ``orthogonal push-out''
$W$ of $N_+$ and~$N_-$, and then try to embed $W$ in $L$ so that the inclusions 
$N_\pm \into L$ are primitive. The last condition is obviously satisfied if
the embedding of $W$ itself is primitive, and the existence of such embeddings
can often be deduced from results of Nikulin.

\begin{definition} \label{def:orthog_pushout}
  Let $\ncap$, $N_+$, $N_-$ be nondegenerate lattices, and assume given
  primitive inclusions $\ncap \into N_+$, $\ncap \into N_-$.
  An \emph{orthogonal pushout} $W = N_+ \perp_\ncap N_-$ is a
  nondegenerate lattice, with a diagram of primitive inclusions:
\[
\xymatrix{
 & N_+ \ar[dr] & \\
\ncap \ar[ur] \ar[dr] &      & W\\
  & N_-\ar[ur] & }
\]
 where
 \begin{itemize}
 \item $\ncap = N_+ \cap N_-, \; W = N_+ + N_-$;
 \item $N_+^\perp \subset N_-, \; N_-^\perp \subset N_+$. 
 \end{itemize}
 Note that $W$ is unique, though it does not always exist, \eg see
 Example \ref{exa:no_pushout}.
\end{definition}

\begin{remark}
In all cases in this paper, $N_+$ and $N_-$ have signature $(1, \, r_+ - 1)$
and $(1, \, r_- - 1)$ and the ``intersection'' $\ncap$ is negative definite of
rank $\rho$.
This ensures that the orthogonal pushout $W$ has signature
$(2, \, r_+ + r_- - \rho - 2)$.
\end{remark}

The perpendicular direct sum $N_+ \perp N_-$ is always
an orthogonal push-out with $\ncap = {N_+ \cap N_-} = (0)$.
We will refer to gluing that uses this push-out
as \emph{perpendicular gluing}, while the term \emph{orthogonal gluing}
allows push-outs with non-trivial intersection $\ncap$.
Some statements simplify for perpendicular gluing (\eg the computation of
$\gdiv p_1(M)$ in Corollary \ref{cor:p1gcd}), but most nice properties are
enjoyed by orthogonal gluing too. Most important is the matching method in
Proposition \ref{prop:orth_gluing}, but there is also a convenient Betti number
formula that follows directly from Theorem \ref{thm:g2topology}.

\begin{lemma}
\label{lem:b2b3}
Any \gtmfd{} $M$ constructed by orthogonal gluing of blocks $Z_\pm$ satisfies
\[ b^2(M) + b^3(M) = b^3(Z_+) + b^3(Z_-) + 2\rk K_+ + 2\rk K_- + 23 . \]
\end{lemma}

\begin{remark*}
Note the previous formula is not always valid if $M$ is not constructed by orthogonal gluing: 
see Example \textit{No \ref{g2:precise}} in Section \ref{sec_g2mfds}.
\end{remark*}

In general, it is not difficult to state a simple criterion for the existence
of orthogonal pushouts in terms of \emph{discriminant groups} $N_\pm^*/N_\pm$;
we do not need to do so here. Instead we demonstrate by the next example
that they do not always exist.

\begin{example}
  \label{exa:no_pushout}
 We show a simple situation where the orthogonal pushout does not
 exist. Indeed, consider the two (isomorphic) lattices $N_+$, $N_-$
 with quadratic form
\[
\begin{pmatrix}
  4 & 1\\
  1 & -2
\end{pmatrix}
\] 
(this is the Picard lattice of a general quartic K3 surface containing a line).
Let us try to form an orthogonal pushout along the common sublattice $\ncap$
perpendicular to the basis vector $\mathbf{e}_1$ of norm $4$. 
Now $\ncap$ is generated by the vector $\mathbf{e}_2^\prime = (-1,4)$ of
norm $-36$. Using the rational basis $\mathbf{e}_1, \mathbf{e}_2^\prime$, we can say
\[
N_+ = N_- = \ZZ^2 + \frac1{4}(1,1)\ZZ \supset \ZZ^2
\quad
\text{with quadratic form}
\quad
\begin{pmatrix}
  4 & 0 \\
  0 & -36
\end{pmatrix}\, .
\]
Thus, if the orthogonal pushout $W = N_+ \perp_\ncap N_-$ exists, then 
\[
W = \ZZ^3 + \frac1{4}(1,1,0)\ZZ + \frac1{4}(0,1,1)\ZZ \supset \ZZ^3
\quad
\text{with quadratic form}
\quad
\begin{pmatrix}
  4 & 0   & 0\\
  0 &-36& 0\\
  0 & 0   & 4
\end{pmatrix}
\]
Note, however, that, in this lattice:
\[
\inner{ \quart (1,1,0), \quart (0,1,1) } =
\begin{pmatrix}
 \frac1{4} & \frac1{4} & 0 
\end{pmatrix}
\begin{pmatrix}
   4 & 0   & 0\\
  0 &-36& 0\\
  0 & 0   & 4
\end{pmatrix}
\begin{pmatrix}
  0 \\
  \frac1{4}\\
  \frac1{4}
\end{pmatrix}
=
\begin{pmatrix}
 \frac1{4} & \frac1{4} & 0 
\end{pmatrix}
\begin{pmatrix}
  0 \\ -9 \\1
\end{pmatrix}
=-\frac{9}{4}
\]
is not an integer, that is, $W$ is not an integral lattice.
\end{example}

Once we have an orthogonal push-out $W$, we look for an embedding $W$ in $L$
such that the inclusions $N_\pm \into L$ are primitive.
In many cases, the following result guarantees the existence of a primitive
lattice embedding $W\subset L$. Here $\ell(W)$ denotes the minimal
number of generators for the discriminant group $W^*/W$ of a
non-degenerate lattice $W$; in particular $\ell(W) \leq \rk W$.

\begin{theorem}
\label{thm:exist_emb}
Let $W$ be an even non-degenerate lattice of signature $(t_+, t_-)$,
and $L$ an even unimodular lattice of indefinite signature $(\ell_+,\ell_-)$.
There exists a primitive embedding $W \into L$ if $t_+ \leq \ell_+$,
$t_- \leq \ell_-$ and
\begin{enumerate}
\item \label{it:rk}
$2\rk W \leq \rk L$, or
\item \label{it:rkell}
$\rk W + \ell(W) < \rk L$
\end{enumerate}
\end{theorem}

\begin{proof}
\ref{it:rk} is Nikulin \cite[Theorem 1.12.4]{nikulin:quadratic}, while
\ref{it:rkell} is \cite[Corollary 1.12.3]{nikulin:quadratic} (see also
Dolgachev \cite[Theorem 1.4.6]{dol:quadratic}).
\end{proof}

If there is a primitive embedding of $W$ into the (unimodular) lattice $L$,
with orthogonal complement $T = W^\perp$, then $W^* \cong L/T$ and
$T^* \cong L/W$ imply that $W^*/W \cong L/(W \perp T) \cong T^*/T$, \ie the
discriminant groups are isomorphic. In particular $\ell(W) \leq \rk T$, so
\begin{equation}
\label{eq:nec_emb}
\rk W + \ell(W) \leq \rk L
\end{equation}
is a \emph{necessary} condition for $W$ to be primitively embeddable in $L$.

In our application $L$ will be the K3 lattice and $W$ will be the
orthogonal pushout of a pair of lattices $N_{\pm}$---the polarising
lattices of a pair of building blocks $Z_{\pm}$.  Therefore $(\ell_+,
\ell_-) = (3,19)$, while $t_+ = 2$ and $\rk W \leq \rk{N_{+}} +
\rk{N_{-}}$ with equality if and only if $W = N_+ \perp N_-$. Hence a
sufficient condition for the existence of a primitive embedding $W
\into L$ is that
\begin{equation}
\label{eq:exist_emb}
\rk{N_{+}} + \rk{N_{-}}\leq 11.
\end{equation}

Sometimes we will look more closely at the discriminant groups, and apply
\ref{thm:exist_emb}(ii). Note that the discriminant group of $N_+ \perp N_-$
is simply the product of the discriminant groups of the two terms.
In particular $\ell(N_+ \perp N_-) \leq \ell(N_+) + \ell(N_-)$
(but equality need not hold, \eg if the discriminants are coprime).

\begin{remark}
\label{rmk:disc}
For any lattice $N$, $N^*$ has a natural quadratic form, given in terms of a
basis by the inverse of the matrix of the form on~$N$. The restriction to $N$
is the original form on~$N$, so if $N$ is even then the discriminant group
$N^*/N$ has a well-defined $\QQ/2\ZZ$-valued quadratic form. For any
overlattice $W'$ of $N_+ \perp N_-$ such that $N_\pm \into W'$ are primitive,
the images of $W'$ in $N_\pm^*/N_\pm$ are anti-isometric with respect to the
discriminant forms. This sets up a correspondence between such overlattices
and pairs of anti-isometric subgroups of the discriminant groups. We will
sometimes use this to find overlattices, and since the overlattice has smaller
discriminant group they can be easier to embed in the K3 lattice $L$.

The method of the proof of Theorem \ref{thm:exist_emb} is to show that
given $W$, there exists a lattice $T$ with anti-isometric discriminant
group. Then the maximal overlattice of $W \perp T$ is unimodular, and
isometric to $L$ by the classification of unimodular indefinite lattices.
\end{remark}

Dolgachev \cite[Theorem 1.4.8]{dol:quadratic}, following
Nikulin \cite[1.14.1-2]{nikulin:quadratic}, also gives a sufficient condition
for the primitive embedding to be unique.

\begin{theorem}
\label{thm:unique_emb}
If in addition $\rk W + \ell(W) + 2 \leq \rk L$ then the primitive embedding
from Theorem \ref{thm:exist_emb} is unique up to automorphisms of $L$.
\end{theorem}

\subsubsection{Deformation theory and matching}

In order to find matching data for building blocks of semi-Fano type we use the
deformation theory input provided by Proposition \ref{prop:k3generic}.

\begin{definition}
Fix an abstract lattice $N$, and an element $A \in N$ with $A^2 = 2g-2 > 0$.
\begin{itemize}
\item
An \emph{$N$-polarised semi-Fano 3-fold} is a semi-Fano 3-fold $Y$ together
with an isometry $N \cong \Pic (Y)$ sending $A$ to $-K_Y$.
\item
A \emph{family of $N$-polarised semi-Fano 3-folds} is a smooth projective
morphism $f\colon \df{Y} \to B$, all of whose fibres $Y_b$ are semi-Fano
3-folds; and a sheaf isometry $g\colon N \to \underline{\Pic}(\df{Y}/B)$ such
that for each $b \in B$, $Y_b$ together with $g_b: N \to \Pic(Y_b)$ is an
$N$-polarised semi-Fano.
\item
Two $N$-polarised semi-Fano 3-folds $Y_1, Y_2$ are \emph{deformation
equivalent} if there is a connected scheme $B$, a family $f\colon \df{Y} \to B$
and $b_1, b_2 \in B$ such that $Y_1 = f^{-1}(b_1)$, $Y_2 = f^{-1}(b_2)$.
\item
A \emph{deformation type} is a deformation equivalence class
$\sff = \{Y_t : t \in T\}$ of semi-Fano 3-folds.
\end{itemize}
\end{definition}

For a smooth $S \in \abs{-K_Y}$ in an $N$-polarised semi-Fano 3-fold $Y$,
the composition $N \cong H^2(Y;\ZZ) \to H^2(S;\ZZ) \cong L$ defines
a primitive embedding $N \into L$.
For a deformation type $\sff$ of $N$-polarised semi-Fanos, this gives an
embedding $N \into L$ that is well-defined up to the action of~$\orth{L}$.

The precise definition of the deformation type $\sff$ is actually not
that crucial in this paper. For the application, it suffices to know that
given a semi-Fano $3$-fold $Y$, there is a collection $\sff$ of semi-Fano
3-folds with the same topology, satisfying the conclusions of the following
result.

\begin{prop}[{\cite[Proposition 6.9]{chnp1}}]
\label{prop:k3generic}
Fix a primitive lattice $N \subset L$, and let
$D_N$ be the Griffiths domain $\{ \Pi \in \bbp(\Lambda^2(N^\perp \otimes \bbc))
: \Pi \wedge \bar\Pi > 0\}$.
Let $\sff$ be a deformation type of $N$-polarised semi-Fano 3-folds $Y$
such that for $S \subset Y$ a smooth anticanonical K3 divisor the restriction
map $\Pic Y \to H^2(S;\bbz)$ is equivalent (for the chosen polarisation 
$N \cong \Pic Y$ and some isomorphism $H^2(S;\bbz) \cong L$) to the inclusion
$N \into L$. Then there exist
\begin{itemize}
\item a $U_\sff \subseteq D_N$ with complement a locally finite
union of complex analytic submanifolds of positive codimension
\item an open subcone $\Amp_\sff$ of the positive cone of $N_{\R}$ %
\end{itemize}
with the following property: for any $\Pi \in U_\sff$ and $k \in \Amp_\sff$
there is $Y \in \sff$, a smooth $S \in |{-K_Y}|$ and a marking
$\hdg : L \to H^2(S; \ZZ)$ such that $\hdg(\Pi) = H^{2,0}(S)$, and $\hdg(k)$ is
the restriction to $S$ of a Kähler class on $Y$.
\end{prop}

\begin{remark}
\label{rmk:ampfamily}
  It is important to distinguish $\Amp_{\sff}$ from the cone
  $\Amp_S \subset N_\bbr$ of Kähler classes on $S$. 
  For example, if $Y$ is semi-Fano (but not Fano) with small anticanonical
  morphism then $-K_Y$ is not a Kähler class on $Y$ but it is when
  restricted to a generic $S$. $\Amp_\sff$ can be a proper subcone of $\Amp_S$
  also for genuine Fanos when the Picard rank is $\geq 2$,
  \eg $Y = \CP^1 \times \CP^1 \times \CP^1$.
\end{remark}

If we apply Proposition \ref{prop:block_from_sf} to $\sff$ to construct a
family of semi-Fano type blocks $\fbb$, then it is immediate that $\fbb$ has
the following property (with $\Amp_\fbb = \Amp_\sff$).

\begin{definition}
\label{def:generic}
Let $N \subset L$ be a primitive sublattice, and $\Amp_\fbb$ an open subcone
of the positive cone in $N_\bbr$. We say that a family of building blocks
$\fbb$ is \emph{$(N, \Amp_\fbb)$-generic} if there is
$U_\fbb \subseteq D_N$ with complement a locally finite
union of complex analytic submanifolds of positive codimension with the
property that:
for any $\Pi \in U_\fbb$ and $k \in \Amp_\fbb$ there is a building block
$(Z,S) \in \fbb$ and a marking $\hdg : L \to H^2(S; \ZZ)$ such that
$\hdg(\Pi) = H^{2,0}(S)$, and $\hdg(k)$ is the image of the restriction
to $S$ of a Kähler class on $Z$.
\end{definition}

Given an embedding in $L$ of the orthogonal push-out we can now 
solve the matching problem for semi-Fano 3-folds.

\begin{prop}
\label{prop:orth_gluing}
  Let $N_\pm \subset L$ be primitive sublattices of signature
  $(1, \, r_\pm-1)$, and let $\fbb_\pm$ be 
  $(N_\pm, \Amp_{\fbb_\pm})$-generic families of building blocks.
  Suppose that
  \begin{enumerate}
  \item $\ncap = N_+ \cap N_-$ is negative definite of rank $\rho$,
  \item $W = N_+ + N_-$ is an orthogonal pushout. 
  \end{enumerate}
Denote by $T_\pm = N_\pm^\perp$ the transcendental lattices, and
let $W_\pm = T_\mp \cap N_\pm \subset N_\pm$ be the perpendicular of $N_\mp$
in~$N_\pm$. Assume also that
\begin{enumerate}[resume]
\item \label{it:amp}
$W_\pm \cap \Amp_{\fbb_\pm} \neq \emptyset$.
\end{enumerate}
Then there exist %
$(Z_\pm, S_\pm) \in \fbb_\pm$ and $N_\pm$-polarised markings
$\hdg_\pm : L \to H^2(S_\pm; \ZZ)$ with period points $\gen{k_\pm, \pm k_0}$,
for an orthonormal triple of positive classes $(k_+, k_-, k_0)$ in $L_\bbr$
such that $k_\pm \in \Amp_{\fbb_\pm}$,
\ie $(k_+, k_-, k_0)$ is a set of matching data adapted to the chosen
pair of embeddings $N_\pm \into L$.
\end{prop}

\begin{proof}
Let $T = W^\perp$. $W_\pm(\bbr)$ and $T(\bbr)$ are real vector spaces of
signature $(1, \, r_\pm - \rho - 1)$ and $(1, \, 21 -r)$ respectively, where
$r = \rk W = r_+ + r_- - \rho$.
A priori, $k_\pm$ and $k_0$ must belong to the positive cones $W_\pm(\bbr)^+$
and $T(\bbr)^+$ respectively.
Consider the real manifold 
\[D= \PP\bigl(W_+(\RR)^+\bigr)\times \PP\bigl(W_-(\RR)^+\bigr)
\times \PP \bigl(T(\RR)^+\bigr)\, .
\]
Below, we need the open subset
$\cala = \{(\ell_+,\ell_-,\ell) \in D : \ell_\pm \subset \Amp_{\fbb_\pm}\}$.
By hypothesis \ref{it:amp}, $\cala$~is nonempty.
We have two Griffiths period domains
 \[  D_{N_\pm} = \{\Pi^2 \subset T_\pm (\RR)\mid \langle \bullet, \bullet
   \rangle_{|\Pi^2}>0\} , \]
and projections
\[
\text{pr}_\pm \colon D \to D_{N_\pm}, \;
(\ell_+, \ell_-,\ell) \mapsto \gen{\ell_\mp, \pm \ell} .
\]
As we stated previously, $D_{N_\pm}$ can be regarded as an open subset of
$\PP(N_\pm^\perp \otimes \CC)$; if $\alpha, \beta$ is an oriented orthonormal
basis of $\Pi \in D_{N_\pm}$ then
$\Pi \mapsto \gen{\alpha+i\beta} \in \PP(N_\pm^\perp \otimes \CC)$.
Given a choice $\alpha$ and $\beta$, we can identify $T_\Pi D_{N_\pm}$ with
pairs $(u,v)$ of vectors in $\Pi^\perp \subseteq T_\pm(\RR)$. Then the complex
structure on $T_\Pi D_{N_\pm}$ is given by $J : (u, v) \mapsto (-v, u)$.

Observe that the real analytic embedded submanifold
$\PP\bigl(W_\mp(\RR)^+\bigr)\times \PP \bigl(T(\RR)^+\bigr) \into D_{N_\pm}$ is
totally real: the tangent space $\mathcal{T}$
at $\Pi = \gen{w,t}$, $w \in W_\mp$, $t \in T(\RR)$ corresponds to $(u,v)$ such
that $u \in w^\perp \subseteq W(\RR)$ and $v \in t^\perp \subseteq T_\mp(\RR)$,
so $J(\mathcal{T})$ is transverse to $\mathcal{T}$.
Now the key point is that the condition that $N_+$ and $N_-$ intersect
orthogonally ensures that this totally real submanifold
has \emph{maximal dimension}: $\dim_\CC D_{N_\pm} = 20-r_\pm$, and
\[ \dim_\RR \PP \bigl(W_\mp(\RR)^+\bigr)\times \PP \bigl(T(\RR)^+\bigr) =
\bigl(r_\mp -\rho - 1\bigr) + \bigl(22 - r - 1\bigr) = 20 - r_\pm . \]

In particular, the submanifold is Zariski dense (in a complex analytic sense),
so it must intersect the subset $U_{\fbb_\pm} \subset D_{N_\pm}$ from
Definition \ref{def:generic}. Actually, we need to use a stronger consequence:
the complement of the preimage of $U_{\fbb_\pm}$ in
$\PP \bigl(W_\mp(\RR)^+\bigr)\times \PP \bigl(T(\RR)^+\bigr)$ is a locally
finite union of real analytic subsets of positive codimension.
Therefore the same is true for $\text{pr}_\pm^{-1}(U_{\fbb_\pm})$.
To conclude the proof, take $(\ell_+, \ell_-, \ell) \in \cala \cap
\text{pr}_+^{-1}(U_{\fbb_+}) \cap \text{pr}_-^{-1}(U_{\fbb_-})$, and
let $k_\pm \in \ell_\pm, k_0 \in \ell$ be unit vectors.
\end{proof}

Proposition \ref{prop:orth_gluing} fulfils the plan for finding compatible
semi-Fano type ACyl Calabi-Yau \mbox{3-folds} outlined at the start of the orthogonal
gluing subsection.
The non-symplectic type blocks of Kovalev and Lee \cite{kovalev:lee}
also satisfy the condition in Definition \ref{def:generic}, as do some
families of blocks obtained by resolving non-generic anticanonical pencils on
semi-Fanos, \eg Example~\ref{exg:2conics}. So we can solve the
matching problem for these kinds of blocks by the same method.

\begin{remark}
\label{r:w:ample}
Note that in perpendicular gluing, hypothesis \ref{it:amp} is automatically
satisfied. This condition may look innocuous, but it adds an extra layer of
difficulty to the problem of finding suitable orthogonal but non-perpendicular
pushouts.

For families of non-symplectic type blocks, we may take $\Amp_\fbb$
in Definition \ref{def:generic} to be the full Kähler cone $\Amp_S$
of a generic $N$-polarised K3 surface (\cf Remark \ref{rmk:kl-def}). Modulo
choice of markings, this consists of all positive classes in $N_\bbr$ that
are orthogonal to all $-2$ classes in~$N$. For these blocks, hypothesis
\ref{it:amp} is therefore equivalent to $\ncap$ not containing any $-2$
classes. This is always a necessary condition, but for semi-Fano type blocks
it is not sufficient, \cf Example \ref{ex:nonex}.
\end{remark}

%% file: sec_g2examples.tex
Our aim in this section is to present in detail concrete examples of \gtwo-manifolds that illustrate 
the main points of what is achievable by our constructions. 
In Section \ref{sec:close} we will give a more systematic overview of 
the range of examples one can construct using these methods 
and some remarks on the basic ``geography'' of the examples.
For each example in this section we compute the
integral cohomology groups, the number of associative submanifolds
arising from the construction, and the first Pontrjagin
class. %
Many of the examples are 2-connected, and for most of these we
can determine the diffeomorphism type completely using the classification
theorems \ref{thm:torfreeclass} and \ref{thm:inertia}.

All examples except \textit{No \ref{g2:precise}} are constructed using 
perpendicular or orthogonal gluing. 
We mostly stick to the building blocks of semi-Fano type that we described in detail 
in our earlier paper \cite[\S 7]{chnp1}; these building blocks are described 
briefly in Examples \ref{exg:species1}--\ref{exg:doubleline}.
\textit{No \ref{g2:precise}} uses ``handcrafted nonorthogonal gluing''. 
This method allows us to construct examples not possible using orthogonal gluing; 
the main drawback is that the method requires much more explicit information 
about K3 moduli spaces than orthogonal gluing. 
This can make constructing such examples a very labour-intensive process. 
Here we give only the simplest possible example to illustrate how the method 
works and its potential subtleties.

We close the section with a pair of examples (Examples \ref{exg:quartic:sl} and \ref{exg:nodal:quartic}) 
in which we can construct families of associative $3$-folds (recall Proposition  \ref{prop:lag_to_assoc}) 
because of the existence of suitable special Lagrangian submanifolds of the building blocks.

\subsection{Building blocks}

A small number of representative examples of building blocks $(Z,S)$,
together with computations of their topological and geometric
invariants, is given in \cite[\S 7]{chnp1}. Here we
give a very brief description of these examples: see also Tables
\ref{table:species1}, \ref{tableg:blocks}
and \ref{table:rank2}.  In each case the
polarising lattice $N$ (the image of $H^2(Z) \to H^2(S)$) has a unique
primitive embedding in~$L$; except in Example \ref{exg:burkhardt_quartic}
this is a direct consequence of Theorem \ref{thm:unique_emb}.

The building blocks $Z$ in Examples
\ref{exg:species1}--\ref{exg:burkhardt_quartic} are of Fano or
semi-Fano type, \ie $Z$ is the blow-up of a smooth Fano or semi-Fano
$Y$ in the base locus of a generic anticanonical pencil on $Y$ (recall
Proposition \ref{prop:block_from_sf}).  Below we will list the Fano or
semi-Fano $Y$ we use to construct the building block $Z$.

\begin{example}
\label{exg:species1}
Take $Y$ to be a Fano ``of the first species'', \ie a member of one of
the $17$ deformation families of smooth Fano $3$-folds with Picard
rank $1$. 
The building blocks $Z$ which arise this way---which we call
\emph{building blocks of rank one Fano type}---are listed in Table
\ref{table:species1}. In the descriptions of our examples of twisted
connected sums, ``\ref{exg:species1}$^r_d$'' refers to the building
block obtained from the rank 1 Fano $Y$ with index $r$ and degree
$-(\frac{1}{r}K_Y)^3 = d$. The polarising lattice is $N = \gen{rd}$.
\begin{table}[h]
\[
\begin{array}[t]{l c  c  c  c  c} \toprule
Y & r & -K_Y^3 & b^{3}(Y)&b^{3}(Z) & \gdiv{c_2(Z)} \\ \midrule
          \CP^3 &  4 & 4^3         &   0 &  66 &   2 \\ 
             Q_2 \subset \CP^4 &  3 & 3^3 \cdot 2 &   0 &  56 &   2 \\ 
                    V_1 \to W_4 &  2 & 2^3         &  42 &  52 &   8 \\ 
                  V_2 \to \CP^3 &  2 & 2^3 \cdot 2 &  20 &  38 &   4 \\ 
              Q_3 \subset \CP^4 &  2 & 2^3 \cdot 3 &  10 &  36 &  24 \\ 
      V_{2\cdot2} \subset \CP^5 &  2 & 2^3 \cdot 4 &   4 &  38 &   4 \\ 
              V_5 \subset \CP^6 &  2 & 2^3 \cdot 5 &   0 &  42 &   8 \\ 
                  V_2 \to \CP^3 &  1 &           2 & 104 & 108 &   2 \\ 
              Q_4 \subset \CP^4 &  1 &           4 &  60 &  66 &   4 \\ 
      V_{2\cdot3} \subset \CP^5 &  1 &           6 &  40 &  48 &   6 \\ 
V_{2\cdot2\cdot2} \subset \CP^6 &  1 &           8 &  28 &  38 &   8 \\ 
           V_{10} \subset \CP^7 &  1 &          10 &  20 &  32 &   2 \\ 
           V_{12} \subset \CP^8 &  1 &          12 &  14 &  28 &  12 \\ 
           V_{14} \subset \CP^9 &  1 &          14 &  10 &  26 &   2 \\ 
        V_{16} \subset \CP^{10} &  1 &          16 &   6 &  24 &   8 \\ 
        V_{18} \subset \CP^{11} &  1 &          18 &   4 &  24 &   6 \\ 
        V_{22} \subset \CP^{13} &  1 &          22 &   0 &  24 &   2 \\ 
        \bottomrule
\end{array} 
\]
\bigskip
\caption{Building blocks $Z$ from Fanos $Y$ with Picard rank 1}
\label{table:species1}
\end{table}

\end{example}

\newcommand{\blockstable}{
\begin{table}
\[
\begin{array}[b]{cccccccc} \toprule
 \text{What} & -K_Y^3& H^2(Z) & N & K & H^3(Z) & \gdiv c_2(Z) & e \\ \midrule
\text{Ex}~\ref{exg:quartic_w_plane}      &4 &\ZZ^3
&\begin{pmatrix}-2&1\\1&4\end{pmatrix}   &(0)&\ZZ^{50}
& 2,4 & 9 \\ \addlinespace[0.5em]
\text{Ex}~\ref{exg:quartic_w_quadric} & 4&\ZZ^3    
&\begin{pmatrix}-2&2\\2&4\end{pmatrix}&(0)&\ZZ^{44}
& 2 & 12\\ \addlinespace[0.5em]
\text{Ex}~\ref{exg:quartic_w_scroll}     &4 &\ZZ^3    
&\begin{pmatrix}-2&3\\3&4\end{pmatrix}&(0)&\ZZ^{34}
& 2,4 & 17\\ \addlinespace[0.5em]
\text{Ex}~\ref{exg:quartic_w_22}          &4 &\ZZ^3    
&\begin{pmatrix}0&4\\4&4\end{pmatrix}&(0)&\ZZ^{36}
& 4 & 16 \\ \addlinespace[0.5em]
\text{Ex}~\ref{exg:burkhardt_quartic}   &4 &\ZZ^{17}
&E_6^\ast (-3)\perp E_8(-1)\perp U & (0) & \ZZ^6
& 2 & 45            \\ \addlinespace[0.5em]
\text{Ex}~\ref{exg:P3_deg}                  &64 &\ZZ^{5}  
&\langle 4\rangle &\ZZ^3&\ZZ^{24} & 2 & 24\\ \addlinespace[0.5em]
\text{Ex}~\ref{exg:2conics} & 64 & \ZZ^4 &
\begin{pmatrix}-2 & 0 & 2 \\ 0 & -2 & 2 \\ 2 & 2 & 4\end{pmatrix} &
(0) & \ZZ^{30} & 2 & 20 \\ \addlinespace[0.5em]
\text{Ex}~\ref{exg:toricV22_a}           &22 &\ZZ^{11}
&E_8(-1)\perp \langle 8\rangle \perp \langle -16 \rangle&(0)&\ZZ^{24}
& 2 & 9 \\ \addlinespace[0.5em]
\text{Ex}~\ref{exg:toricV22_b}          &22&\ZZ^{23} 
&E_8(-1)\perp \langle 8\rangle \perp \langle -16 \rangle&\ZZ^{12}&(0)
& 2 & 33 \\ \addlinespace[0.5em]
\text{Ex}~\ref{exg:doubleline}          & 4 & \ZZ^3 
& \gen{4} \perp \gen{-2} & (0) & \ZZ^{46}
& 2 & 12 \\ \bottomrule
\end{array}
\]
\bigskip
\caption{A small number of examples of building blocks (reproduced from \cite[Table~2]{chnp1}).}
\label{tableg:blocks}
\end{table}}

\begin{example}
\label{exg:MoriMukai}
Similarly, we can take $Y$ to be any of the Fano 3-folds of Picard
rank $\geq 2$ classified by Mori-Mukai \cite{MM1,MM2,MM3,MM4,MM5}.  We
list some building blocks of this type separately in
Table~\ref{table:rank2}.  In our final Table~\ref{table:g2ex} of
examples of \gtwo-manifolds, the notation $Z=Ex~\ref{exg:MoriMukai}_n$
signifies the building block $Z$ of Fano type obtained from the
\emph{rank~2 Fano} 3-fold $Y$ listed as no.~$n$ in the table in ~\cite{MM5} 
(and also in our Table \ref{table:rank2}).
We call these \emph{building blocks of rank $2$ Fano type}.
\end{example}
Examples \ref{exg:quartic_w_plane}--\ref{exg:burkhardt_quartic} are
building blocks of semi-Fano type where the semi-Fano $Y$ is obtained
as a projective small resolution of a Fano 3-fold $X$ with nodal
singularities. For a given $X$ there may be be several non-isomorphic
small resolutions $Y$, but they all have the same cohomology. However,
the value of $\gdiv c_2(Z)$ may depend on the choice of small
resolution $Y \to X$.

\begin{example}
  \label{exg:quartic_w_plane}
  Fix a 2-plane $\Pi\subset \PP^4$ and let $\Pi \subset X\subset
  \PP^4$ be a general quartic 3-fold containing~$\Pi$.  Let $Y$ be one
  of the two projective small resolutions of $X$.
\end{example}

\begin{example}
  \label{exg:quartic_w_quadric}
  Fix a quadric surface $Q=Q^2_2\subset \PP^4$ and let $Q \subset
  X\subset \PP^4$ be a general quartic 3-fold containing $Q$.  Let $Y$
  be one of the two projective small resolutions of $X$.
\end{example}

\begin{example}
  \label{exg:quartic_w_scroll}
  Fix a cubic scroll surface $\FF\subset \PP^4$ and let $\FF \subset
  X\subset \PP^4$ be a general quartic 3-fold containing $\FF$.  Let
  $Y$ be one of the two projective small resolutions of $X$.
\end{example}

\begin{example}
\label{exg:quartic_w_22}
  Fix the complete intersection of two quadrics $F=F_{2,2}\subset
  \PP^4$ and let $F \subset X\subset \PP^4$ be a general quartic
  3-fold containing $F$.
  Let $Y$ be one of the two projective small resolutions of $X$.
\end{example}

\begin{example}
  \label{exg:burkhardt_quartic}
  The Burkhardt quartic 3-fold is the hypersurface
\[
X= \bigl(x_0^4-x_0(x_1^3+x_2^3+x_3^3+x_4^3)+3x_1x_2x_3x_4=0\bigr)\subset \PP^4.
\]
$X$ contains $40$ planes, has $45$ ordinary nodes as singularities,
defect $\sigma = 15$ (recall \eqref{e:defect}), and admits 
projective small resolutions. (See Finkelnberg's thesis
\cite{finkelnberg89:burk} for these and other facts on the Burkhardt
quartic.)
We take $Y$ to be one particular projective small resolution of $X$.
The polarising lattice $N$ has rank 16 and discriminant group $(\ZZ/3\ZZ)^5$.
Its orthogonal complement $T \subset L$ is the rank 6 lattice
$A_2(-1) \perp 2U(3)$,
where $A_2(-1)$ and $U(3)$ denote the rank 2 lattices with intersection forms
\[ \begin{pmatrix} -2 & 1 \\ 1 & -2 \end{pmatrix}   \textrm{ and }
\begin{pmatrix} 0 & 3 \\ 3 & 0 \end{pmatrix}  \]
respectively.
In \cite[Example 7.7]{chnp1} we deduce the uniqueness
of the embedding $N \subset L$ from that of $T \subset L$.
\end{example}

The next two examples arise by blowing up the base locus of a
\emph{non-generic} anticanonical pencil on $\CP^3$, \ie they do not come from an application
of Proposition \ref{prop:block_from_sf}. In these cases extra work is required both to verify 
that the topological conditions of a building block (recall Definition \ref{dfng:BLOCK}) are satisfied 
and that the matching arguments of \S \ref{sec:k3} can be applied.

\begin{example}
  \label{exg:P3_deg}
  Consider the non-generic AC (anti-canonical) pencil
  $|S_0, S_\infty|\subset |\oo(4)|$, where
  $$S_0=(x_0x_1x_2x_3=0)$$ is the sum of the four coordinate planes,
  and $S_\infty$ is a nonsingular quartic surface meeting all
  coordinate planes transversely. The base curve of the pencil is the
  union $C=\sum_{i=0}^3\Gamma_i$ of the four nonsingular curves
  $\Gamma_i =(x_i=0) \cap S_\infty$. Let $Z$ be obtained from
  $Y=\PP^3$ by blowing up the four base curves one at a time.
  Any smooth quartic K3 appears as a fibre of a building block of this kind,
  so even though we are using non-generic pencils we can apply the same
  orthogonal gluing argument as for building blocks obtained by resolving
  generic pencils.
\end{example}

\begin{example}
\label{exg:2conics}
Fix two general conics $C_1, C_2 \subset \CP^3$, and take a generic
pencil of quartic K3 surfaces containing both $C_1$ and $C_2$. The
base locus $C$ consists of $C_1$, $C_2$ and a degree 12 curve~$C_3$
(of genus 15) meeting each of $C_1$ and $C_2$ in 10 points. Let $Z$ be
the result of first blowing up~$C_1$, then the proper transform of~$C_3$,
and then the proper transform of~$C_2$, and let $S$ be the
proper transform of a smooth element of the chosen pencil on~$\CP^3$.
$(Z,S)$ is a building block, with 20 $(-1,-1)$ curves
corresponding to the double points of $C$. (Blowing up the components
of $C$ in a different order changes $Z$ by flopping some of the 20
exceptional curves, but does not change the data listed in
Table~\ref{tableg:blocks}.)

$S$ contains the pair of conics $C_1$, $C_2$, so these represent classes
in $N = \Pic S$. Together with the
hyperplane class $A$ they are the  basis of a subgroup $N\subset \Pic S$, and in this basis the quadratic
form on $N$ is
\[
\begin{pmatrix}-2 & 0 & 2 \\ 0 & -2 & 2 \\ 2 & 2 & 4\end{pmatrix} . 
\]
We check by hand that this family of blocks satisfies the conditions
of Definition \ref{def:generic}, so that the orthogonal matching
Proposition \ref{prop:orth_gluing} can be applied to it. The main point
is that a generic $N$-polarised K3 $S$ appears as the fibre in some block
$(Z,S)$ in the family.

Let $D_N$ be the Griffiths domain for $N$. It is explained in
\cite[Example 7.9]{chnp1} that a generic marked K3 $S$ with
period in $D_N$ embeds as a quartic in $\CP^3$, and contains a pair of
conics. We can then form a block $(Z,S)$ by blowing up the intersection
of $S$ with a generic quartic containing those two conics.
Thus there is a $U_\fbb \subset D_N$ with complement a locally finite union
of complex analytic subsets of positive codimension, such that for any
$\Pi \in U_\fbb$ there is a building block $(Z,S)$ in our family, with
$\Pi$ the period of a marking for $S$.
 
Next, let $E_i$ be the exceptional divisor in $Z$ over $C_i$ ($E_i$ is
isomorphic to the projectivisation of the normal bundle of $C_i$,
blown up at points corresponding to intersections with those
components of $C$ blown up after $C_i$). The pull-back $H$ to $Z$ of
the hyperplane class on $\CP^3$ is nef, but it fails to be positive on
the fibres of $E_i$. On the other hand, $S$ is positive on almost all
of the fibres. For small $\lambda_0 > 0$, $H + \lambda_0 S$---which has
image $A$ in $H^2(S)$---is nef and positive on all curves except the
$\oo(-1) \oplus \oo(-1)$ curves over the $20$ intersection points of
$C_1 \sqcup C_2$ with $C_3$. By adding $-\lambda_1 E_1 + \lambda_2 E_2$
for small $\lambda_i > 0$ we get a Kähler class, with image $A -
\lambda_1C_1 + \lambda_2 C_2$ in $H^2(S)$. Therefore there is an open
subcone $\Amp_\fbb$ of the positive cone in $N_\bbr$ that can be taken
as restrictions of Kähler classes on $Z$ ($A$ spans an edge of
$\Amp_\fbb$).

Thus the family $\fbb$ is $(N, \Amp_\fbb)$-generic, and can be used in
orthogonal gluing.
\end{example}

\blockstable

Examples \ref{exg:toricV22_a} and \ref{exg:toricV22_b} 
are obtained from the same toric semi-Fano $3$-fold $Y$ 
by blowing up a generic AC pencil and a nongeneric AC pencil on $Y$ respectively.

\begin{example}
\label{exg:toricV22_a} 
Let $X$ be the terminal Gorenstein toric Fano 3-fold with Fano
polytope the reflexive polytope in $\Hom (\CC^\times, \TT)$ with
vertices (this is polytope 1942 in the Sage implementation of Kreuzer
and Skarke's database of 4319 reflexive polytopes in 3 dimensions)
\[
\left(\begin{array}{rrrrrrrrrrrrr}
1 & 0 & 0 & -1 & 1 & 1 & -1 & -1 & -1 & 1 & 0 & 0 & 0 \\
0 & 1 & 0 & 1 & 0 & -1 & 1 & 0 & 0 & -1 & 0 & -1 & -1 \\
0 & 0 & 1 & 1 & -1 & 0 & 0 & 1 & 0 & -1 & -1 & 0 & -1
\end{array}\right).
\]
  Let $Y$ be a projective small resolution of $X$, and $Z$ the blow-up of $Y$
  in the base locus of a generic AC pencil.
\end{example}

\begin{example}
  \label{exg:toricV22_b}
  We construct the building block $Z$ by blowing up a different (non-generic) pencil
  on the toric semi-Fano $3$-fold $Y$ used in the previous example.
One can show that any generic
anticanonical divisor in $Y$ appears as a fibre in a building block of this
kind, so we can apply the orthogonal gluing argument when attempting to
find matchings involving this block.
\end{example}

The final example comes from a semi-Fano 3-fold whose anticanonical morphism
is not small. Even though it is not constructed as a small resolution of
a nodal variety, it still contains some curves with normal bundle
$\oo(-1) \oplus \oo(-1)$.

\begin{example}
\label{exg:doubleline}
Let $X \subset \CP^4$ be a generic quartic containing a double line,
$Y$ the crepant resolution of $X$, and $Z$ the blow-up of $Y$ in the base locus
of a generic AC pencil.
The exceptional set of $Y \to X$ is a conic bundle with 6 degenerate fibres.
Each degenerate fibre consists of two $\CP^1$s intersecting in a single point.
Each of these 12 $\CP^1$s has normal bundle $\oo(-1) \oplus \oo(-1)$.
\end{example}

\subsection{Examples of compact \gtmfd s from orthogonal gluing}

We start with pairs of building blocks $Z_\pm$ taken from the examples
listed above and construct compact \gtwo-manifolds from such pairs by
using orthogonal gluing to solve the matching problem.  We summarise
the invariants of the resulting \gtmfd s in Table \ref{table:g2ex}.

More specifically given a pair of building blocks $Z_{\pm}$ with
corresponding polarising lattices $N_{\pm}$ first we make a choice of
an \emph{orthogonal push-out} $W= N_+ \perp_\ncap N_-$ of the pair
$N_\pm$ as in Definition \ref{def:orthog_pushout}; for a given pair of
lattices $N_\pm$ there is often some choice in this.  Recall that
\emph{perpendicular gluing} refers to the special case when we choose
$W=N_+ \perp N_-$, \ie $\ncap = N_+ \cap N_- = (0)$.  In order to
satisfy conditions (i) and (ii) of Proposition \ref{prop:orth_gluing},
we then find an embedding $W \into L$ such that the inclusions $N_\pm
\into L$ are primitive. Usually we achieve this by applying Theorem
\ref{thm:exist_emb} to find a primitive embedding $W \into L$;
we refer to this as \emph{primitive orthogonal gluing} or
\emph{primitive perpendicular gluing}. In the perpendicular case
Proposition \ref{prop:orth_gluing} then produces matching data, and
therefore compact \gtwo-manifolds by appeal to Theorem
\ref{thm:g2glue} and Corollary \ref{cor:matching}. In the
non-perpendicular case, we also need to calculate the Kähler cones of
$Z_\pm$ to verify condition \ref{prop:orth_gluing}(iii).

The topology of the resulting \gtmfd{} depends only on the blocks and
the choice of push-out. The integral cohomology groups can readily be
computed using Theorem \ref{thm:g2topology} and the data in Tables
\ref{table:species1}, \ref{tableg:blocks}  and \ref{table:rank2}.

The following observation is helpful for identifying the torsion in
$H^3$ and $H^4$.
\begin{lemma}
\label{lem:torsion}
  Let $L$ be a unimodular lattice, $N_+, N_-\subset L$ two primitive
  submodules and $T_+, T_-$ their perpendicular complements in $L$. Then
\begin{align*}
L/(N_+ + N_-) &= \coker (N_+ \to T_-^*) = \coker (N_- \to T_+^*) \, , \\
L/(N_+ + T_-) &= \coker (N_+ \to N_-^*) = \coker (T_- \to T_+^*)\, . 
\end{align*} 
\end{lemma}
In the case of perpendicular gluing $p_1$ is also straightforward to compute;
Corollary \ref{cor:p1gcd} tells us that it suffices to know the greatest
divisors of $c_2$ of the building blocks, which we also included in the tables.
For non-perpendicular gluing, we have to work a little bit harder to compute
$p_1$, using some of the details of the $c_2$ calculation from
\cite[\S 5]{chnp1}.

The simplest building blocks to match are the 17 families of building blocks of rank one Fano type
described in Example \ref{exg:species1} and summarised in Table \ref{table:species1}. 
\gtwo-manifolds obtained by matching pairs of rank one Fanos already appear in  \cite[\S 8]{kovalev:connectsums} but 
we can now give much more precise information about the topology of \gtwo-manifolds constructed this way including, in most
cases, their diffeomorphism type. 
The most straightforward way to achieve matching in this case is to use primitive perpendicular gluing, 
\ie to choose a primitive lattice embedding of the rank two lattice $W=N_{+}\perp N_{-}$ into the K3 lattice $L$. 
Existence and uniqueness (up to lattice automorphisms of $L$) of this embedding follow from Theorems \ref{thm:exist_emb}
and \ref{thm:unique_emb}.
However even in this case there are other ways to achieve 
matching which lead to \gtwo-manifolds with the same Betti numbers but different 
integral cohomology groups; see example \textit{No \ref{g2:qqtorsion}} below.
For now though we restrict attention to matching by primitive perpendicular gluing and consider the topology of the 
resulting compact \gtwo-manifolds.

\subsubsection*{Perpendicular gluing of pairs of rank 1 smooth Fano $3$-folds}
By Lemma \ref{l:2connected:g2} any twisted connected sum $\gtwo$-manifold $M$
arising by primitive perpendicular gluing of blocks of semi-Fano or Fano type
is $2$-connected (recall from Proposition \ref{prop:block_from_sf} that $K=0$
for any block of semi-Fano type) and has $H^4(M)$ torsion-free.
Hence the almost-diffeomorphism classification of Theorem \ref{thm:torfreeclass} applies 
to $M$. Recall also that from Lemma \ref{lem:p1range} we have  $\gdiv p_1(M) \in \{4,8,12,16,24,48\}$ for any 
twisted connected sum $\gtwo$-manifold and 
that Corollary \ref{c:inertia:g2} restricts the number of 
diffeomorphism types in a given almost diffeomorphism class according to $\gdiv p_1(M)$.
In particular, there are at most $8$ diffeomorphism classes realising the same value of $b^3(M)$.

The data of all possible $153=\tfrac{1}{2} \cdot 18 \cdot 17$ such matching pairs is collected in Table \ref{tab:rank1}.
We summarise some of the main features apparent from this table.

\begin{table}[p]
{\small
\begin{tabular}{cccccccc} \toprule
$b$ & $\#$ & \multicolumn{6}{c}{$\gdiv(p_{1})$} \\ \cmidrule(l){3-8}
 &  & 4 & 8 & 12 & 16 & 24 & 48 \\ \midrule
48 & 6 & 4 & 0 & 1 & 1 & 0 & 0\\ 
50 & 3 & 3 & 0 & 0 & 0 & 0 & 0\\
52 & 4 & 2 & 1 & 1 & 0 & 0 & 0\\
54 & 1 & 1 & 0 & 0 & 0 & 0 & 0\\
56 & 4 & 3 & 0 & 0 & 0 & 1 & 0\\
58 & 1 & 1 & 0 & 0 & 0 & 0 & 0\\
60 & 4 & 2 & 0 & 1 & 1 & 0 & 0\\
62 & 10 & 7 & 2 & 0 & 1 & 0 & 0\\
64 & 5 & 4 & 0 & 0 & 0 & 1 & 0\\
66 & 6 & 2 & 3 & 0 & 1 & 0 & 0\\
68 & 2 & 2 & 0 & 0 & 0 & 0 & 0\\
70 & 4 & 3 & 1 & 0 & 0 & 0 & 0\\
72 & 4 & 2 & 0 & 1 & 0 & 0 & 1\\
74 & 5 & 2 & 2 & 0 & 1 & 0 & 0\\
76 & 10 & 2 & 5 & 1 & 2 & 0 & 0\\
78 & 2 & 1 & 0 & 0 & 1 & 0 & 0\\
80 & 8 & 4 & 3 & 0 & 1 & 0 & 0\\
82 & 1 & 1 & 0 & 0 & 0 & 0 & 0\\
84 & 4 & 2 & 0 & 1 & 1 & 0 & 0\\
86 & 3 & 3 & 0 & 0 & 0 & 0 & 0\\
88 & 2 & 1 & 0 & 0 & 1 & 0 & 0\\
90 & 10 & 6 & 3 & 0 & 1 & 0 & 0\\
92 & 3 & 3 & 0 & 0 & 0 & 0 & 0\\
94 & 6 & 4 & 1 & 0 & 1 & 0 & 0\\
96 & 1 & 0 & 0 & 1 & 0 & 0 & 0\\
98 & 3 & 3 & 0 & 0 & 0 & 0 & 0\\
100 & 1 & 1 & 0 & 0 & 0 & 0 & 0\\
102 & 2 & 1 & 1 & 0 & 0 & 0 & 0\\
104 & 8 & 4 & 3 & 0 & 1 & 0 & 0\\
108 & 3 & 2 & 1 & 0 & 0 & 0 & 0\\
112 & 1 & 1 & 0 & 0 & 0 & 0 & 0\\
114 & 2 & 2 & 0 & 0 & 0 & 0 & 0\\
118 & 2 & 1 & 1 & 0 & 0 & 0 & 0\\
122 & 2 & 2 & 0 & 0 & 0 & 0 & 0\\
132 & 6 & 5 & 1 & 0 & 0 & 0 & 0\\
134 & 1 & 1 & 0 & 0 & 0 & 0 & 0\\
136 & 1 & 1 & 0 & 0 & 0 & 0 & 0\\
140 & 1 & 1 & 0 & 0 & 0 & 0 & 0\\
144 & 1 & 1 & 0 & 0 & 0 & 0 & 0\\
146 & 3 & 3 & 0 & 0 & 0 & 0 & 0\\
150 & 1 & 1 & 0 & 0 & 0 & 0 & 0\\
156 & 1 & 1 & 0 & 0 & 0 & 0 & 0\\
160 & 1 & 1 & 0 & 0 & 0 & 0 & 0\\
164 & 1 & 1 & 0 & 0 & 0 & 0 & 0\\
174 & 2 & 2 & 0 & 0 & 0 & 0 & 0\\
216 & 1 & 1 & 0 & 0 & 0 & 0 & 0\\
\addlinespace[0.5em] \midrule
 & 153 & 101 & 28 & 7 &14 & 2 & 1 \\
\bottomrule
\end{tabular}
}
\bigskip
\caption{\small %
  Betti numbers and almost-diffeomorphism types
  of $2$-connected twisted connected sum \gtmfd s $M$ constructed by
  perpendicular gluing from pairs of rank $1$ Fano $3$-folds.
  $b^{3}(M)=b^{4}(M)=b+23$; 
  $\#$ gives number of instances of a given value of $b$, further broken down
  according to divisibility of $p_1(M)$ on right of table.}
\label{tab:rank1}
\end{table}

\begin{enumerate}
\item
46 different values of $b^3(M)$ arise with $71 \le b^3(M) \le 239$.
\item
All six permitted integers $\{4,8,12,16,24,48\}$ occur as 
$\gdiv p_1(M)$  for some $M$ in Table \ref{tab:rank1}.
\item
82 different almost-diffeomorphism types occur.
\item
By Corollary \ref{c:inertia:g2} the diffeomorphism type is uniquely determined
except in the $14$ cases in which $\gdiv p_1(M) = 16$ and the $1$ case in which
$\gdiv p_1(M)=48$.
\item There are exactly two ways to construct a $2$-connected
  $\gtwo$-manifold with $b^3(M)= 76+23=99$ and $\gdiv p_1(M) = 16$:
either take both blocks from the family Example
\ref{exg:species1}$^2_4$, or match \ref{exg:species1}$^2_1$ to
\ref{exg:species1}$^1_{16}$.
By Corollary \ref{c:inertia:g2} there are precisely two diffeomorphism
classes in the almost diffeomorphism type of such a $2$-connected
$7$-manifold $M$. A natural question is therefore: are these two
almost diffeomorphic twisted connected sum $\gtwo$-manifolds
diffeomorphic or not? To answer it would probably require the calculation of
a generalised Eells-Kuiper invariant as discussed in Remark \ref{rmk:EK}.
(Note this is the only way that homeomorphic but not diffeomorphic
$2$-connected $7$-manifolds could arise from primitive perpendicular
gluing of pairs of building blocks of rank one Fano type.)

\item There are many ways to use primitive perpendicular gluing of
  different pairs of building blocks of rank one Fano type to produce
  \emph{diffeomorphic} 2-connected $\gtwo$-manifolds; in Table
  \ref{tab:rank1} we simply look at any of the four columns where
  $\gdiv p_1 \mid 24$ and find an entry in any row in that column
  which is greater than $1$. There are many such entries in the table.
  Of the $46$ values of $b^3$ that occur in Table \ref{tab:rank1},
  $15$ of those can occur for a unique choice of pair of rank $1$
  Fanos. For the remaining $31$ values of $b^3$ we see that except for
  four cases ($b=78, 88, 102, 118$ in the table; recall $b^3=b+23$
  there) we can find at least two different pairs of rank one Fano
  building blocks that yield diffeomorphic 2-connected 7-manifolds
  with $b^3(M)=b^3$.

  One concrete way to get distinct pairs of building blocks of rank
  one Fano type which yield diffeomorphic $\gtwo$-manifolds is to take
  the pair (a) (\ref{exg:species1}$^1_{22}$,
  \ref{exg:species1}$^1_{22}$) or the pair
  (b)~(\ref{exg:species1}$^1_{22}$,~\ref{exg:species1}$^1_{18}$)
  These both yield a 2-connected $\gtwo$-manifold $M$ with
  $b^3=48+23=71$ and $\gdiv p_1=4$.  (The pairs
  (\ref{exg:species1}$^1_{22}$, \ref{exg:species1}$^1_{16}$) and
  (\ref{exg:species1}$^1_{18}$, \ref{exg:species1}$^1_{16}$) are the
  two other pairs yielding the same $7$-manifold $M$.) By Remark
  \ref{r:2connected:7mfd} $M$ is diffeomorphic to the connected sum of
  $M_{1,0}$ with $70$ copies of $\Sph^3 \times \Sph^4$ where $M_{1,0}$
  denotes the unique $\Sph^3$-bundle over $\Sph^4$ with Euler number
  $0$ and $p_1(M)=4\cdot 1 \in H^4(M_{1,0}) \cong \Z$.
\end{enumerate}

\subsubsection{Detailed examples}

We now describe in detail a small number of examples to illustrate
some of the main points. Consulting the overview given at the beginning of 
Section \ref{sec:close} may also benefit the reader.

The first example shows one way in which it is possible to produce different 
\gtwo-manifolds from the same pair of building blocks $Z_{\pm}$.

\rslsub{g2:qqtorsion}

We take both $Z_+$ and $Z_-$ to be building blocks of Fano type
obtained from a smooth quartic in $\CP^4$ (Example \ref{exg:species1}$^1_4$).
Table \ref{tab:rank1} already includes the twisted connected sum of these
two blocks given by embedding $W = N_+ \perp N_- \cong \gen{4} \perp \gen{4}$
primitively in $L$; the entry has $b = 132$, $\gdiv p_1 = 8$. However, we can
also consider a \emph{non-primitive} embedding of $W$ in $L$ 
for which the resulting inclusions $N_\pm \into L$ are still primitive:
$W$ is isometric to the index 2 sublattice $\{(x,y) : x = y \mod 2 \}$ of
$\gen{2} \perp \gen{2}$, so a primitive embedding of the latter
in~$L$ (which exists by Theorem \ref{thm:exist_emb}, or indeed by inspection)
gives an embedding $W \into L$ with cotorsion $L/W \cong \ZZ/2\ZZ$. Using this
``non-primitive'' perpendicular matching we get a twisted connected sum with
the same Betti numbers and $\gdiv p_1$ as before, but now
$\Tor H^3(M) \cong \ZZ/2\ZZ$ (recall Corollary \ref{cor:torsion}).
In particular, although Theorem \ref{thm:g2topology} shows that 
$M$ is simply-connected and has $H^{2}(M)=0$ it is no longer $2$-connected.

\begin{remark}
\label{rmk:h3tor}
In a similar way, one can get alternative perpendicular matchings with torsion
in $H^3$ for many other pairs of building blocks, whether of
rank 1 Fano type or otherwise. Whether there exist suitable overlattices of
$N_+ \perp N_-$ reduces to a problem about the discriminant groups of $N_+$
and $N_-$, as discussed in Remark \ref{rmk:disc}.
Carrying out such an analysis for the rank 1 pairs one allows us to realise $\cg{k}$ as $\Tor H^3(M)$
of twisted connected sums for $2 \leq k \leq 5$, and a total of 41 triples
of invariants $(b^4(M), \, \gdiv p_1(M), \, k)$ (in addition to the 82
with $H^3(M)$ torsion-free).
\end{remark}

Our remaining examples use building blocks from Tables \ref{tableg:blocks} and \ref{table:rank2}.

\rslsub{g2:q}
We take $Z_+$ to be the building block from Example \ref{exg:quartic_w_plane},
and $Z_-$ from Examples \ref{exg:quartic_w_plane}--\ref{exg:quartic_w_22}
and use primitive perpendicular gluing to achieve matching.
In all these cases the polarising lattices $N_{\pm}$ have signature $(1,1)$ 
and hence \mbox{$W:= N_{+} \perp N_{-}$} has signature $(2,2)$.
Therefore by Theorems \ref{thm:exist_emb} and \ref{thm:unique_emb} $W$ admits a
primitive embedding $W \into L$ which is unique up to automorphisms of $L$. Now
we apply Proposition \ref{prop:orth_gluing} to solve the matching problem 
noting that hypothesis (iii) is automatically satisfied because we are using
perpendicular gluing. Observe that when we choose $Z_-$ from
\ref{exg:quartic_w_plane}, \ref{exg:quartic_w_scroll} or
\ref{exg:quartic_w_22}, $p_1(M)$, and hence the diffeomorphism type of~$M$,
depends on the choice of resolution used for the semi-Fanos.

\rslsub{g2:dp3q}

We match blocks from Example \ref{exg:P3_deg} and
Example \ref{exg:species1}$^1_4$ by primitive perpendicular gluing. Because
Example~\ref{exg:P3_deg} has $\rk K = 3$, the twisted connected sum \gtmfd{}
has $b^2(M) = 3$.

\rslsub{g2:v22q}

We use perpendicular gluing to match the semi-Fano type blocks $Z_{\pm}$
from Examples \ref{exg:toricV22_a} and \ref{exg:doubleline} respectively.
In this case we cannot  appeal to Theorem \ref{thm:exist_emb}(i) 
to guarantee we can embed $W = N_+ \perp N_-$ in $L$ because
$\rk W = 2 + 10 = 12 > 22/2$.
However, the discriminant group $W^*/W$ is 
$\cg{8} \times \cg{4} \times \cg{4} \times \cg{2}$, which is generated
by 4 elements. So $\ell(W) = 4$ and we can apply \ref{thm:exist_emb}(ii) to get
a primitive embedding (and it is unique by Theorem \ref{thm:unique_emb}).

We have identified 9 rigid $\CP^1$'s in $Z_+$ and 12 in $Z_-$.
Using Proposition \ref{prop:cx_to_assoc} we thus find 21 associative
$\Sph^1 \times \Sph^2$ in $M$.
The 12 from $Z_-$ come in pairs that are close together,
as they arise from pairs of $\CP^1 \subset Z_-$ that intersect. However,
there is no a priori reason that the associatives in $M$ should intersect
after the perturbation in Proposition \ref{prop:cx_to_assoc}. 

\rsl{g2:b1}
\rsl{g2:b1t}

\subsubsection{No \ref{g2:b1}--\ref{g2:b1t}}

In these examples, we use perpendicular gluing to match a block $Z_+$ 
arising from the Burkhardt quartic (Example~\ref{exg:burkhardt_quartic}) 
with a block $Z_-$ of Fano type arising from
a Fano 3-fold of Picard rank 1 (Example~\ref{exg:species1}). %
Let $r$ and $d$ be the rank and degree of the Fano 3-fold used.

The polarising lattice $N_+$ of the Burkhardt quartic block $Z_+$ has rank 16,
while $N_-$ is generated by a single vector of square-norm $m = rd$. Note
that because $\rk N_- = 1$, we must a priori choose the embeddings
$N_\pm \into L$ to be perpendicular to have any chance of finding
matching data, since this involves finding a (Kähler) class in $N_-$
that is orthogonal to $N_+$.  So we seek embeddings of $W: = N_+ \perp
\gen{m}$ in the K3 lattice $L$, so that each of the two sublattices
$N_{\pm}$ is primitive in $L$; recall however, that we do not insist
that the embedding of the whole lattice $W$ is primitive in $L$.
Because of the high rank of $W$ some work is required to demonstrate
existence of such an embedding and for this we will need to use
precise information about the lattice $N_{+}$.  Recall from Example
\ref{exg:burkhardt_quartic} that $N_+$ has a unique
primitive embedding in $L$; its orthogonal complement in $L$ is $T =
A_2(-1) \perp 2U(3)$. The problem is therefore equivalent to finding a
primitive vector $x \in T$ of square-norm $m$, so that we can take the image
of $N_-$ to be~$\gen{x}$.  (Theorem \ref{thm:exist_emb} is of no use
for finding the primitive embedding $N_- \into T$ since $T$ is not
unimodular.)

The discriminant group of $W$ is simply the product of the discriminant groups of the two orthogonal summands
\[ W^*/W \cong (\ZZ/3\ZZ)^5 \times \ZZ/m\ZZ. \]
Consider first the case when $3 \mid m$. Then $\ell(W) = 6$.
Since $\rk W = 17$, \eqref{eq:nec_emb} is not satisfied, and there can be
no \emph{primitive} embedding $W \into L$. On the other hand, we can certainly find a
primitive vector $x$ of square-norm $m$ in $U(3) \subset T$, and thus we get
embeddings $W \into L$ that allow us to match Example
\ref{exg:burkhardt_quartic} to \ref{exg:species1}$^1_6$,
\ref{exg:species1}$^1_{12}$, \ref{exg:species1}$^1_{18}$,
\ref{exg:species1}$^2_3$ or \ref{exg:species1}$^3_2$. We label these
examples \textit{No~\ref{g2:b1t}a-e}.
In all five cases $\Tor L/W \cong \ZZ/3\ZZ$ by Lemma \ref{lem:torsion}, so the
resulting \gtmfd s have $\Tor H^3(M) \cong \ZZ/3\ZZ$.

If $m$ is not divisible by 3, then
$(\ZZ/3\ZZ)^5 \times \ZZ/m\ZZ \cong (\ZZ/3\ZZ)^4 \times \ZZ/3m\ZZ$
and $\ell(W) = 5$. Therefore we are just on the borderline where the existence
of a primitive embedding $W \into L$ is not excluded by \eqref{eq:nec_emb}, but
also not guaranteed by Theorem \ref{thm:exist_emb}.
In fact, all elements of $A_2(-1) \perp 2U(3)$ have square-norm $0$ or $1$
mod 3, so if $m = 2$ mod 3 there is no suitable embedding $W \into L$, and therefore we
cannot match Example \ref{exg:burkhardt_quartic} with
\ref{exg:species1}$^1_{2}$, \ref{exg:species1}$^1_{8}$,
\ref{exg:species1}$^1_{14}$, \ref{exg:species1}$^2_1$ or
\ref{exg:species1}$^2_4$ at all.
On the other hand, $A_2(-1)$ does contain a primitive vector of square-norm
$-2$ and $U(3)$ contains vectors of square-norm $3k$ for any $k$; thus, if
$m = 3k-2$ we can find the desired primitive $x \in T$, and the resulting
embedding $W \into L$ is primitive by Lemma \ref{lem:torsion}.
Hence we can match \ref{exg:burkhardt_quartic} to \ref{exg:species1}$^1_4$,
\ref{exg:species1}$^1_{10}$, \ref{exg:species1}$^1_{16}$,
\ref{exg:species1}$^1_{22}$, \ref{exg:species1}$^2_2$, \ref{exg:species1}$^2_5$
and \ref{exg:species1}$^4_1$ using primitive perpendicular gluing to get
2-connected \gtmfd s, which we label \textit{No \ref{g2:b1}a-g}.

Since $\gdiv c_2(Z_+) = 2$, all the \gtmfd s we get this way have
$\gdiv p_1(M) = 4$. Note that \textit{No~\ref{g2:b1}a} and
\textit{\ref{g2:b1}g} are both %
2-connected with $b^3(M) = 95$, so %
are diffeomorphic by Theorem \ref{thm:torfreeclass}.
\textit{No~\ref{g2:b1}c, \ref{g2:b1}d} and \textit{\ref{g2:b1t}c} all have
$b^3(M) = 53$, but \textit{No~\ref{g2:b1t}c} %
has $\Tor H^3(M) = \ZZ/3\ZZ$ so is not diffeomorphic to the first two.

\rslsub{g2:v22v22}

We match two copies $Z_{\pm}$ of blocks from Example~\ref{exg:toricV22_b} using perpendicular gluing. 
Let $N_0=\langle 8\rangle\perp \langle -16\rangle$,
and let $N_{\pm}$ be two copies of the polarising lattice $E_8(-1)\perp N_0$ of the block. We
need to construct an embedding of $N_+\perp N_-$ in the K3 lattice $L$. First
we embed $2N_0$ in $3U$ by the matrix
\[
\begin{pmatrix}
  4 & 0 & 0 &-4\\
  1 & 0 & 0 & 1\\
  0 &-8 & 0 & 0\\
  0 & 1 & 0 & 0\\
  0 & 0 & 4 &-4\\
  0 & 0 & 1 & 1
\end{pmatrix}.
\]
Note that each of the two copies of $N_0$ is embedded primitively, but
$3U/2N_0 \cong \ZZ^2\oplus (\ZZ/8\ZZ)$.
(For a finite index overlattice of $2N_0$ to be primitively embeddable in $3U$
its discriminant group can have at most 2 generators according
to~\eqref{eq:nec_emb}; since the discriminant group of $2N_0$ is $(\ZZ/8\ZZ
\times \ZZ/16\ZZ)^2$ such an overlattice must have index at least 8, so there
is no way to embed $2N_0$ into $3U$ with smaller cotorsion.)

Next we embed $N_+\perp N_-$ in $L=3U \perp 2E_8(-1)$ by embedding
$N_0\perp N_0$ in $3U$ as above, the first copy of $E_8(-1)$ in the first copy
of $E_8(-1)$, and the second in the second.
By Corollary \ref{cor:torsion} $\Tor{H^3}$ of the glued \gtmfd{} is
$\ZZ/8\ZZ$.
Since $N_+$ is embedded perpendicular to $N_-$ there will be no torsion in
$H^4$ of the \gtmfd s.

\rsl{g2:cotorsion}
\subsubsection{No \ref{g2:cotorsion}: orthogonal gluing with large cotorsion}
We use a pair of building blocks $Z_{\pm}$ of semi-Fano type 
obtained from the construction of Example~\ref{exg:quartic_w_22}, \ie starting with a quartic
3-fold containing a quartic del Pezzo surface $F=F_{2,2}$ (the
complete intersection of two quadrics). 
We aim to use ``non-primitive'' perpendicular gluing to achieve ``cotorsion'' as large as possible.
The polarising lattice $N_+\cong N_-$ is the
integral lattice with matrix
\[
\begin{pmatrix}
  4&4\\
  4&0
\end{pmatrix}
\]
and discriminant $(\ZZ/4\ZZ)^2$. We
construct an explicit embedding of $W=N_+ \perp N_-$ in $L$ with cotorsion
$L/W=(\ZZ/4\ZZ)^2$---the largest compatible with the requirement that
both $N_\pm\subset L$ be primitive embeddings. Consider the lattice
\[
W\cong \ZZ^4, \quad
\text{with intersection matrix}
\quad
B=
\begin{pmatrix}
  4 & 4 & 0 & 0\\
  4 & 0 & 0 & 0\\
  0 & 0 & 4 & 4\\
  0 & 0 & 4 & 0
\end{pmatrix}.
\]
Then embed $W$ in $2U$ via the matrix 
\[
\begin{pmatrix}
  2 & 0 &-2 & 0 \\
  1 & 1 & 0 &-1 \\
  2 & 0 & 2 & 0 \\
  0 & 1 & 1 & 1
\end{pmatrix}.
\]
We can check that the embedding is isometric, that the restrictions
to $N_\pm$ are primitive and that $(2U)/W \cong (\ZZ/4\ZZ)^2$. Next, compose
with the obvious primitive embedding $2U \hookrightarrow L$.

More abstractly, we could use Nikulin's theory of lattices
\cite[\S1]{nikulin:quadratic}. $N$ is anti-isometric to itself, and hence so
is the form on its discriminant group $(\cg{4})^2$. Therefore Remark
\ref{rmk:disc} immediately provides overlattices $W'$ of $W$ with $W'/W$ any of
the six subgroups of $(\ZZ/4\ZZ)^2$.

Similar principles are at work in \textit{No \ref{g2:v22v22}} (there $N_0$
is anti-isometric to itself).

\rsl{g2:mm}
\subsubsection{No \ref{g2:mm}: orthogonal gluing with nontrivial intersection}

For this family of examples we glue orthogonally (but not
perpendicularly) building blocks $Z_{\pm}$ of rank two Fano type, \cf
Example~\ref{exg:MoriMukai}.  Note that we could of course choose a
primitive embedding of the signature $(2,2)$ lattice $N_{+} \perp
N_{-}$ into $L$ and therefore match $Z_{\pm}$ by perpendicular gluing.
As we have seen this would yield $2$-connected $7$-manifolds with
torsion-free $H^{4}$.  Instead here we choose to use orthogonal gluing
where the intersection $\ncap = N_+ \cap N_-$ has rank one; this
will give rise to a series of examples with $H^{2}(M) \cong\Z$ and
illustrates again how the same pair of building blocks---matched in
different ways---yields different smooth $7$-manifolds.

We will use the rank two Fanos which are No 2, 6, 10, 12, 21 and 24
from the Mori-Mukai list.  Table \ref{table:rank2} summarises the
information we need about these rank two Fanos; the Picard lattices
$N$ of $Y$ are computed in a basis $L,M$ of supporting divisors, \ie
the (closure of the) ample cone of $Y$ is spanned by $L$ and $M$.
\begin{table}[ht]
  \centering
\[
  \begin{array}{ccccccc} \toprule
No & -K_Y^3 & H^2(Z) & N & H^3(Y) & H^3 (Z) &
\begin{array}{c}\gdiv c_2(Z) \\ \textrm{mod } A^\perp\end{array} \\
\midrule
2 & 6 & \ZZ^3 & \begin{pmatrix}0&2\\2&2\end{pmatrix}&\ZZ^{40} & \ZZ^{48} & 6\\
6  &12 &\ZZ^3 &\begin{pmatrix}2&4\\4&2\end{pmatrix}&\ZZ^{18}&\ZZ^{32} & 12\\
10&16 &\ZZ^3 &\begin{pmatrix}8&4\\4&0\end{pmatrix}&\ZZ^{6} &\ZZ^{24} & 8\\
12&20 &\ZZ^3 &\begin{pmatrix}4&6\\6&4\end{pmatrix}&\ZZ^{6} &\ZZ^{28} & 4\\
21&28 &\ZZ^3 &\begin{pmatrix}6&8\\8&6\end{pmatrix}&(0)&\ZZ^{30} & 4\\
24 & 30 & \ZZ^3 &\begin{pmatrix}2&5\\5&2\end{pmatrix}&(0) & \ZZ^{32} & 12 \\
\bottomrule
  \end{array}
\]
\caption{Some building blocks from rank 2 Fanos}
\label{table:rank2}
\end{table}

In all cases we choose $A=L+M$ as our ample class in the lattice
(this coincides with $-K_Y$ except for No 24, where $-K_Y = 2L + M$)
and push out along a common $\ncap = A^\perp$.
To verify that the pushout exists, we present $N$ as an overlattice of
$\gen{A} \perp \ncap$:
{\allowdisplaybreaks \begin{align*}
\begin{pmatrix} 0&2\\2&2\end{pmatrix}&=\frac1{3}(1,1)\ZZ+\ZZ^2 \quad
\text{in} \quad \begin{pmatrix}6&0\\0&-6 \end{pmatrix}\\
\begin{pmatrix} 2&4\\4&2\end{pmatrix}&=\frac1{2}(1,1)\ZZ+\ZZ^2 \quad
\text{in} \quad \begin{pmatrix}12&0\\0&-4 \end{pmatrix}\\
\begin{pmatrix} 8&4\\4&0\end{pmatrix} &=\frac1{4}(3,1)\ZZ+\ZZ^2 \quad
\text{in} \quad \begin{pmatrix}16&0\\0&-16 \end{pmatrix}\\
\begin{pmatrix} 4&6\\6&4 \end{pmatrix}&=\frac1{2}(1,1)\ZZ+\ZZ^2 \quad
\text{in} \quad \begin{pmatrix}20&0\\0&-4 \end{pmatrix}\\
\begin{pmatrix} 6&8\\8&6\end{pmatrix}&=\frac1{2}(1,1)\ZZ+\ZZ^2 \quad
\text{in} \quad \begin{pmatrix}28&0\\0&-4 \end{pmatrix} \\
\begin{pmatrix} 2&5\\5&2\end{pmatrix}&=\frac1{2}(1,1)\ZZ+\ZZ^2 \quad
\text{in} \quad \begin{pmatrix}14&0\\0&-6 \end{pmatrix}.
\end{align*}
We} see that we can form \gtmfd s $M$ with
$H^2(M) = \ncap = N_+ \cap N_- \cong \ZZ$ by
matching any pair taken from Nos 6, 12 and 21, matching 10 to itself, and 2 to 24.
In each case the image of $N_\pm$ in $N_\mp^*$ is primitive, so there is no
contribution to the torsion of $H^4(M)$.

To compute $p_1(M)$, Corollary \ref{cor:p1gcd} explains that we need to find
the greatest divisor of $c_2(Z_\pm)$ modulo the image of $\ncap = A^\perp$ in
$N_\pm^* \subset H^4(Z_\pm)$. By \cite[Lemma 5.18]{chnp1}, this is
the greatest common divisor of 24 and $c_2(Y_\pm) + c_1(Y_\pm)^2$ modulo the
image of $\ncap$ in $N_\pm^* \cong H^4(Y)$. The latter is determined by the
restriction of $c_2(Y_\pm) + c_1(Y_\pm)^2$ to divisors in the orthogonal
complement to~$\ncap$, \ie just to $A$ itself. For the cases where we use
$A = -K_Y$, the relation $c_2(Y)(-K_Y) = \chi(K3) = 24$ implies that
$\gdiv(c_2(Z) \mod A^\perp) = \gcd(24, -K_Y^3)$.

For No 24 we must do a little more work. This Fano 3-fold is a generic
bidegree (1,2) divisor in $\bbp^2 \times \bbp^2$. The class $A$ has bidegree
(1,1). The projection of a generic divisor in the class to the second $\bbp^2$
factor contracts 7 $(-1)$ curves, so $c_2(A) - c_1(A)^2 = 8$.
By \cite[Lemma 5.15]{chnp1}, the evaluation of $c_2(Y) + c_1(Y)^2$
on $A$ is $8 + 28 = 36$, so $\gdiv(c_2(Z) \mod A^\perp) = 12$.

\rsl{g2:h4tor}
\subsubsection{No \ref{g2:h4tor}: torsion in $H^4$}
In this example, we use orthogonal gluing with non-trivial intersection
arranged so that there is some torsion in $H^4$ of the twisted connected sum.
We take both $Z_+$ and $Z_-$ to be building blocks from
Example \ref{exg:2conics}, that is $\CP^3$ blown up at the components $C_1$,
$C_3$, $C_2$ (in that order) of the base locus $C$ of a pencil of quartics
containing a fixed pair of conics $C_1$, $C_2$.

In the notation from Example \ref{exg:2conics},
the triple $A-C_1-C_2, \, C_1-C_2, \, A$ spans an index 2 sublattice
$N' \subset N$ with intersection form
\[ N' = \begin{pmatrix}-8 & 0 & 0 \\ 0 & -4 & 0 \\ 0 & 0 & 4\end{pmatrix} . \]
$N = N' + \half(1,1,1)$, and we can form an orthogonal push-out
$W = N_+ \perp_\ncap N_-$, identifying the sublattices $\ncap \cong \gen{-8}$
spanned by $A-C_1-C_2$ in each copy $N_\pm$. Note that the image of $N_\pm$
in $N_\mp^*$ has cotorsion $\cg2$. Therefore the twisted connected sum has
$\Tor H^4(M) \cong (\cg2)^2$ by Lemma \ref{lem:torsion}.

To apply Proposition \ref{prop:orth_gluing} to find matching data, we need to
check that $\Amp_\fbb \cap W \not= \emptyset$, where
$W = \ncap^\perp = \gen{A, C_1 - C_2} \subset N_\bbr$.
From the analysis of $\Amp_\fbb$ in Example \ref{exg:2conics}, we see that
$A + \lambda(-C_1 + C_2) \in \Amp_\fbb \cap W$ for small $\lambda > 0$.

Because $H^4(M)$ has only 2-torsion, and $p_1(M)$ is
divisible by 4 a priori, $\gdiv p_1(M)$ is the same as the greatest divisor
of the image of $p_1(M)$ is the free part of $H^4(M)$.
To compute the latter, it suffices by Proposition \ref{prop:p1y} and Lemma
\ref{lem:ymv} to find $\gdiv c_2(Z_\pm)$ modulo the primitive
overlattice of the image of $N$ in $H^4(Z_\pm)$.
This amounts to evaluating $c_2(Z_\pm)$ on divisors
representing classes in $H^2(Z_\pm)$ whose image in $H^2(S)$ is orthogonal
to $\ncap$. The group of such divisors is spanned by $S$ itself,
$E_1 + E_2 + E_3$ and $E_1 - E_3$. $c_2(Z_\pm)$ evaluated on $S$ is 24 as
usual. On the other two basis elements we see from \cite[Proposition
5.11]{chnp1} that it evaluates to 64 and 20, respectively. Thus
the greatest common divisor is 4, and $\gdiv p_1(M) = 8$.

\newcommand{\summarytable}{
\begin{large}
\begin{table}[tp]
\[
\begin{array}[t]{r@{}lllrrccrcr} \toprule
No & & \ccell{Z_+} & \ccell{Z_-}
& b^2 & b^3 & TH^3 &  TH^4 & a_0 & \text{what} & p_1 \\ \midrule
\ref{g2:qqtorsion}& & Ex~\ref{exg:species1}^1_4 &Ex~\ref{exg:species1}^1_4
& 0 & 155 & 2 &  & 0 & \gen{4} \perp \gen{4} & 8 \\
\ref{g2:q}&a &Ex~\ref{exg:quartic_w_plane}&Ex~\ref{exg:quartic_w_plane}
& 0 & 123 & & & 18 & N_+ \perp N_- & 4,8 \\ 
\ref{g2:q}&b &Ex~\ref{exg:quartic_w_plane}&Ex~\ref{exg:quartic_w_quadric}
& 0 & 117 & & & 21 & N_+\perp N_- & 4 \\
\ref{g2:q}&c &Ex~\ref{exg:quartic_w_plane}&Ex~\ref{exg:quartic_w_scroll}
& 0 & 107 & & & 26 & N_+\perp N_- & 4,8 \\
\ref{g2:q}&d & Ex~\ref{exg:quartic_w_plane}&Ex~\ref{exg:quartic_w_22}
& 0 & 109 & & & 25 & N_+\perp N_-& 4,8 \\
\ref{g2:dp3q}& & Ex~\ref{exg:species1}^1_4&Ex~\ref{exg:P3_deg}
& 3 & 116 & & & 24 & \gen{4} \perp N_- & 4 \\
\ref{g2:v22q}& & Ex~\ref{exg:doubleline}&Ex~\ref{exg:toricV22_a}
& 0 & 93 & & & 21 & N_+ \perp N_- & 4 \\
\ref{g2:b1}&a & Ex~\ref{exg:burkhardt_quartic}&Ex~\ref{exg:species1}^1_4
& 0 & 95 & & & 45 & N_+\perp \gen{4} & 4 \\
\ref{g2:b1}&b & Ex~\ref{exg:burkhardt_quartic}&Ex~\ref{exg:species1}^1_{10}
& 0 & 61 & & & 45 & N_+\perp \gen{10} & 4 \\
\ref{g2:b1}&c & Ex~\ref{exg:burkhardt_quartic}&Ex~\ref{exg:species1}^1_{16}
& 0 & 53 & & & 45 & N_+\perp \gen{16} & 4 \\ 
\ref{g2:b1}&d & Ex~\ref{exg:burkhardt_quartic}&Ex~\ref{exg:species1}^1_{22}
& 0 & 53 & & & 45 & N_+ \perp \gen{22} & 4 \\
\ref{g2:b1}&e & Ex~\ref{exg:burkhardt_quartic}&Ex~\ref{exg:species1}^2_2
& 0 & 67 & & & 45 & N_+ \perp \gen{4} & 4 \\
\ref{g2:b1}&f & Ex~\ref{exg:burkhardt_quartic}&Ex~\ref{exg:species1}^2_5
& 0 & 71 & & & 45 & N_+ \perp \gen{10} & 4 \\
\ref{g2:b1}&g & Ex~\ref{exg:burkhardt_quartic}&Ex~\ref{exg:species1}^4_1
& 0 & 95 & & & 45 & N_+ \perp \gen{4} & 4 \\
\ref{g2:b1t}&a & Ex~\ref{exg:burkhardt_quartic}&Ex~\ref{exg:species1}^1_6 &
0 & 77 & 3 & & 45 & N_+\perp \gen{6} & 4 \\
\ref{g2:b1t}&b & Ex~\ref{exg:burkhardt_quartic}&Ex~\ref{exg:species1}^1_{12}
& 0 & 57 & 3 & & 45 & N_+\perp \gen{12} & 4 \\
\ref{g2:b1t}&c & Ex~\ref{exg:burkhardt_quartic}&Ex~\ref{exg:species1}^1_{18}
& 0 & 53 & 3 & & 45 & N_+\perp \gen{18} & 4 \\
\ref{g2:b1t}&d & Ex~\ref{exg:burkhardt_quartic}&Ex~\ref{exg:species1}^2_3
& 0 & 65 & 3 & & 45 & N_+ \perp \gen{6} & 4 \\
\ref{g2:b1t}&e & Ex~\ref{exg:burkhardt_quartic}&Ex~\ref{exg:species1}^3_2
& 0 & 85 & 3 & & 45 & N_+ \perp \gen{6} & 4 \\
\ref{g2:v22v22}& & Ex~\ref{exg:toricV22_b}&Ex~\ref{exg:toricV22_b}
& 24 & 47 & 8 & & 66 & N_+\perp N_- & 4 \\
\ref{g2:cotorsion}& & Ex~\ref{exg:quartic_w_22}&Ex~\ref{exg:quartic_w_22}
& 0 & 95 & 4{\cdot}4 & & 32 & N_+ \perp N_- & 8 \\
\ref{g2:mm}&a & Ex~\ref{exg:MoriMukai}_2 & Ex~\ref{exg:MoriMukai}_{24}
& 1 & 82 & & & 0  & N_+ \perp_{\gen{-6}} N_- & 12 \\
\ref{g2:mm}&b & Ex~\ref{exg:MoriMukai}_6 & Ex~\ref{exg:MoriMukai}_6
& 1 & 86 & & & 0  & N_+ \perp_{\gen{-4}} N_- & 24 \\
\ref{g2:mm}&c &Ex~\ref{exg:MoriMukai}_{10}&Ex~\ref{exg:MoriMukai}_{10}
& 1 & 70 &  & & 0 & N_+ {\perp_{\gen{-16}}} N_- & 16 \\ 
\ref{g2:mm}&d & Ex~\ref{exg:MoriMukai}_{12}&Ex~\ref{exg:MoriMukai}_{12}
&1 & 78 &  & & 0 & N_+ \perp_{\gen{-4}} N_- & 8 \\
\ref{g2:mm}&e & Ex~\ref{exg:MoriMukai}_{21}&Ex~\ref{exg:MoriMukai}_{21}
&1 & 82 &  & & 0 & N_+ \perp_{\gen{-4}} N_- & 8 \\
\ref{g2:mm}&f & Ex~\ref{exg:MoriMukai}_{6}&Ex~\ref{exg:MoriMukai}_{12}
&1 & 82 &  & & 0 & N_+ \perp_{\gen{-4}} N_- & 8 \\
\ref{g2:mm}&g & Ex~\ref{exg:MoriMukai}_{6}&Ex~\ref{exg:MoriMukai}_{21}
&1 & 84 &  & & 0 & N_+ \perp_{\gen{-4}} N_- & 8 \\
\ref{g2:mm}&h & Ex~\ref{exg:MoriMukai}_{12}&Ex~\ref{exg:MoriMukai}_{21}
&1 & 80 &  & & 0 & N_+ \perp_{\gen{-4}} N_- & 8 \\
\ref{g2:h4tor}& & Ex~\ref{exg:2conics} & Ex~\ref{exg:2conics}
& 1 & 82 & & 2^2 & 40 & N_+ \perp_{\gen{-8}} N_- & 8 \\
\ref{g2:precise}& & Ex~\ref{exg:quartic_w_22}&Ex~\ref{exg:quartic_w_22}
& 0 & 93 & & & 32 & \text{non-orth} & 48 \\
\bottomrule
\end{array}
\]
\caption{A small number of examples of \gtmfd s. $a_{0}$ is the number of rigid associative $3$-folds 
diffeomorphic to $\Sph^1 \times \Sph^2$ we can exhibit.}
\label{table:g2ex}
\end{table}
\end{large}}

\subsection{Handcrafted nonorthogonal K3 gluing}
\label{sec:precise-k3-gluing}

Gluing by means of the orthogonal pushout construction is convenient
because it reduces the problem of finding compatible pairs of ACyl
Calabi-Yau $3$-folds $V_{\pm}$ arising from a given pair of
deformation types of building blocks $\fbb_{\pm}$, essentially to
arithmetic considerations involving the pair of polarising lattices
$N_{\pm}$ of the two families.  This allows us to produce large
numbers of compact \gtwo-manifolds with relatively little labour; see
Section \ref{sec:close} for a discussion of many further such
examples.  However, in most cases arising in practice, we expect to be
able to glue much more generally. Unfortunately, this seems to require
much more precise explicit information about K3 moduli spaces and that
information is very expensive to obtain. Here we just give the
simplest possible example(s).

The general scheme is as follows:
\begin{itemize}
\item Choose semi-Fano deformation types $\sff_{\pm}$ with Picard
  lattices $N_{\pm}$. Also choose $H_\pm \in N_\pm$ that
  correspond to ample (Kähler) classes on the semi-Fanos
  ($H_\pm \in \Amp_{\sff_\pm}$ in the terminology of Proposition
  \ref{prop:k3generic}).
  In the end, we plan to glue blocks $Z_{\pm}$ obtained from
  semi-Fanos from these families by blowing up AC curves.
\item Choose a %
  lattice $W=N_+\oplus N_-$ where $N_+$, $N_-$
  \emph{are not necessarily orthogonal}, but where $W$ has signature
  $(2,r-2)$ and $H_+ \in N_-^{\perp}$ and $H_- \in N_+^{\perp}$.
  Embed $W$ primitively in $L$.
\item Let $\Lambda_+ = H_{-}^\perp \subset W$ and
  $\Lambda_- = H_+^\perp \subset W$.
  Construct projective models for $\Lambda_\pm$-polarised K3s to show that
  the generic K3s can still be found as hyperplane sections of semi-Fanos
  from the starting classes $\sff_\pm$.
\item Among the semi-Fano type building blocks constructed from $\sff_\pm$
  we can therefore find a \emph{subfamily} that is
  $(\Lambda_\pm, \Amp_{\sff_\pm})$-generic in the sense of Definition
  \ref{def:generic} (except that the cone $\Amp_{\sff_\pm}$ is not open
  in $\Lambda_\pm(\bbr)$, but that is unimportant here).
  Since we made sure that $\Amp_{\sff_\pm} \cap \, \Lambda_\mp^\perp$ is
  non-empty, Proposition~\ref{prop:orth_gluing} shows that we can glue.
\end{itemize}

\begin{remark*}
Note that even though the K3 fibres have $\Pic S_\pm = \Lambda_\pm$,
the images of $H^2(Z_\pm)$ in $H^2(S_\pm)$ are still $N_\pm$; the topology
of the twisted connected sum involves the embeddings of $N_\pm$ in $L$
and not $\Lambda_\pm$.
\end{remark*}
In the construction of the projective models we use the following well-known:
\begin{lemma}[\mbox{\cite[Chapter~3]{MR1442522}}]
  \label{lem:k3dichotomies}
Let $S$ be a K3 surface, and $A$ a nef line bundle on $S$ with $A^2>0$
(that is, $A$ is nef and big). 

\noindent \textup{(I)} Either:
\begin{itemize}
\item $|A|$ has no fixed part, or:
\item $|A|$ is monogonal, that is: $A=aE+\Gamma$ where $E^2=0$, $E\cdot
  \Gamma=1$, $\Gamma^2=-2$, and $a\geq 1$.
\end{itemize}
\noindent \textup{(II)} Assume that $|A|$ is not monogonal. Then $|A|$ is base
point free and either:
\begin{itemize}
\item the morphism given by $|A|$ is birational onto its image and an
  isomorphism away from a finite union of $-2$ curves, or
\item $|A|$ is hyperelliptic, that is, one of the following cases:
  \textup{(a)} $A^2=2$ and $S$ is a double cover of~$\PP^2$;
  \textup{(b)} $A=2B$ with $B^2=2$ and $S$ is a double cover of the
    Veronese surface; or
  \textup{(c)}~$S$~has an elliptic pencil $E$ with $A\cdot E =2$. 
\end{itemize}
\end{lemma}

\rslsub{g2:precise}

We plan to glue two building blocks obtained by blowing up AC
curves on two semi-Fano 3-folds $Y_{\pm}$ that are small resolutions of
quartic 3-folds $X_\pm$ containing a $2\cdot 2$ complete intersection in
$\PP^4$ (Example \ref{exg:quartic_w_22}).
In a basis $A=-K_Y$, $E$ the Picard lattice $N$ of $Y_\pm$ has quadratic form
\[
\begin{pmatrix}
  4 & 4 \\ 4 & 0
\end{pmatrix} .
\]
Note that $A$ is not ample on $Y$. We change basis to $H=A+E, \; E$ so
the intersection form is
\[
\begin{pmatrix}
  12 & 4 \\ 4 & 0
\end{pmatrix} .
\]
We glue together two copies $N_{\pm}$ of $N$ into a lattice $W$ with
basis $H_{+},E_+,H_-,E_-$ and intersection form
\[
\begin{pmatrix}
  12 & 4 & 0 & 0\\
  4 & 0 & 0 & 1 \\
  0 & 0 & 12 & 4\\
  0 & 1 & 4 & 0
\end{pmatrix} .
\]
Note that $W$ has signature $(2,2)$. 

Let $\Lambda_+=H_-^\perp$ and $\Lambda_-=H_+^\perp$: we want to glue
$\Lambda_+$-polarised K3 surfaces with $\Lambda_-$-polarised K3
surfaces.  Write $\Lambda=\Lambda_+\cong \Lambda_-$: we will show that
a generic $\Lambda$-polarized K3 surface $S$---that is, one for which
$\Pic S =\Lambda$---is always the hyperplane section of a
quartic $X$ as above.

$\Lambda$ has basis $H=H_+, \; E=E_+, \; \Sigma=-H_-+3E_-$ and the
intersection form in this basis is
\[
\begin{pmatrix}
  12 & 4 & 0\\
  4 & 0 & 3  \\
  0 & 3 & -12
\end{pmatrix} .
\]
To study this, it is best to change basis back to $A=H-E, \; E, \; \Gamma=H-E-\Sigma$:
\[
\begin{pmatrix}
  4 & 4 & 7 \\
  4 & 0 & 1\\
  7 & 1 & -2
\end{pmatrix} .
\]
First we show that the class $A$ is ample and not hyperelliptic, and
thus embeds $S$ as a quartic surface in $\PP^3$ containing a
$2\cdot 2$ curve.
Indeed we claim:
\begin{enumerate}
\item There is no vector $v\in \Lambda$ with $A\cdot v=1$ and $v^2 =
  0$;
\item There is no vector $v\in \Lambda$ with $A\cdot v=2$ and $v^2 =
  0$;
\item There is no vector $v\in \Lambda$ with $A\cdot v=0$ and $v^2=-2$,
\end{enumerate}
Indeed, write $v=(x,y,z)$.

\summarytable

For the first we know $4x+4y+7z=1$ and $4x^2 + 8xy + 14xz + 2yz - 2z^2
= 0$. Use the linear equation to get $x$ in terms of $y$ and $z$, then
substitute and clean up. We end up with
\[
16y^2 + 48yz + 57z^2 = 1 .
\]
This is easy to rule out, because the conic is positive definite.
In fact we can just complete the square and write the left hand
side as
\[
(4y+6z)^2 + 21z^2 = 1.
\]
This immediately gives $z=0$ (otherwise the left hand side is $\geq  21$),
and hence $16y^2=1$ which is impossible.
Similarly, a counter-example to the second claim gives 
\[
16y^2 + 48yz + 57z^2 = 4 ,
\]
which yields to the same technique. 
The third gives 
\[
16y^2 + 48yz + 57z^2 = 8 ,
\]
which again cannot work for the same reasons.

Consider now the moduli stack $\mathfrak{K}^{\Lambda,A}$ of
$(\Lambda,A)$-polarised K3 surfaces introduced in \cite[\S6]{chnp1}.
This involves a choice of a certain partition $\Delta^+ \sqcup \Delta^-$ of the
set $\Delta = \{\delta \in N \mid \delta^2 = -2\}$; by (iii),
we can take $\Delta^+=\{\delta \in \Delta \mid A\cdot \delta >0\}$.
It follows from Lemma~\ref{lem:k3dichotomies}
above that if $S$ is a generic surface of the moduli stack---that is,
one for which $\Pic S = \Lambda$ exactly---then $A$ is ample
on $S$ and that it embeds $S$ as a quartic in $\PP^3$. 

All this goes to show that $S$ embeds in $\PP^3$ as a nonsingular K3 with an
equation of the form
\[
a_2b_2+c_2d_2=0 .
\] 
Now consider $\PP^3$ as $\{x_4=0\} \subset \PP^4$: it is elementary to see
that, if $\tilde{a}_2, \tilde{b}_2, \tilde{c}_2, \tilde{d}_2$ are
general forms in $x_0,\dots, x_4$ that give $a,b,c,d$ when restricted
to $\PP^3$, then 
\[
X = \{ \tilde{a}_2\tilde{b}_2+ \tilde{c}_2 \tilde{d}_2=0 \} \subset \PP^4
\]
is a 3-fold of the required type containing $S$ as a hyperplane
section.

To work out $p_1$, we need to understand $c_2(Z_+)$ modulo the image of
$N_-$ in $N_+^* \into H^4(Z_+)$ and vice versa.
By \cite[Lemma 5.18]{chnp1}, we
can compute it by taking the greatest common divisor of 24 and
$c_2(Y_1) + c_1(Y_1)^2$ evaluated on divisors in $N_+ \cap N_-^\perp$.
In this case, $N_+ \cap N_-^\perp$ is generated by $H_+ = A_+ + E_+$.
The restriction of $c_2(Y_1) + c_1(Y_1)^2$ to the first term is
$\chi(K3) + (-K_Y)^3 = 24 + 4 = 28$ (since $A$ is just $-K_Y$), while it is
computed in \cite[Example 7.6]{chnp1} that the restriction
to $E_+$ is 44.
Hence $\gdiv(c_2(Z_+) \mod \Imag(N_-)) = \gcd(24, \, 28+44) = 24$.

It is straightforward to assemble the remaining topological information for the
entry in Table \ref{table:g2ex}.
Note that the usual relation for $b^2(M) + b^3(M)$
(Lemma \ref{lem:b2b3}) is not satisfied since the gluing is not orthogonal;
in particular the value of $b^{3}(M)$ is different from \textit{No
\ref{g2:cotorsion}}, even though that example uses the same building blocks.

\subsection{Obstructed associatives}

Let us illustrate how one can apply Proposition
\ref{prop:lag_to_assoc} to construct families of associatives -- including some
obstructed ones -- in compact \gtmfd s from special Lagrangian rational
homology spheres in building blocks. The easiest way to exhibit concrete
examples of the latter is to use real algebraic geometry; complex
conjugation on an algebraic variety is an antiholomorphic involution, and
the fixed set of an antiholomorphic involution on a Calabi-Yau manifold
is special Lagrangian (with some phase). If we construct a building block
from a real (semi) Fano $Y$ by blowing up a real anticanonical curve, then the
building block also has a real structure. A component $L$ of the real locus of $Y$ 
not meeting the chosen anticanonical divisor gives rise to a
special Lagrangian in the ACyl Calabi-Yau $V$.
To apply Proposition \ref{prop:lag_to_assoc} we require that $b^1(L) = 0$
and that $[L] \in H_3(Y; \bbr)$ is non-zero (see Remark \ref{rmk:lag_homology}).

Given a building block with a suitable special Lagrangian, we still need
to match it to another building block to construct a \gtmfd.
The `orthogonal gluing' argument is unfortunately not very compatible with the
use of real algebraic geometry, so it is not so easy to write down a list of
building blocks containing special Lagrangian spheres and claim that each can
be matched with a list of building blocks. Instead we limit ourselves to
showing that for some examples we can find at least some matching.

When we construct an $\Sph^1$-family of associatives like this, with a
map $f$ from $\Sph^1$ to a \mbox{1-parameter} family of \gtstr s that specifies by
which \gtstr{} the members of the family are calibrated, critical points of $f$
correspond to associatives with obstructions. If the entire $\Sph^1$-family is
associative with respect to the same \gtstr{} (\ie $f$ is constant) then they
are all obstructed, but one would expect that the critical points of $f$ are
isolated.
As one moves in the 1-parameter family of \gtstr s and approaches a local
extreme value, two associatives move together, coincide as a single obstructed
associative, and then disappear (or vice versa).

\begin{example}
\label{exg:quartic:sl}
We consider a particular block $Z_+$ from Example \ref{exg:species1}$^1_4$.
Let $Y$ be the quartic \mbox{3-fold} in $\bbc P^4$ defined by
$Q(X) = -X_0^4 + X_1^4 + X_2^4 + X_3^4 + X_4^4 = 0$.
Its real locus $L$ is homeomorphic
to $\Sph^3$ and does not meet the anticanonical divisor $X_0 = 0$.
If we blow up the intersection of $X_0 = 0$ and another real hyperplane section
of $Y$ to form $Z_+$, then $Z_+$ has a real structure and an anticanonical
divisor $S_+$ that does not meet the real locus $\Sph^3$. We can give
$V_+ = Z_+ {\setminus} S_+$ an ACyl Calabi--Yau metric that is invariant under
the real structure $\sigma$, and it then contains a special Lagrangian
$L \cong \Sph^3$.

More precisely, up to sign $V_+$ has a unique holomorphic volume form
$\Omega = \alpha +i\beta$ such that $\sigma^*\Omega = \bar\Omega$, and
$L$ (correctly oriented) is calibrated by $\alpha$. On the cylindrical
end $\bbrp \times \Sph^1 \times S$ we can write the Kähler form as
$dt \wedge d\anglen + \omega^I$ and
$\alpha = d\anglen \wedge \omega^J + dt \wedge \omega^K$.
The real structure $\sigma$ on $S_+$ preserves $\omega^K$ and reverses
$\omega^I$ and $\omega^J$. On the other hand, the involution
$\mbox{$(X_1 : X_2 : X_3 : X_4)$} \mapsto (X_2 : X_1 : X_3 : X_4)$ has fixed set of
dimension 1, so defines a non-symplectic isometry $\tau$ on $S_+$,
\ie $\tau^*\omega^I = \omega^I$
while $\tau^*(\omega^J + i\omega^K) = -(\omega^J + i \omega^K)$.
Under a \hk rotation $S_+ \to S_-$ (Definition \ref{def:hkrot}),
$\tau\sigma$ therefore corresponds to a non-symplectic involution.
The fixed set of $\tau\sigma$ is homeomorphic to 2 copies of $\Sph^2$ (any
point in the fixed set can be written uniquely as
$(e^{i\theta} : e^{-i\theta} : x_3 : x_4)$, with $\theta \in [0,\pi)$,
$x_3, x_4 \in \bbr$ and $x_3^2 + x_4^2 = 2\real(e^{i4\theta})$), and the
quotient $(S_- \times \bbc)/\tau\sigma$ can be resolved by blow-up to form an
ACyl Calabi--Yau $V_-$ of non-symplectic type
(\cf Remark \ref{rmk:kl-def}) that is compatible with $V_+$.

To apply Proposition \ref{prop:lag_to_assoc}, we need to check that $L$ is not
homologous to $0$ in $Y$. If it is, then so is
its preimage $\hat L$ in $H_4$ of the unit normal bundle of $Y$ in $\bbc P^4$.
For any homogeneous polynomial $P$ of degree 3,
$\frac{P}{Q^2}(X_0dX_1dX_2dX_3dX_4 - X_1dX_0dX_2dX_3dX_4 + \ldots)$
defines a meromorphic 4-form on $\bbc P^4$ with poles only on $Y$.
Using the affine chart $X_0 = 1$, its integral over $\hat L$ reduces to
$\int_L P(X) X \lrcorner dX_1dX_2dX_3dX_4$, which is non-zero
if we choose \eg $P = X_0(X_1^2 + X_2^2 + X_3^2 + X_4^2)$.
\end{example}

\begin{remark*}
For $L$ to be homologically non-zero in $Y$ it was sufficient to find one
meromorphic form with poles on $Y$ and non-zero integral over $\hat L$.
In fact this condition is also necessary (Griffiths \cite{griffiths69}).
\end{remark*}

\begin{remark*}
In the example above, the compact \gtmfd{} $(M,\varphi_0)$ has a
\gtwo-involution $\psi$, which acts on the first half $\Sph^1 \times V$ by
$(-1,\sigma)$. $\psi$ acts as a reflection on the $\Sph^1$-family of
associatives in $\Sph^1 \times M$.
We can also choose the 1-parameter family of \gtstr s
$\{\varphi_t: t \in (-\epsilon,\epsilon)\}$ in which the associative family
becomes unobstructed so that $\psi^* \varphi_t = \varphi_{-t}$.
The map $f : \Sph^1 \to (-\epsilon,\epsilon)$ must be $\psi$-equivariant.
That $f$ maps the fixed points of the reflection in $\Sph^1$ to $\varphi_0$
corresponds
to the fact that the fixed locus of $\psi$ in the \gtmfd{} is associative.
But considering only the fixed locus we cannot tell whether it is rigid
or not. By considering the whole $\Sph^1$-family we see that it
contains some obstructed associatives.
\end{remark*}

\begin{example}
\label{exg:nodal:quartic}
Let $Q_0 \subset \bbc P^4$ be a real quartic $3$-fold with a single nodal
singularity. The singular point must then be real. Suppose that the local model
of the singularity is $x^2 + y^2 + z^2 + w^2 = 0$ (in real coordinates).
Then the real locus of a (real) deformation $Q$ of $Q_0$ that smooths
the singularity to $x^2 + y^2 + z^2 + w^2 = \epsilon > 0$ contains an $\Sph^3$.
Because this is the vanishing cycle of a quartic with a single node it is
homologically non-zero.
\end{example}

%% file: sec_close.tex
We now move from concrete examples to a more general discussion of the
possibilities of the construction and some of the prospects for future developments.

\subsection{Overview}
At this point it will probably be helpful to give an overview of what has been achieved so far  
and also to reflect on some of the lessons learned from the examples given in the previous section.
We begin by recalling the main degrees of freedom in the construction.

First we have the \emph{choice of building blocks}: according to Proposition
\ref{prop:block_from_sf}, we can for almost (recall the Assumption before
\ref{dg:semifano}) any Fano or semi-Fano $3$-fold blow up a generic AC pencil
to get a building block of Fano or semi-Fano type (in the sense of Definition
\ref{d:accy:sf}), which has $K=0$ (recall \eqref{eq:k_def}).
For some Fanos and semi-Fanos we can instead choose to blow up a nongeneric AC pencil to obtain building blocks, 
\eg Examples \ref{exg:P3_deg},\,\ref{exg:2conics} and \ref{exg:toricV22_b} give building blocks obtained from 
nongeneric AC pencils on $\CP^{3}$, $\CP^{3}$ and a particular toric semi-Fano $3$-fold respectively.  
As illustrated by these examples, depending on the pencil being blown up these blocks may or may not have $K=0$. 
In such cases care must be taken to ensure that the blocks satisfy the conditions in Definition \ref{def:generic}, which are used in our matching 
arguments (and also the topological properties assumed in our calculations of the cohomology of $M$).
The 74 blocks of non-symplectic type constructed by Kovalev-Lee all have $ K \neq 0$
(however, recall Remark \ref{R:k3:involution:g2}).

Choosing a pair of deformation types of building blocks $\fbb_{\pm}$ fixes the
pair of polarising lattices~$N_{\pm}$.  Let $n_{\pm} = \rk{N_{\pm}}$.
Choosing the building blocks also fixes the number $e_{\pm}$ of compact rigid curves in 
$V_{\pm}=Z_{\pm} \setminus S_{\pm}$.
By Theorem \ref{thm:g2topology}(ii) $b^{2}(M) \ge \rk{K_{+}} + \rk{K_{-}}$. 
In particular, to obtain $\gtwo$-manifolds with $b^{2}(M)=0$ (\eg if we want to
construct $2$-connected manifolds) we must choose blocks with $K_{\pm}=0$.

Next we choose the \emph{method of matching}: perpendicular gluing, orthogonal gluing or handcrafted gluing. 
For simplicity and because it is difficult at this stage to say anything very systematic about handcrafted gluing 
here we stick to commentary on perpendicular or orthogonal gluing. 

\subsubsection{Primitive perpendicular gluing}
Whenever $N_{+}\perp N_{-}$ can be primitively embedded in the K3 lattice $L$ then we can match the building blocks
$\fbb_{\pm}$ by primitive perpendicular matching; this always yields a $2$-connected $7$-manifold with torsion-free 
cohomology to which we may apply the general classification theory described in \S \ref{sec:top}.
$N_{+} \perp N_{-}$ always embeds primitively in $L$ if 
\begin{equation}
\label{E:emb:crude}
n_{+} + n_{-} \le 11
\end{equation}
(see \eqref{eq:exist_emb}). If we are able to compute the lattices $N_{\pm}$ in
detail (and not just their ranks $n_{\pm}$) then we can determine their
discriminant groups and hence determine $\ell= \ell(N_+ \perp N_{-})$; 
by \ref{thm:exist_emb}(ii) $N_{-}\perp N_{+}$ admits a primitive embedding in $L$ if 
\begin{equation}
\label{E:emb:disc}
n_{+}+n_{-} + \ell < 22.
\end{equation}
See \textit{No~\ref{g2:v22q}} for an example satisfying the second inequality
but not the first. %

To summarise: when it applies primitive perpendicular gluing requires little effort and yields $2$-connected 
$7$-manifolds with torsion-free cohomology; it therefore allows us to produce 
many \gtwo-manifolds for which we understand the diffeomorphism type.

\subsubsection{Non-primitive perpendicular gluing}
If $N_{-}\perp N_{+}$ admits a primitive embedding in $L$ then
$n_{+}+n_{-} + \ell \le 22$. 
For some pairs of blocks (\eg the Burkhardt block matched to any semi-Fano or Fano block with Picard rank 
greater than 1) this inequality is violated and hence ${N_{+} \perp N_{-}}$ admits no primitive embeddings in $L$.
Nevertheless, $N_{+}\perp N_{-}$ may admit non-primitive embeddings in $L$ which are primitive 
when restricted to both $N_{+}$ and $N_{-}$ (as exhibited in \textit{No~\ref{g2:b1t}}\,). 
In this case the resulting \gtwo-manifold is simply-connected with $H^{2}=0$ and $\Tor{H^{3}} \neq 0$ 
(and therefore $H_{2}=\pi_{2} \neq 0$).
Even if $N_{-} \perp N_{+}$ does admit a primitive embedding in $L$, it may still admit 
non-primitive embeddings in $L$ which are primitive on each factor~$N_{\pm}$. 
In this case different matchings of the same blocks can produce both $2$-connected and non-$2$-connected \gtwo-manifolds 
with the same Betti numbers, distinguished by the torsion in~$H^{3}$.
As explained in Remark \ref{rmk:disc}, the problem of finding suitable
non-primitive embeddings $N_+ \perp N_- \into L$ reduces to a fairly manageable
analysis of the discriminant groups of $N_\pm$; this is used in
\textit{No \ref{g2:qqtorsion}} and \textit{\ref{g2:cotorsion}} and
Remark \ref{rmk:h3tor}.
Later in this section we describe the prospects of extending the known smooth
classification results to $1$-connected $7$-manifolds with $\pi_{2}(M)$ a
finite cyclic group.

\subsubsection{Orthogonal gluing}
If $\min{(n_{-},n_{+})} > 1$ then we may also attempt to use orthogonal but not perpendicular gluing; 
this will always yield manifolds with $b^{2}>0$. 
In this case we encounter two additional problems. 
The first is the arithmetic problem of finding a
non-trivial lattice $\ncap$ that can be primitively embedded in both $N_+$ and
$N_-$, such that the push-out $W = N_+ \perp_\ncap N_-$ is an integral lattice.
For instance, if we want $\ncap = \gen{-m}$ then we look for primitive vectors
$x_\pm \in N_\pm$ of square-norm $-m$, and Example \ref{exa:no_pushout}
demonstrates that we need that if the image of the
orthogonal projection of $N_\pm$ to $\gen{x_\pm}$ is
$\frac{1}{d_\pm}\gen{x_\pm}$, then $d_+d_- \mid m$.
The second problem is to satisfy condition (iii) in Proposition
\ref{prop:orth_gluing}, that the orthogonal complement of $R$ in $N_\pm$
contains some classes that are ample on the building blocks.
For non-symplectic type blocks it would be enough to check
that $\ncap$ contains no $-2$ classes (\cf Remark \ref{r:w:ample}),
but the semi-Fano case is more subtle.

\begin{example}
\label{ex:nonex}
The polarising lattice $N$ of the block $Z$ in
Example \ref{exg:quartic_w_quadric} can be presented as
\[ \begin{pmatrix} -2&2\\2&4\end{pmatrix} = \frac1{2}(1,1)\ZZ+\ZZ^2 \quad
\text{in} \quad \begin{pmatrix}-12&0\\0&4 \end{pmatrix} , \]
and there is an orthogonal push-out $W = N \perp_\ncap N$ with
$\ncap = \gen{-12}$. However, the pre-image in $\Pic Y$ of the square-norm
4 vector $x \in N$ that is orthogonal to $\ncap$ is exactly $-K_Y$ of the
quartic semi-Fano $Y$ that the block is obtained from.
Since $Y$ is not a genuine Fano, $-K_Y$ is not an ample class on~$Y$;
indeed, the anticanonical morphism contracts 12 curves. Any pre-image of
$x$ in $\Pic Z$ evaluates to 0 on the 12 exceptional curves, so cannot be
ample on $Z$. Therefore there can be no matchings compatible with this $W$.
\end{example}

For a given pair of blocks, there need not be any suitable $R$. For blocks
with $\rk N = 2$ there are not very many degrees of freedom in choosing $R$,
but we found some solutions in \textit{No \ref{g2:mm}}. There we used rank 2
Fano type blocks, and it was not too hard to compute the ample cones.
We expect that solutions become more plentiful as the rank increases,
but on the other hand the ample cones are more complicated to describe,
and it is less practical to search for solutions by hand.
The supply of toric semi-Fano blocks described below
should be suitable for an automated search.

Twisted connected sums that use perpendicular gluing (whether primitive or not)
always have $H^4(M)$ torsion-free. To get non-trivial torsion in $H^4(M)$
from orthogonal gluing, we need to find a non-trivial orthogonal push-out
$W = N_+ \perp_\ncap N_-$, but with conditions on $R$ that are more restrictive
than merely ensuring that $W$ is an integral lattice. For example, if
$R = \gen{-m}$ then in the notation used above the condition for $W$ to be
an integral lattice is that $d_+d_- \mid m$, but Lemma \ref{lem:torsion}
shows that $\Tor H^4(M) \cong (\ZZ/k\ZZ)^2$ where $k = \frac{m}{d_+d_-}$
(\eg a matching with the data $W = N \perp_R N$ from Example \ref{ex:nonex}
would give $\Tor H^4(M) \cong (\ZZ/3\ZZ)^2$). Again we expect that solutions
are easier to find when $\rk N_\pm$ are larger; in \textit{No \ref{g2:h4tor}},
our only explicit example with non-trivial
torsion in~$H^4(M)$, we used blocks with rank 3.

\subsection{Mass-production and Geography}
In this section we describe some general features
of the \gtwo-manifolds that can be mass-produced using our methods.

\subsubsection*{$\gtwo$-manifolds from pairs of smooth Fano $3$-folds}
We previously described in detail the 2-connected manifolds that can be
constructed as twisted connected sums of rank 1 Fano type building blocks; 
we now outline what can be achieved if we drop the rank 1 assumption.
Orthogonal, but non-perpendicular, matching of building blocks of rank two Fano
type was considered in \textit{No~\ref{g2:mm}} in the previous section,  but the easiest way to
mass-produce examples is to consider primitive perpendicular gluing again. 

Consulting the Mukai-Mori classification shows that out of 
$5564 = \tfrac{1}{2} \times (106 \times 105)$ possible pairs of Fano $3$-folds, $5401$ satisfy 
the (crude) rank condition \eqref{E:emb:crude} and therefore admit a primitive embedding of $N_{-}\perp N_{+}$
in $L$. We call the $163$ pairs that fail to satisfy the previous inequality the exceptional Fano pairs. 
With more work one could compute the Picard lattices of all the smooth Fano $3$-folds 
and determine their discriminant groups in order to check whether the refined rank/discriminant condition 
\eqref{E:emb:disc} is satisfied; this would yield further matching pairs.
(Recall from the Mukai-Mori classification that every Fano $3$-fold $F$ has
Picard rank $\rho(F) \le 10$ and if $\rho=\rho(F) \ge 6$ 
then $F$ is biholomorphic to a product $\CP^{1} \times S_{11-\rho}$; here $S_{d}$ denotes a del Pezzo surface 
of degree $d$ with $1 \le d \le 5$ and is obtained by blowing up $\CP^{2}$ in $9-d$ sufficiently general points.
Therefore all 163 of the exceptional Fano pairs include at least one product Fano $3$-fold $\CP^{1} \times S_{d}$.
For simplicity, we shall not consider these exceptional pairs any further in this paper.)

All $5401$ \gtwo-manifolds produced are $2$-connected and have torsion-free cohomology;  therefore
they are classified up to almost-diffeomorphism by $b^{3}(M)$ and $\gdiv{p_{1}(M)}$. 
$70$ values of $b^{3}(M)$ are realised by primitive  perpendicular gluing of nonexceptional pairs of Fanos (versus 46 values
from pairs of rank 1 Fanos). We have $b^{3}(M) = 23 + b$ for some even integer $b$ where
either $b=216$ or $32 \le b \le 174$; in the latter case all even values of $b$ within the range are realised except for 
$b=126, 168$ or $170$. In particular, the smallest value of $b^{3}(M)$ produced this way is $32+23=55$.

We have not studied systematically the values of $\gdiv{p_{1}(M)}$ arising from pairs of higher rank Fanos 
but this would be possible (though time-consuming) 
by adapting the methods used elsewhere in the paper and in
\cite[\S 5]{chnp1}.
However, Corollary \ref{c:inertia:g2} implies that there are at most $8$ possible 
diffeomorphism types with the same value of $b^{3}$. So while there are over 5\,000 Fano pairs we can match 
using perpendicular gluing, there are at the very most 560 diffeomorphism types of \gtwo-manifold that can be realised this way. 
In other words, there are many ways of finding different pairs of 
perpendicularly glued ACyl Calabi-Yau 3-folds of Fano type which yield 
\gtmetric s on the \emph{same} smooth $7$-manifold.
For example, there are 292 different matching  Fano-type pairs that 
give rise to a smooth $2$-connected $7$-manifold $M$ with $b^{3}(M)=97$.
Since the metrics ``see'' the long cylindrical neck, they cannot be isometric
unless the building blocks are diffeomorphic. If they belong to the same
component of the \gtwo moduli space then the path connecting them therefore
cannot merely be some small perturbation. 

Beyond the $2$-connected world we could also seek Fano type matching pairs using non-primitive 
perpendicular gluing.
This yields \gtwo-manifolds with the same Betti numbers and $p_{1}$ as the $2$-connected examples constructed 
via primitive perpendicular gluing, but which have nontrivial $\Tor H^{3}$.
As we have already discussed such non-primitive embeddings of $N_{+}\perp N_{-}$ are related 
to its overlattices and therefore to properties of the discriminant groups of $N_{\pm}$. 
Given the variety of ways to find matching Fano pairs yielding the same value of $b^{3}$ it seems likely 
that non-primitive perpendicular gluing will yield a considerably greater number of topological types, 
distinguished by $\Tor{H^{3}}$. We will not pursue this any further in the current paper.

\subsubsection{\gtwo-manifolds from toric semi-Fanos}
One very abundant source of semi-Fano $3$-folds are the \emph{toric} semi-Fano $3$-folds. 
We refer the reader to \cite[\S 8]{chnp1} for a more detailed review of their construction and properties.
The anticanonical model of a smooth toric weak Fano $3$-fold is a toric Fano $3$-fold with at worst 
Gorenstein canonical singularities. Such toric Fano $3$-folds correspond (uniquely up to isomorphism) 
to combinatorial objects called \emph{reflexive polytopes}. Kreuzer-Skarke gave an algorithm to classify 
reflexive polytopes and showed that there are $4319$ $3$-dimensional reflexive polytopes of which
$18$ correspond to smooth toric Fanos and $82$ to terminal toric Fanos (the latter have only ODP (ordinary double point) singularities). 
Every Gorenstein canonical toric Fano $3$-fold admits at least one and often many projective crepant resolutions; 
moreover, all these crepant resolutions are toric and so can be enumerated purely combinatorially.
 
Using this fact we have together with Tom Coates and Al Kasprzyk 
enumerated all (smooth) toric semi-Fano $3$-folds up to isomorphism.
For instance, we found that there are $1009$ non-isomorphic toric semi-Fano $3$-folds 
whose AC morphism is small; these correspond to projective small resolutions of the $82$ terminal reflexive polytopes.
Because any toric semi-Fano $3$-fold with small AC morphism is rigid, these all give rise to 
non deformation equivalent toric semi-Fano $3$-folds. 

While not every toric semi-Fano $3$-fold is rigid many of them are and 
rigidity/nonrigidity depends only on the Fano polytope and 
not the choice of projective crepant resolution; 
in total there are $526\,230$ non-isomorphic toric semi-Fano $3$-folds of which $435\,459$ are rigid 
(including the $18$ smooth toric Fanos and the $1\,009$ arising from the $82$ terminal 
reflexive polytopes already mentioned). 
Thus we have at least $435\,459$ deformation types of toric semi-Fanos.
(For the remaining non-rigid toric semi-Fanos more work would be needed to 
understand how many new deformation types these realise.)

Now consider all pairs of blocks of Fano-type or rigid toric semi-Fano type.
Already 39\,584 matching pairs are obtained by primitive perpendicular gluing 
of pairs that satisfy the crude inequality \eqref{E:emb:crude} 
(or 15\,027 pairs if we only include rigid toric semi-Fanos with small AC morphism).
On the other hand no new values of $b^{3}(M)$ are achieved this way. 
In other words, in the $2$-connected world
the main effects of allowing toric semi-Fanos (with small AC morphism) in addition to Fano $3$-folds are: 
(i) to increase significantly the number of different ways of using the twisted connected sum construction 
to produce \gtwo-metrics on the same smooth $7$-manifold 
and (ii) to produce many smooth $7$-manifolds on which we have \gtwo-metrics with different numbers 
of obvious rigid associatives.

The reason why the number of  Fano/rigid toric semi-Fano type pairs satisfying \eqref{E:emb:crude} is 
small compared to the numbers of deformation types 
of rigid toric semi-Fanos is that over 400\,000 of these deformation types have 
Picard rank 11 or greater and therefore can never lead to a pair satisfying \eqref{E:emb:crude}.
 Many more matching pairs would be obtained if we computed the discriminant groups of the toric Picard lattices and applied 
\eqref{E:emb:disc}---the toric semi-Fanos with large Picard rank tend to have many nonisomorphic 
flops and their discriminant groups are typically relatively small; 
a more systematic study of \gtwo-manifolds obtained via toric semi-Fano type blocks 
will be described elsewhere \cite{toric:g2}.

\begin{table}
\[
\begin{array}[b]{ccccccc} \toprule
 \text{Polytope} & \rho(N) & A_{N} & \ell(N) & \text{no. of resolutions} & g \\ \midrule
3282 & 14 & 2 \cdot 3^2 \cdot 4 & 2 & 46720 & 8\\
3267 & 14 & 2 \cdot 5 \cdot 8 & 2 & 44120 & 8\\
3415 & 15 & 2^2 \cdot 16 & 3 & 35775 & 7\\
3452 & 15 & 2 \cdot 3^3 & 3 & 34118 & 7\\
3297 & 14 & 3 \cdot 27 & 2 & 24216 & 8\\ 
2989 & 13 & 4 \cdot 19 & 1 & 23400 & 9\\
3033 & 13 & 3 \cdot 32 & 1 & 16092 & 9\\
3013 & 13 & 2 \cdot 5 \cdot 9 & 1 & 13770 & 9\\
3026 & 13 & 8 \cdot 11 & 1 & 12771 & 9\\
2986 & 13 & 3^2 \cdot 8 & 2 & 12528 & 9\\
3018 & 13 & 3 \cdot 4 \cdot 7 & 1 & 8770 & 9\\
2683 & 12 & 3 \cdot 29 & 1 & 8280 & 10\\
\bottomrule
\end{array}
\]
\bigskip
\caption{The top dozen rigid semi-small Gorenstein Fano $3$-polytopes ordered by the number of 
nonisomorphic projective crepant resolutions they admit. 
$\rho(N)$, $A_{N}$ and $\ell(N)$  
denote the rank of the polarising lattice $N$, the orders of cyclic factors in
the discriminant group of $N$, and its discriminant rank, respectively, and $g$
denotes the anticanonical genus of the polytope; 
the number in column one refers to the index of the polytope in the Sage implementation 
of the Kreuzer-Skarke list of reflexive $3$-polytopes.}
\label{tableg:semifano:massive}
\end{table}

For now we content ourselves with a demonstration of the plethora of matching pairs that can be obtained using 
information about the discriminant groups of the polarising lattices of rigid toric semi-Fanos.
For other reasons Rohsiepe computed the discriminant groups for the polarising lattices 
corresponding to any reflexive 3-polytope \cite{rohsiepe:paper}. 
Table \ref{tableg:semifano:massive} lists the dozen ``most prolific'' toric semismall reflexive polytopes, 
\ie the semismall reflexive polytopes with the most non-isomorphic crepant projective resolutions, 
along with the rank of the polarising lattice $N$, its discriminant group $A_{N}$, 
its discriminant rank $\ell(N)$,  the number of non-isomorphic projective crepant resolutions 
and the anticanonical genus of the polytope.

For all but two of the dozen polytopes (numbers $3\,415$ and $3\,452$) 
it follows from the data in Table \ref{tableg:semifano:massive} 
that $W = N \perp N_{-}$ satisfies \eqref{E:emb:disc} for \emph{any} rank two polarising lattice $N_{-}$. 
In fact, using the complete criterion for embeddings given by Nikulin 
\cite[Theorem 1.12.2]{nikulin:quadratic} one can show that if $N$ is the
polarising lattice corresponding to polytope $3\,415$ then, because its
discriminant form splits of an orthogonal $\cg{2}$ summand (the form is
$\diag(1/2, 1/2, 1/16)$), one can still primitively embed $N \perp N_{-}$ in
$L$ for any rank two polarising lattice $N_{-}$.
Thus each deformation type of rank 2 block in combination with the rigid semi-Fanos generated by these 
11 polytopes yields almost 250\,000 matching pairs. 
(In fact, it seems likely that a systematic study of the discriminant groups associated with 
rigid toric semi-Fanos will show that almost all of them can be primitively perpendicularly matched 
to any block of rank at most two). 

We know that there are precisely $17$ deformation types of smooth Fano $3$-folds of rank 1. 
We also know from \cite[\S 8]{chnp1} that there are over $150+36$ deformation types of rank 2 semi-Fano or Fano $3$-folds, 
but the precise number of deformation types has yet to be determined.
It follows that just these eleven prolific rigid semismall polytopes along with known blocks of rank at most two 
generate over 50\,000\,000 ($246\,442 \times (17+36+150) = 50\,027\,726$)
matching pairs via primitive perpendicular gluing.

\subsubsection{The geography of $2$-connected twisted connected sums}
Let $M$ be a $2$-connected twisted connected sum \gtwo-manifold. 
All such examples constructed in this paper so far have $55 \le b^{3}\le 239$. 
If $M$ is obtained from perpendicular gluing of blocks of Fano or semi-Fano type then there is an absolute 
lower bound for $b^{3}$ of $31=22+1+4+4$ (because any Fano or semi-Fano $3$-fold $Y$
has anticanonical genus $g$ at least $2$ and $b^{3}(Z)=b^{3}(Y)+2g(Y)$).
To achieve this lower bound we would need to find a semi-Fano $3$-fold $Y$ with $b^{3}(Y)=0$ and 
$g(Y)=2$, \ie $Y$ should be a resolution of a singular sextic double solid.
Recently Arap, Cutrone and Marshburn \cite{arap:c:m} claimed the existence of
such a smooth semi-Fano $3$-fold with small AC morphism and Picard rank
$\rho=2$; we have not verified this example in detail ourselves. 

Assuming the existence of such a smooth semi-Fano then (because $\rho=2$) we can immediately 
match such a block to itself by primitive perpendicular gluing and thus exhibit 
a $2$-connected twisted connected sum \gtwo-manifold with $b^{3}=31$, the smallest possible value of~$b^{3}$.
Moreover, because $Y$ has such small Picard rank we can also primitively perpendicularly match it 
to many other blocks of Fano or semi-Fano type. 
Hence the existence of this extremal $Y$ gives rise to a sequence of 7 new values of $b^{3}$ less than $55$ 
and also gives $b^{3}=149$ which was previously a ``gap'' in the sequence of values of $b^{3}$.

The ongoing classification programme for rank 2 weak Fanos (see \cite[\S 8]{chnp1} for an overview) 
looks likely to produce other 
rank two weak Fanos $Y$ whose corresponding block $Z$ has small~$b^{3}$, 
\eg there is potentially a rank 2 small resolution $Y$ of a terminal quartic with $b^{3}(Y)=0$ and hence $b^{3}(Z)=6$.
In this way it seems quite likely that essentially all odd numbers between $31$ and $189$ should 
be realised as $b^{3}$ of some $2$-connected twisted connected sum.
(The existence of the quartic semi-Fano described above would only leave gaps at 37 and 39).

It is somewhat curious that the building blocks we know with largest
$b^{3} = 108$ come from smooth sextic double solids, so that both the smallest
and largest values $b^3$ for $2$-connected twisted connected sums arise from
sextic double solids. Among Fano type blocks, there is a big gap down to the
next highest value 66 for $b^3$. There are a few other blocks that can be used
to construct 2-connected twisted connected sums with $197<b^{3}<239$,
\eg some of the non-symplectic smoothing blocks described in Remark
\ref{R:k3:involution:g2} or blocks obtained from a small resolution of a nodal
sextic double solid with relatively few nodes (\eg a block from a sextic double
solid with $15$ nodes \cite[Example 1.5]{cheltsov10} has $b^{3}=80$).

\subsubsection{Examples with positive $b^2$}
We have already seen how (non-perpendicular) orthogonal gluing can be used to construct \gtwo-manifolds 
with $b^{2}>0$. However, for the two reasons we observed at the beginning of this section it can be 
somewhat labour-intensive to implement. Perhaps a more economical way to mass-produce 
\gtmfd s with $b^{2}>0$ is to use perpendicular gluing and to choose at least one block with $K \neq 0$, 
\eg Example \ref{exg:toricV22_b} has $\rk K = 12$ and arises by blowing up a nongeneric AC pencil 
on a toric semi-Fano $3$-fold.

One uniform source of building blocks with $K \not= 0$ are the 74 
non-symplectic type blocks described in Remark \ref{rmk:kl-def}. For example,
there is a K3 surface $S$ with non-symplectic involution whose action on
$H^2(S)$ fixes $N_+ = 2E_8(-1) \perp U$, and from which one may construct a
building block $Z_+$ with $\rk K_+ = 20$ and $b^3(Z_+) = 8$. If $Z_-$ is
any Fano or semi-Fano type block whose polarising lattice $N_-$ has rank $\leq 2$,
then $N_-$ can be embedded primitively in $2U$ by Theorem \ref{thm:exist_emb};
thus $N_+ \perp N_-$ can be embedded in $L$, and we can use perpendicular
gluing to construct a \gtmfd{} $M$ with $b^2(M) = 20$,
$b^3(M) = b^3(Z_-) + 51$. As there are many Fano and semi-Fano 3-folds with Picard rank
$\leq 2$, we can ask how many different values of $b^3(M)$ are realised this way. 
Considering only the Fano-type blocks we obtain $18$ values of $b^{3}(Z_{-}) \in \{24, \dots, 66\} \cup \{108\}$ 
and where $\{46,54,60,62\}$ are omitted.
Since the classification of rank two weak Fano $3$-folds is still in progress we cannot currently say 
precisely what values can be obtained for semi-Fano type blocks.  
However, matching allowing rank two semi-Fano type blocks should enable us to obtain further examples 
with small values of $b^{3}(Z_{-})$, potentially as low as $4$.
Of course using a semi-Fano block also allows us to 
construct numerous \gtwo-manifolds with $b^{2}>0$ that contain rigid associative $3$-folds.
There are also non-symplectic type blocks $Z_+$ with $\rk K_+$ taking any even
value between $0$ and $20$, such that $N_+$ embeds in $2E_8(-1) \oplus U$
and which therefore can be matched using perpendicular gluing to any Fano or 
semi-Fano type block of rank up to 2 as above. 

One could also use primitive perpendicular gluing to match Example \ref{exg:P3_deg} (which has $\rk K=3$ and $N = \gen{4}$) 
with any Fano block or any semi-Fano block of rank up to $10$.  
There will therefore be many \gtwo-manifolds with $b^{2}=3$.
Finally with further work one could calculate $\gdiv{p_{1}}$ for all the examples above
and thereby distinguish further topologically distinct \gtwo-manifolds.

\subsection{Prospects of further smooth classification of simply-connected
spin 7-manifolds}

For the large number of examples of 2-connected twisted connected sums with
$H^4(M)$ torsion-free, we can use the classification result
\ref{thm:torfreeclass} to determine the almost-diffeomorphism type.
Except for the relatively few examples where $\gdiv p_1(M) = 16$ or $48$
we pin down the diffeomorphism class completely; as explained in Remark
\ref{rmk:EK}, in those two cases we would need to compute a generalised
Eells-Kuiper invariant to eliminate the remaining ambiguity in the smooth
structure.

We have however also constructed (and explained how to construct many more)
examples with relatively simple cohomology, but with $\pi_2(M)$ non-trivial.
The classification results for simply-connected spin but not 2-connected
7-manifolds available in the literature mostly require (at least) that
$H^4(M)$ is finite.
With cues from Diarmuid Crowley, we speculate about what analogues one can hope
to prove when $H^4(M)$ is infinite but torsion-free.

\subsubsection{$\pi_2(M)$ finite cyclic} Using non-primitive but perpendicular
gluing, we can find many examples with $H^2(M) = 0$ and $\Tor H^3(M)$ a
cyclic group $\cg{k}$. Then $\pi_2(M) \cong \cg{k}$. As before, the
isomorphism class of the pair $(H^4(M), p_1(M))$ is an obvious homeomorphism
invariant, and when $H^4(M)$ is torsion-free it is equivalent to
$(b^4(M), \gdiv p_1(M))$. Now we have an additional invariant given by the
square $z^2 \in H^4(M; \cg{k})$ modulo $((\cg{k})^*)^2$ of a generator
$z \in H^2(M; \cg{k}) \cong \cg{k}$.
We expect that (for a fixed $k$) the class of the triple $(H^4(M), p_1(M), z^2)$ determines
the almost-diffeomorphism type of $M$. If $k$ is prime, and $x \in H^4(M)$ is
a primitive element of which $p_1(M)$ is a multiple, this means specifying
$b^4(M)$, $\gdiv p_1(M)$, whether $z^2 = 0$, if not whether it is a multiple of
the mod $k$ reduction of $x$, and if so whether the coefficient is a quadratic
residue mod $k$.

\subsubsection{$\pi_2(M)$ infinite cyclic} In \textit{No \ref{g2:mm}} we gave
examples with $H^*(M)$ torsion-free, and $H^2(M) \cong \ZZ$. Then
$\pi_2(M) \cong \ZZ$. If $z \in H^2(M)$ is a generator, then the isomorphism
class of the triple $(H^4(M), p_1(M), z^2)$ is an obvious invariant.
In the setting where $H^4(M)$ is finite instead of torsion-free, and generated
by $z^2$ and $p_{1/2}(M)$, Kreck and Stolz \cite{kreck88} proved a
classification result in terms of a triple of invariants
$s_1, s_2, s_3 \in \QQ/\ZZ$
(when $H^4(M) = 0$, $s_1$ corresponds to the Eells-Kuiper invariant).
By analogy we expect that $(H^4(M), p_1(M), z^2)$ may not suffice to
determine even the homotopy type of $M$ on its own, but that it may be possible
that together with some generalisations of $s_2$ and $s_3$ it determines the
almost-diffeomorphism type (and a generalised Eells-Kuiper invariant would
pin down the precise diffeomorphism type).

\subsubsection{Formality and torsion-free $\pi_2(M)$} 
Hepworth \cite{hepworth05} generalised
the work of Kreck and Stolz to the case when $\pi_2(M) \cong \ZZ^k$ (but
still under assumptions requiring $H^4(M)$ to be finite). Some of the
generalised Kreck-Stolz invariants can in
this case be interpreted in terms of the Massey product structure on
the cohomology of $M$. We expect that the classification problem
when $\pi_2(M) \cong \ZZ^k$ with $H^4(M)$ torsion-free should also be
greatly simplified if we restrict to the case when all Massey products vanish;
this happens in particular if $M$ is \emph{formal}. Deligne \emph{et al}
\cite{deligne75} showed that any Kähler manifold is formal, and it is an
interesting problem whether the same is true for \gtmfd s. For the special
case of twisted connected sums, one may ask whether formality can be deduced
from the way they are built from pairs of Kähler manifolds.
Cavalcanti \cite{cavalcanti06} shows that any simply-connected 7-manifold $M$
with $b^2(M) \leq 1$ is formal (so we did not need to consider Massey products
when $\pi_2(M)$ is cyclic), and that $b^2(M) \leq 2$ suffices for formality
if $M$ is a \gtmfd. Formality of \gtmfd s is also studied by
Verbitsky \cite{verbitsky11}.

\subsection{\gtwo-transitions}
In this section we make some more speculative remarks on how compact \gtwo-manifolds constructed 
by gluing different ACyl Calabi-Yau $3$-folds may be seen as related. 
The basic idea is that well-known transitions for  Calabi-Yau $3$-folds, \ie flops and conifold transitions,  
can yield via the twisted connected sum construction analogous \gtwo-transitions.
We begin by recalling the basic features of these well-known $3$-fold transitions.

\subsubsection{$3$-fold transitions: flops and conifold transitions}
Recall that the ordinary double point (ODP)\[\{z_1^2 + z_2^2 + z_3^2 + z_4^2=0\} \subset \bbc^4\]
admits two (isomorphic) projective small resolutions (each of which replaces the
singularity at the origin with a $\CP^1$ with normal bundle $\oo(-1) \oplus \oo(-1)$) 
and a smoothing $\{z_1^2 + z_2^2 + z_3^2 + z_4^2 = t \}$
which replaces the singularity with a Lagrangian $3$-sphere.
Let $X$ be a nodal Fano $3$-fold, \ie $X$ has only finitely many singular points each (locally analytically)
modelled on the ODP. By Namikawa's deformation results \cite{namikawa:fano:smooth} $X$ is globally smoothable, 
\ie $X$ is smoothable to a family of smooth Fano $3$-folds $F_{t}$. 
Suppose that $X$ also admits a projective small resolution~$Y$. 
Then $Y$ is a smooth semi-Fano $3$-fold and the transition from the smooth semi-Fano $Y$ to a smooth Fano $F$ 
is called a \emph{conifold transition}.
(Conifold transitions have traditionally been studied in the context of $3$-folds with $c_{1}=0$, \ie the Calabi-Yau case; 
unlike the Fano condition the condition $c_{1}=0$ is preserved under conifold transitions.)
We can often also \emph{flop} a given projective small resolution $Y$ of a nodal $3$-fold $X$; 
this has the effect of changing the choice of which of the two projective small resolutions of the ODP is used 
at some of the nodes. In the semi-Fano world flopping $Y$ yields other 
smooth semi-Fano $3$-folds which also have $X$ as their AC models.

Flopping will lead to $6$-manifolds with the same integral cohomology groups 
but typically with different cohomology rings.
The topological effect of a conifold transition is to replace a finite number of two-spheres 
(with normal bundle of a given type) with the same number of three-spheres (also with normal bundle of a given type).
Care must be exercised in understanding how this 
sort of topological surgery changes the topology---in particular 
the cycles being created or destroyed need not be homologically independent.

\begin{remark}
\label{r:metric:conifold}
There is also a conjectural picture of conifold transitions as metric transitions and not 
just topological or complex-geometric transitions.
Recall that the $3$-fold ordinary double point admits a Ricci-flat Kähler (KRF) cone metric; 
the term \emph{conifold} often refers to the ODP endowed with this KRF metric.
The smoothings and the small resolutions of the ODP admit Ricci-flat K\"ahler metrics 
asymptotic to the conifold metric.

Suppose $X_{0}$ is a nodal projective $3$-fold with trivial canonical bundle.
It has long been conjectured that the nodal variety $X_{0}$ should admit 
a Ricci-flat K\"ahler metric, smooth away from the nodes and 
asymptotic to the conifold metric at each node. This is however still unproven.
If $X_{0}$ smoothes to a family of smooth $3$-folds $X_{t}$ with $c_{1}=0$
then by Yau's result $X_{t}$ admits Ricci-flat K\"ahler metrics in each K\"ahler class.
\emph{Given} the existence of a conically singular KRF metric on $X_{0}$ one
can instead use gluing methods to construct smooth KRF metrics on $X_{t}$,
\ie one glues in an appropriately scaled copy of the smoothing of the ODP
endowed with its asymptotically conical KRF metric \cite{chan06, chan09}.
This yields a $1$-parameter family of KRF metrics that converges as $t \ra 0$
to the conically singular metric on $X_0$.
Similarly, if $X_0$ admits a projective small resolution $Y$ then one could use
gluing methods to construct \mbox{1-parameter} families of smooth KRF metrics
on $Y$ that degenerate back to the conically singular KRF metric on $X_0$.

We emphasise that this metric picture remains conjectural since the existence
of the conically singular KRF metric on $X_{0}$ remains open.
\end{remark}

\subsubsection{Related \gtmfd s}
Suppose that we have a conifold transition $Y \to X \to F$ between a smooth semi-Fano $Y$ and a smooth Fano 
$3$-fold $F$ via the nodal Fano $3$-fold $X$.
Using Proposition \ref{prop:block_from_sf} we can generate building blocks $Z_{Y}$ and $Z_{F}$ 
and hence via Theorem \ref{thm:acyl_limits} also (families of) ACyl Calabi-Yau manifolds $V_{Y}$ and $V_{F}$.
Suppose that $\fbb_{-}$ is another family of building blocks chosen so that 
there are ACyl Calabi-Yau structures in its deformation family compatible with some $V_{Y}$ and $V_{F}$.
Then we can construct the resulting twisted connected sum \gtwo-manifolds $M_{Y}$ and $M_{F}$ 
and regard them as related \gtwo-manifolds.
We could of course replace the conifold transition above with a flop $Y \to Y'$ 
and proceed as in the previous case to obtain related \gtwo-manifolds $M_{Y}$
and $M_{Y'}$.
We use the term \emph{\gtwo-transition} to describe either of these operations.
\begin{remark}
\label{r:g2:dream}
By analogy with Remark \ref{r:metric:conifold} one might hope for a stronger metric counterpart 
of this relation between $M_{Y}$ and $M_{F}$.
The ultimate aim would be to find
families of \gtmetric s on $M_{Y}$ and $M_{F}$ that converge to the same singular
\gtmfd{}---with transverse conifold singularities along $\Sph^1$s---but 
that is currently out of reach. We will discuss the difficulties later. 
For the time being we use the relation between $M_{Y}$ and $M_{F}$ (or $M_{Y'}$)  of being
descended from related pairs of ACyl Calabi-Yaus as an organising principle.
\end{remark}

At the level of the $3$-folds passing from the original Fano $F$ to the semi-Fano $3$-fold $Y$ has three principal effects:
\begin{enumerate}
\item
$b^{2}$ increases when passing from $F$ to $Y$ (recall \ref{e:b2:res} and the fact that the existence of a projective small 
resolution of $X$ forces its defect $\sigma(X)$ to be positive). Hence the K3 surfaces $S_{Y} \in \abs{-K_{Y}}$ appearing 
in any semi-Fano $Y \in \sff$ are more special than the K3 surfaces $S_{F} \in \abs{-K_{F}}$ for any Fano $F \in \mathcal{F}$.
\item
$b^{3}$ typically decreases when passing from $F$ to $Y$, often by more than
the increase in $b^2$, but it may also stay constant (recall \ref{e:b3:res});
\item
$Y$ unlike $F$ contains compact rigid rational curves $C_{1}, \ldots ,C_{e}$
which do not intersect smooth anticanonical divisors $S_{Y} \in \abs{-K_{Y}}$.
Each such curve $C_{i}$ gives rise to a compact rigid rational curve in the
associated ACyl Calabi-Yau structures on $Z_Y \setminus S_Y$.
\end{enumerate}
One can think of the defect $\sigma$ of the nodal degeneration $X$ as giving a way to stratify the possible nodal degenerations 
of the original family of smooth Fanos $\mathcal{F}$, \ie we can order our  hierarchy according to the defect 
of the degeneration $X$: by \eqref{e:b2:res} this is the same as ordering by the rank of $\Pic{Y}$.

Each of the three effects above has significance for obtaining \gtwo-manifolds 
by matching with the family of blocks $\fbb_-$.
\begin{enumerate}
\item[(i')]
The possible asymptotic K3 surfaces $S_{Y}$ of ACyl Calabi-Yau structures obtained from a semi-Fano $Y$
are more special than those $S_{F}$ obtained via the original Fano $F$. 
We interpret this as follows: 
it should be harder to match in the deformation family of ACyl Calabi-Yau structures obtained from the semi-Fano $3$-fold $Y$ than 
for those obtained from the original Fano $F$.
\item[(ii')]
\emph{Assuming} that we can use perpendicular gluing to achieve matching of 
ACyl Calabi-Yau structures obtained from both the semi-Fano $Y$ 
or from the original Fano $F$,  then we will usually obtain topologically distinct $2$-connected $\gtwo$-manifolds 
$M_{Y}$ and~$M_{F}$.
At the level of complex $3$-folds passing from $F$ to $Y$ decreases 
$b^{3}$ at the expense of increasing $b^{2}$;  however, at the level of \gtwo-manifolds 
this transition decreases $b^{3}(M)$ while maintaining $b^{2}(M)=0$.
In this sense one can think of the transition from $F$ to $Y$ as yielding
\mbox{$2$-connected} \gtwo-manifolds which are topologically smaller. 
\item[(iii')]
The $e$ rigid rational curves $C_{i}$ give rise to $e$ new compact associative $3$-folds in $M_{Y}$ compared to $M_{F}$. 
\end{enumerate}

Moreover, by passing to (deformation types of) semi-Fanos associated with
different nodal degenerations of $F$ one can obtain 
$\gtwo$-manifolds with successively smaller and smaller topology: 
see below for concrete examples obtained by degenerating quartics.
In this sense \gtwo conifold transitions create a ``hierarchy'' of related \gtwo manifolds.
Similarly different nodal degenerations $X$ of $F$ allow us to vary the number $e$. 
In some cases by choosing different nodal degenerations we can vary $e$ without changing $b^{3}$ 
of the resulting $2$-connected \gtwo-manifold. This gives one way to exhibit \gtwo-metrics on the same underlying 
smooth $7$-manifold with different numbers of obvious compact rigid associative $3$-folds.

\begin{remark*}
More generally, if after making a transition from a Fano $F$ to a semi-Fano $Y$ 
we can no longer match by perpendicular gluing but can instead match using the more general orthogonal gluing 
then $b^{2}(M)$ can increase under the transition from $F$ to $Y$. 
However, by Lemma \ref{lem:b2b3} the sum $b^{2}(M)+b^{3}(M)$ 
cannot increase when passing from $F$ to $Y$ and usually must decrease.
\end{remark*}

\subsubsection{Matching quartic type blocks}
We now give a concrete illustration of the general discussion above using ACyl Calabi-Yau $3$-folds associated with
various quartic $3$-folds. Eight families of ACyl Calabi-Yau $3$-folds associated with various quartics 
appear already in this paper: 
one family of Fano type obtained from any smooth quartic (Example \ref{exg:species1}$^{1}_{4}$), 
six families of semi-Fano type obtained from projective small resolutions of defect 1 nodal quartics 
(Examples \ref{exg:quartic_w_plane}--\ref{exg:quartic_w_22}---recall that in
two of these four examples different choices
of small resolution lead to non deformation equivalent semi-Fanos)
and the family of semi-Fano type obtained from a projective small resolution of the Burkhardt quartic 
(Example \ref{exg:burkhardt_quartic}). The polarising lattices $N$ in these cases are respectively: $\gen 4$, the rank 2 lattices 
 and the rank 16 lattice listed in table \ref{tableg:blocks}. 
In all cases we have $b^{3}(Z) = b^{3}(Y) + 6$ where
for the smooth quartic we have $b^{3}(Y)=60$; 
for the four examples with $\rk{N}=2$ applying \eqref{e:b3:res} we see that $b^{3}(Y)$ decreases linearly
with the number of nodes $e$ in the AC model $X$ and 
for the Burkhardt quartic we have $b^{3}(Y)=0$. 
(So one might hope to use the Burkhardt block to produce \gtwo-manifolds with small $b^{3}$. 
However, its large Picard rank makes it difficult to match via perpendicular gluing as we explained earlier.)

If for the moment we remove the Burkhardt example from consideration there are 15 pairs $N_{\pm}$ of lattices
we can choose. 
In these 15 cases $N_{+}\perp N_{-}$ is a lattice of signature $(2,0)$, $(2,1)$ or $(2,2)$, 
which therefore may be primitively embedded in the K3 lattice $L$.
Hence we can match all 15 pairs by primitive perpendicular gluing to obtain  a series of $2$-connected \gtwo-manifolds with 
torsion-free cohomology and $\gdiv{p_{1}}=4$ or $8$. The 14 values of $b^{3}$ realised this way are 
\begin{equation}
\label{e:b3:nod:quart}
b^{3} \in \{91, 93, 95, 101, 103, 107, 109, 111, 117, 123, 125, 133, 139, 155\}.
\end{equation}
The number of rigid associative $3$-folds diffeomorphic to
$\Sph^{1}\times \Sph^{2}$ we can realise in these \mbox{$\gtwo$-manifolds} is 
\[
a_0 \in \{0, 9, 12, 16, 17, 18, 21, 24, 25, 26, 28, 29, 32, 33, 34\}.
\]
These examples illustrate (ii'): passing from an initial family of Fano-type blocks 
(which corresponds to the \gtwo-manifold with $b^{3}=155$) 
to related semi-Fano type blocks 
via nodal degeneration and projective small resolution (conifold transitions) 
leads---via the twisted connected sum construction---to a family of  related but ``smaller'' \gtwo-manifolds.

These examples also illustrate how we can exhibit \gtwo-metrics on the same $7$-manifold with different numbers 
of obvious rigid associative $3$-folds:
matching Example \ref{exg:quartic_w_plane} with itself or \ref{exg:species1}$^{1}_{4}$ with \ref{exg:quartic_w_scroll} 
both yield manifolds with $b^{3}=123=50+50+23=34+66+23$; 
in both cases depending on the choice of small resolutions made we can achieve either $\gdiv{p_{1}}(M)=4$ or $8$, 
and in each case the almost-diffeomorphism type contains a unique diffeomorphism type.
For the first matching pair we have $a_{0}=18=9+9$ whereas for the second we have $a_{0}=17=17+0$. 

\begin{remark}
It is natural to wonder whether the fact that there are different numbers of obvious rigid associative $3$-folds 
for \gtwo-metrics  on the same $7$-manifold $M$ can be used to show these metrics are not in 
the same connected component of the moduli space of $\gtwo$-metrics on~$M$.
\end{remark}

These examples also illustrate another somewhat subtle 
point related to flops and their effect on the topology of the associated  \gtwo-manifolds.
In Examples \ref{exg:quartic_w_plane} and \ref{exg:quartic_w_scroll} flopping 
leads to two non-diffeomorphic blocks of semi-Fano type (arising from different projective small 
resolutions of the same nodal quartic $X$). These blocks are distinguished by $\gdiv{c_{2}(Z)}$ which is $2$ or $4$ 
depending on the choice of small resolution made. If we match these blocks to
Example \ref{exg:quartic_w_quadric} then, because that block has
$\gdiv{c_{2}(Z)}=2$, Corollary \ref{cor:p1gcd} implies that we obtain
diffeomorphic \gtwo-manifolds irrespective of our choice of small resolution. 
On the other hand Example \ref{exg:species1}$^{1}_{4}$ has $\gdiv{c_{2}}(Z)=4$,
so if we match with that then the diffeomorphism type of the resulting
\gtwo-manifolds does depend on our choice of small resolution. 

Now suppose we choose $N_{+}$ to be the rank 16 polarising lattice of the Burkhardt example.  Then because of 
the high rank of $N_{+}$, primitive perpendicular matching is now much more difficult to achieve. 
Nevertheless,  \textit{No~\ref{g2:b1}a} showed that it is possible to match blocks of Burkhardt type 
with blocks obtained from smooth quartics using primitive perpendicular gluing. 
This yields a $2$-connected $7$-manifold $M_{95}^{4}$ with torsion-free cohomology, $\gdiv{p_{1}}=4$, 
$b^{3}=95$ (by the classification theory there is a unique such smooth $7$-manifold) and containing $45$ rigid associative $\Sph^{1}\times \Sph^{2}$s.

\begin{remark}
It is not possible to achieve primitive perpendicular gluing using the
Burkhardt block and any of the other quartic-related blocks. In fact, if
$N_{-}$ is any polarising lattice of rank greater than $1$ then it is not
possible to embed $N_{+}\perp N_{-}$ primitively in $L$ because
that would violate the necessary condition \eqref{eq:nec_emb} (recall that
$\ell(N_{+})=5$ for the polarising lattice of the Burkhardt block).
This illustrates what we mean in (i').
\end{remark}

\subsubsection{General terminal degenerations and smaller \gtwo-manifolds}
In the discussion above for simplicity (and because all the examples we presented were of this type) 
we referred only to degenerations of smooth Fanos to singular Fanos with only 
ordinary nodes. However, we could consider more general degenerations, especially to Fano $3$-folds 
with terminal Gorenstein singularities and seek projective small resolutions of these also. 
This will lead to a wider variety of semi-Fanos  related to a single deformation family of smooth Fano $3$-folds. 
For example, one could look for  projective small resolutions of defect one terminal quartics with worse than ODPs, thereby 
generalising Examples \ref{exg:quartic_w_plane}--\ref{exg:quartic_w_22}.

As part of the partial classification scheme (summarised in \cite{chnp1}) 
for smooth weak Fano \mbox{$3$-folds} of rank two,
Cutrone-Marshburn \cite[Nos. 54--76,\,Table 2]{cutrone:marshburn} present a 
list of 19 potential candidates for rank two weak Fano $3$-folds $Y$ with small AC morphism 
which (if they exist) can be obtained as projective small resolutions of terminal quartics. 
They all arise as the blowup of a smooth 
rank 1 Fano $3$-fold along a smooth curve of known degree and genus (their numerical link types are all E1-E1); 
this makes it straightforward 
to compute $b^{3}(Y)$ (and hence also $b^{3}(Z)$ for the associated block) 
and the polarising lattices for these putative examples. 
The possible values for $b^{3}(Z)$ which arise from their list are 
\[
b^{3}(Z) \in \{6, 8, 10, 12, 14, 16, 18, 20, 26, 28, 34\}, 
\]
whereas for our previously discussed quartic-related blocks we had
\[
b^{3}(Z) \in \{6,34, 36, 44, 50, 66\},
\]
where $6$ and $66$ in the latter list are realised by the Burkhardt  block and the smooth quartic block respectively.

Assuming all these examples could be realised as rank two semi-Fano $3$-folds (which admittedly may not be the case) 
then because they have rank two polarising lattices we can certainly achieve primitive perpendicular 
gluing of any such pair of semi-Fano blocks. This would yield $2$-connected $7$-manifolds with torsion-free cohomology with the 
following $46$ values of $b^{3}$
\[
b^{3} \in \{35, 37, \ldots ,111, 115, 117, 123, 125, 133, 139, 155\}, 
\]
compared to the $14$ values of $b^{3}$ (all of which satisfied $b^{3} \ge 91$) 
obtained in \eqref{e:b3:nod:quart} from the smooth and nodal quartic blocks 
we already discussed.

The main thing we learn from this discussion is that by allowing ourselves to degenerate to worse than ordinary 
nodes we may be able to 
obtain semi-Fano $3$-folds with small $b^{3}$ without having to dramatically increase $b^{2}$---as happens 
for example in the Burkhardt example where we achieve $b^{3}(Y)=0$ but the price is that $b^{2}(Y)=16$. 
As we have seen the large Picard rank of the Burkhardt example makes it very problematic to match; 
by contrast we could perpendicularly glue these resolutions of terminal defect one quartics to most 
blocks of semi-Fano type. 
The price we pay for allowing semi-Fano $3$-folds constructed by resolving 
more general terminal degenerations is that we can no longer 
be guaranteed to be able to produce rigid associative $3$-folds in the resulting \gtwo-manifolds.

\begin{remark*}
It would be particularly interesting to know the existence (or not) of Nos. 57 and 71 in 
\cite[Table 2]{cutrone:marshburn} since both would give rise to rank two semi-Fanos with $b^{3}(Y)=0$; 
for this one needs to study rational curves of degree $8$ in  the rank one Fano $V_{22}$ or in the smooth quadric $Q$ 
respectively. 
\end{remark*}

\subsubsection{Metric \gtwo-transitions}
We indicate the technical problems that currently stop us from using the
twisted connected sum construction to produce families of \gtmetric s
degenerating to a compact singular \gtmfd.
One basic problem is that given a sequence of ACyl Calabi-Yau metrics
on $V_+$ that degenerate in a satisfactory way, we would need to be able to
match the asymptotic limits of this whole family to ACyl Calabi-Yau metrics on
$V_-$, in a continuous manner. One situation where this problem is simplified
is when $Z_+$ is a building block like Example \ref{exg:P3_deg}.
There is a sequence of Kähler classes on $Z_+$ that shrink the
$(-1, -1)$-curves in $Z_+$, but whose restriction to $S_+$ is constant.
We should therefore find a sequence of degenerating ACyl Calabi-Yau structures
on $V_+$ with a fixed asymptotic limit, that can be matched to a fixed
ACyl Calabi-Yau 3-fold $V_-$.

This easy case of the matching-in-families problem is different from the
sort of transitions we were discussing earlier in the subsection, in
that the degeneration of the ACyl Calabi-Yau structures on $V_+$ comes from
changing the resolution of a non-generic pencil rather than from a degeneration
of the underlying semi-Fano to a nodal Fano $X$.
In the latter case, suppose that the sequence of ACyl Calabi-Yau
structures on $V_+$ can be matched with ones on $V_-$.
Then, because $H^2(V_+) \to H^2(S_+)$ is injective for semi-Fano type blocks,
the sequence of matching data must degenerate too in some sense.
The problem of finding the limiting matching data therefore turns out to be too
constrained to use the argument of Proposition \ref{prop:orth_gluing}
(the number of missing degrees of freedom is exactly the defect of $X$)
This is the same kind of problem as in handcrafted gluing, and requires the
same remedy: more detailed information about the deformation theory.

Given a matching in families, another problem is to control the neck length
parameter in the gluing.
If we can find a 1-parameter family of compatible pairs of ACyl metrics
on $V_+$ and $V_-$, then Theorem \ref{thm:g2glue} says that for each pair
there is a parameter $T$ such that we can form a twisted connected sum with
neck length $T$. Joyce's perturbation results give bounds on $T$ in terms of
the geometry of $V_+$ and $V_-$. If the family of metrics on $V_-$ degenerates
to a metric with conical singularities, then the upper bound for $T$ goes
to infinity, so there is no guarantee that we can find a 1-parameter
family of \gtmetric s whose Gromov-Hausdorff limit is compact.
It may be difficult to get around this without resolving the conjecture
about existence of conically singular Calabi-Yau metrics.